\documentclass[a4paper,notitlepage,12pt]{article}
\usepackage[sectionbib]{natbib}
\usepackage{amsmath,mathtools}
\usepackage{amssymb,amsthm} 
\usepackage{geometry}
\usepackage{setspace}
\usepackage{enumitem}
\usepackage{dsfont}
\usepackage{algorithm,algpseudocode,float}
\usepackage{lipsum}
\usepackage{graphicx}
\usepackage{multirow}
\usepackage{booktabs}
\usepackage{caption}
\usepackage{makecell}
\usepackage{mathabx}
\usepackage{xr}
\usepackage{xcolor}
\usepackage{longtable}
\usepackage[normalem]{ulem}

\allowdisplaybreaks[1]
\usepackage{array}
\usepackage{xr}

\makeatletter
\newcommand*{\addFileDependency}[1]{%
  \typeout{(#1)}
  \@addtofilelist{#1}
  \IfFileExists{#1}{}{\typeout{No file #1.}}
}
\makeatother

\makeatletter
\newenvironment{breakablealgorithm}
{
		\begin{center}
			\refstepcounter{algorithm}
			\hrule height.8pt depth0pt \kern2pt
			\renewcommand{\caption}[2][\relax]{
				{\raggedright\textbf{\ALG@name~\thealgorithm} ##2\par}%
				\ifx\relax##1\relax 
				\addcontentsline{loa}{algorithm}{\protect\numberline{\thealgorithm}##2}%
				\else 
				\addcontentsline{loa}{algorithm}{\protect\numberline{\thealgorithm}##1}%
				\fi
				\kern2pt\hrule\kern2pt
			}
		}{
		\kern2pt\hrule\relax
	\end{center}
}
\makeatother

\usepackage[
  colorlinks,
  linkcolor=blue,
  citecolor=blue,
  filecolor=magenta,
  urlcolor=cyan,
  pageanchor=true
]{hyperref}

\usepackage[titletoc,title]{appendix}

\newcommand{\startsupplement}[1][Supplementary Material]{%
  \clearpage
  \appendix
  \addappheadtotoc
  \begin{center}
    {\LARGE #1\par}
  \end{center}
}

\newtheorem{assumption}{Assumption}
\newtheorem{definition}{Definition}

\newtheorem{theorem}{Theorem}
\newtheorem{lemma}{Lemma}
\newtheorem{corollary}{Corollary}
\newtheorem{proposition}{Proposition}

\newcommand{\widebreve}[1]{\breve{#1}}
\newcommand{\bm}{\mathbf}
\newcommand{\bbm}{\boldsymbol}

\title{A robust and scalable estimation for high-dimensional volatility models}
\author{Kejun Chen$^{a}$, Yuchang Lin$^{b}$, Qianqian Zhu$^{a}$\\
\textit{$^{a}$Shanghai University of Finance and Economics, $^{b}$The University of Hong Kong}}

\date{}

\begin{document}
	\newcolumntype{L}[1]{>{\raggedright\arraybackslash}p{#1}}
	\newcolumntype{C}[1]{>{\centering\arraybackslash}p{#1}}
	\newcolumntype{R}[1]{>{\raggedleft\arraybackslash}p{#1}}
\maketitle
\begin{abstract}
This paper introduces a robust and computationally efficient estimation framework for high-dimensional volatility models in the BEKK-ARCH class. The proposed approach employs data truncation to ensure robustness against heavy-tailed distributions and utilizes a regularized least squares method for efficient optimization in high-dimensional settings. Non-asymptotic error bounds are established for the resulting estimators under heavy-tailed regimes, and the minimax optimal convergence rate is derived. Moreover, a robust BIC and a Ridge-type estimator are introduced for selecting the model order and the number of BEKK components, respectively, with their selection consistency established under heavy-tailed settings. Simulation studies demonstrate finite-sample performance of the proposed method, and two empirical applications illustrate its practical utility. The results show that the new framework outperforms existing alternatives in both computational speed and forecasting accuracy.
\end{abstract}

\newpage

	\section{Introduction}\label{sec:introduction}

	Volatility modeling plays a fundamental role in risk management, asset pricing and portfolio selection \citep{bauwens2006multivariate}. In today's financial markets, institutions routinely deal with large portfolios containing numerous correlated assets, resulting in high-dimensional data environments. Moreover, financial returns are characterized by heavy-tailed distributions and frequent outliers, phenomena often driven by investor herding behavior and sudden market shocks. These inherent features, especially the high-dimensionality and tail risk, complicate the reliable estimation of conditional covariances. Consequently, there is a growing need for estimation approaches that are not only capable of capturing complex dependence structures in high dimensions but also remain robust to extreme observations. This underscores the necessity of developing robust modeling frameworks tailored for high-dimensional financial assets in the era of big data.

	Over recent decades, considerable progress has been made in the development of methodologies for financial volatility modeling. The autoregressive conditional heteroskedastic (ARCH) and generalized autoregressive conditional heteroskedastic (GARCH) models provide a powerful framework for capturing the time-varying conditional variances in asset returns \citep{engle1982autoregressive,bollerslev1986generalized,francq2019garch}, which are crucial for risk management and asset pricing. 
	Nevertheless, such univariate models cannot capture the dependence structure across many assets. This gap has spurred multivariate extensions to model the conditional covariance or correlation matrices, leading to a variety of influential frameworks including the vector GARCH model \citep{bollerslev1988capital}, the constant conditional correlation (CCC) model \citep{bollerslev1990modelling}, the dynamic conditional correlation (DCC) model \citep{engle2002dynamic,aielli2013dynamic} and the Baba, Engle, Kraft and Kroner's (BEKK) GARCH model \citep{engle1995multivariate}.
	For these volatility models, Gaussian quasi-maximum likelihood estimation (QMLE) is a standard choice. However, the multivariate models fitted by QMLE have two limitations: 
	\begin{enumerate}
		\item[(L1)] (curse of dimensionality) the QMLE gets increasingly unstable and computationally intensive as the cross-sectional dimension $N$ grows. This is because optimizing the QMLE involves repeatedly inverting $N \times N$ conditional covariance matrices \citep{pakel2021,francq2019garch}. 
		\item[(L2)] (non-robustness) the QMLE is not robust for heavy-tailed financial data, as it is sensitive to outliers \citep{boudt2010robust} and requires high-order moment conditions on the data process \citep{comte2003asymptotic,hafner2009multivariate}. 
	\end{enumerate}

    To address (L1), structural restrictions are introduced. 
    For instance, factor GARCH models assume that volatility is driven by a small number of latent factors \citep{lanne2007multivariate, hafner2009factor}, while restricted BEKK specifications constrain parameter matrices to be diagonal or scalar. However, such models often prove overly restrictive, potentially failing to capture more complex dynamics and cross-sectional dependencies. 
    On the other hand, several alternative estimation procedures have been developed to enhance feasibility and computational efficiency for multivariate volatility models, including variance targeting estimation \citep{francq2016variance}, equation by equation approaches \citep{francq2016estimating}, composite likelihood methods \citep{engle2019large}, and penalized methods \citep{poignard2021high,yao20241}. 
    Nevertheless, these techniques either rely on strong assumptions or lack theoretical guarantees in high-dimensional settings. 
    
    To overcome (L2), a number of robust estimation methods have been proposed to mitigate the impact of heavy tails and outliers. Employing a robust loss function offers a useful approach in this regard. 
    Commonly used robust losses include the least absolute deviation (LAD) loss \citep{wang2007robust}, the quantile loss \citep{koenker1978regression}, and the Huber loss \citep{huber1964robust,sun2020adaptive}. However, LAD- and quantile-based criteria are often computationally expensive due to their non-smooth nature, limiting their applicability in high-dimensional settings. Moreover, Huber loss primarily down-weights outliers in response but does not account for heavy-tailed covariates. Consequently, estimators based on this loss may be less effective in multivariate volatility modeling, where both response and covariates are heavy-tailed. 
    An alternative strategy for robust estimation is data truncation. 
    For example, \cite{lugosi2021robust} developed trimmed estimators for the mean, whereas \cite{fan2021shrinkage} and \cite{wang2024robust} considered truncated variants of the sample covariance matrix. 
    There also exist hybrid approaches that integrate robust losses with truncation including the works of \cite{muler2002robust} on univariate ARCH models and \cite{muler2009robust} on ARMA models. However, research on robust estimation for high-dimensional volatility models remains limited, particularly for methods employing data truncation.

	To address the aforementioned research gaps, this paper introduces a scalable and robust estimation method for high-dimensional volatility models with heavy-tailed data. Specifically, we consider the BEKK-ARCH model, which admits a VAR representation after vectorization \citep{caporin2012we}. We propose a regularized least squares estimator (LSE) that incorporates data truncation to achieve robustness while imposing row-wise sparsity in the coefficient matrices to handle high dimensionality. This design preserves the computational efficiency of least squares while mitigating the impact of heavy-tailedness. Importantly, our truncation is applied to the returns themselves rather than to the derived response vector in the VAR representation, distinguishing our method from existing truncation-based robust methods for VAR models \citep{wang2023rate}. The proposed framework is not only novel in robust volatility estimation, but also offers broad applicability to other high-dimensional linear problems, such as robust tensor regressions. We validate its practical value through empirical studies in portfolio construction, where our method demonstrates superior performance in both estimation speed and predictive accuracy compared to existing benchmarks.

	Our main contributions are threefold. 
        First, we develop a robust and computationally efficient two-stage estimation framework for high-dimensional BEKK-ARCH models. The approach not only integrates data truncation with the VAR representation, but also introduces a novel inverse mapping to recover the original BEKK coefficient matrices. This recovery procedure is applicable to general bilinear structures of the form $\mathbf{A}\mathbf{X}\mathbf{B}$ and offers a broadly useful tool for estimating $\mathbf{A}$ and $\mathbf{B}$ given $\mathbf{X}$, which is of independent interest.

         Second, we establish matching non-asymptotic upper and minimax lower bounds for our estimators, thereby proving their minimax optimality under heavy-tailed data. Remarkably, these theoretical guarantees are achieved under only an element-wise $(4 + 4\epsilon)$-moment condition, making them valid across a wide range of heavy-tailed processes.
        Third, we develop a robust high-dimensional Bayesian information criterion (BIC) for order selection and a Ridge-type estimator for determining the number of BEKK-ARCH components, and we rigorously establish their selection consistency under heavy-tailed settings. To the best of our knowledge, this is the first work that provides theoretical guarantees for component number selection in BEKK models, a critical problem that has been largely overlooked in the existing literature.

	The remainder of this paper is organized as follows. Section~\ref{sec:methodology} introduces the BEKK-ARCH model and its VAR representation, and proposes a two-stage estimation procedure consisting of a robust regularized least-squares estimator for the VAR representation and a recovery step for the original BEKK coefficients. Section \ref{sec:algorithm and implementation} outlines the algorithm and implementation details for the proposed estimation procedure, including a BIC-based lag-order selection criterion and a Ridge-type estimator for determining the number of BEKK-ARCH components. Section \ref{sec:theory} establishes non-asymptotic theories for estimators in high-dimensional settings and proves the consistency of model-selection procedures. Simulation results and two real-data applications are presented in Sections \ref{sec:simulation}--\ref{sec:real data analysis}, respectively. Section \ref{sec:conclusion and discussion} provides concluding remarks and discussion. All technical proofs and additional materials are provided in the Supplementary Material.

	Throughout this paper, we adopt the following notation. Scalars are denoted by roman letters (e.g., $a,\tau$), vectors by bold lowercase (e.g., $\bm{u},\bbm{\eta}$), and matrices by bold uppercase letters (e.g., $\mathbf{A},\bm{\Theta}$). For a vector $\bm{u}\in\mathbb{R}^N$ and $q\in[0,\infty]$, define the $\ell_q$-norm as $\|\bm{u}\|_q=(\sum_{i=1}^N |u_i|^q)^{1/q}$ (with $\|\bm{u}\|_0=\sum_i \mathds{1}(u_i\neq 0)$ and $\|\bm{u}\|_\infty=\max_i |u_i|$). For a matrix $\mathbf{A}=(a_{ij})\in\mathbb{R}^{N_1\times N_2}$, let $\operatorname{vec}(\mathbf{A})$ be the column-stacking vectorization and $\operatorname{vech}(\mathbf{A})$ the half-vectorization. Their inverse operations are denoted by $\operatorname{vec}^{-1}(\cdot)$ and $\operatorname{vech}^{-1}(\cdot)$, respectively. The $j$-th column of $\mathbf{A}$ and the submatrix consisting of columns indexed by a set $J$ are denoted by $\mathbf{A}_{\cdot,j}$ and $\mathbf{A}_{\cdot,J}$, respectively. We define the Frobenius inner product as $\langle \mathbf{A},\mathbf{B}\rangle=\mathrm{tr}(\mathbf{A}^\mathrm{T} \mathbf{B})$, with the Frobenius norm $\|\mathbf{A}\|_{\mathrm{F}}=\langle \mathbf{A},\mathbf{A}\rangle^{1/2}$. The operator norm is given by $\|\mathbf{A}\|_{\mathrm{op}}=\sigma_1(\mathbf{A})$, and the nuclear norm by $\|\mathbf{A}\|_*=\sum_i \sigma_i(\mathbf{A})$, where $\sigma_i(\mathbf{A})$ denotes the $i$-th singular value of $\mathbf{A}$. The elementwise max norm is defined as $\|\mathbf{A}\|_{\infty,\infty}=\max_{i,j}|a_{ij}|$. For grouped (columnwise) norms, we write $\|\mathbf{A}\|_{p,q}=\big(\sum_{j=1}^{N_2}\|\mathbf{A}_{\bm{\cdot},j}\|_p^{\,q}\big)^{1/q}$. The Kronecker product is denoted by $\otimes$. We use $C$ to denote a generic positive constant that is independent of the dimension and sample size. For sequences $\{a_n\}$ and $\{b_n\}$, we write $a_n\gtrsim b_n$ if $a_n\ge C b_n$ for all $n$, and $a_n\asymp b_n$ if both $a_n\gtrsim b_n$ and $b_n\gtrsim a_n$ hold.
    The dataset in Section \ref{sec:real data analysis} and computer programs for the analysis are available at \url{https://github.com/CKKQ/High-dim-BEKK}. 
	
	\section{Methodology}\label{sec:methodology}
	\subsection{BEKK-ARCH model and its equivalent form}\label{sec:model}
	Consider a BEKK-ARCH model of order $p$ for an $N$-dimensional time series $\{\bm{r}_{t}\}$: 
	\begin{align}\label{eq:BEKK-ARCH}
		\begin{cases}
			\bm{r}_t = \bm{\Sigma}_t^\text{1/2}\bbm{\eta}_t,\\
			\bm{\Sigma}_t = \bbm{\Omega} +\sum_{i=1}^{p}\sum_{k=1}^{K_i} \mathbf{A}_{ik}\bm{r}_{t-i}\bm{r}_{t-i}^{\mathrm{T}}\mathbf{A}_{ik}^{\mathrm{T}},
		\end{cases}
	\end{align}
	where $\bm{r}_{t}, \bbm{\eta}_{t} \in \mathbb{R}^{N}$, $\{\bbm{\eta}_{t}\}$ are independent and identically distributed random vectors with zero mean and identity covariance matrix, $\bm{\Sigma}_{t}\in \mathbb{R}^{N \times N}$ is the conditional covariance matrix of $\bm{r}_{t}$ given the $\sigma$-field $\mathcal{F}_{t-1} = \sigma\{\bm{r}_{t-1}, \bm{r}_{t-2}, \ldots\}$, $\bbm{\Omega} \in \mathbb{R}^{N \times N}$ and $\mathbf{A}_{ik}\in \mathbb{R}^{N \times N}$ with $1\leq i\leq p$ and $1\leq k\leq K_i$ are parameter matrices, and $\bbm{\Omega}$ is assumed to be positive definite. 
	Model \eqref{eq:BEKK-ARCH} permits varying values of $K_i$ for $1\leq i\leq p$, while the restricted case $K_1=\cdots=K_p=K$ is typically employed for practical convenience. 
	To accommodate high-dimensional settings of Model \eqref{eq:BEKK-ARCH}, we assume that each $\mathbf{A}_{ik}$ is row-wise sparse with $\max_{1\leq j\leq N}\|(\mathbf{A}_{ik})_{j,\cdot}\|_0\leq s_A$, while no structural assumption is imposed on $\bbm{\Omega}$. This condition still allows all diagonal entries of $\mathbf{A}_{ik}$ to be nonzero, so that Model \eqref{eq:BEKK-ARCH} can remain nondegenerate.
	Noted that $\mathbf{A}_{ik}$ and $-\mathbf{A}_{ik}$ can generate the same $\bm{\Sigma}_t$, making the mapping from $\mathbf{A}_{ik}$ to $\bm{\Sigma}_t$ non-unique. 
	To solve this identification issue, we assume one nonzero element of each $\mathbf{A}_{ik}$ to be positive as in \cite{engle1995multivariate}. 

        The BEKK model in \eqref{eq:BEKK-ARCH} is a widely used framework for modeling conditional covariance matrices of financial returns. However, its applicability in high-dimensional settings is severely hindered by two major limitations. First, it involves $O(N^2)$ parameters, making it susceptible to the curse of dimensionality. Second, conventional likelihood-based estimation becomes computationally unstable and often infeasible for large $N$, due to the need for repeated inversion of large conditional covariance matrices during optimization. 
        To reduce the dimension without inducing degeneracy in the conditional covariance, we impose row-wise sparsity on $\bbm{\Omega}$ and on each $\mathbf{A}_{ik}$, subject to the constraint that their diagonal elements remain nonzero.        

        We then leverage the VAR representation of model \eqref{eq:BEKK-ARCH} to enable the use of a regularized least squares estimator (LSE), which circumvents the computational bottlenecks of likelihood-based methods and ensures both tractability and theoretical robustness. 
        
        Specifically, model \eqref{eq:BEKK-ARCH} can be rewritten in the following linear form:
        \begin{align}\label{eq:vec linear form}
                \operatorname{vec}(\bm{r}_t\bm{r}_t^{\mathrm{T}}) = \operatorname{vec}(\bbm{\Omega}) + \sum_{i=1}^{p}\sum_{k=1}^{K_i} (\mathbf{A}_{ik} \otimes \mathbf{A}_{ik})\operatorname{vec}(\bm{r}_{t-i}\bm{r}_{t-i}^{\mathrm{T}}) + \operatorname{vec}(\mathbf{E}_t), 
        \end{align} 
        where $\mathbf{E}_t=\bm{\Sigma}_t^{1/2} (\bbm{\eta}_t\bbm{\eta}_t^{\mathrm{T}}-\mathbf{I}_N) \bm{\Sigma}_t^{1/2}$ with $\mathbb{E}(\mathbf{E}_t)=\bm{0}$. 
        To address the redundancy in $\operatorname{vec}(\bm{r}_t\bm{r}_t^{\mathrm{T}})$ caused by the symmetry of matrix $\bm{r}_t\bm{r}_t^{\mathrm{T}}$, we use its half-vectorization instead. 
        Denote $\bm{y}_t=\operatorname{vech}(\bm{r}_t\bm{r}_t^{\mathrm{T}}) \in \mathbb{R}^d$ with $d=N(N+1)/2$, $\bbm{\omega}=\operatorname{vech}(\bbm{\Omega}) \in \mathbb{R}^d$, and $\bm{\Phi}_i=\mathbf{D}_N^{\dagger}(\sum_{k=1}^{K_i}\mathbf{A}_{ik} \otimes \mathbf{A}_{ik})\mathbf{D}_N\in \mathbb{R}^{d \times d}$, where $\mathbf{D}_N \in \mathbb{R}^{N^2 \times d}$ is the duplication matrix and $\mathbf{D}_N^{\dagger}=(\mathbf{D}_N^{\mathrm{T}}\mathbf{D}_N)^{-1}\mathbf{D}_N^{\mathrm{T}} \in \mathbb{R}^{d \times N^2}$ is the corresponding elimination matrix; see Section S.6.1 of Supplementary Material for their detailed forms. Then, model \eqref{eq:BEKK-ARCH} can be reformulated into the following VAR form: 
        \begin{align}\label{eq:vech linear form}
                \bm{y}_t= \bbm{\omega}+\sum_{i=1}^{p}\bm{\Phi}_i\bm{y}_{t-i}+\bm{e}_t,
        \end{align}   
        where $\bm{e}_t=\operatorname{vech}(\mathbf{E}_t) \in \mathbb{R}^d$ is a zero-mean innovation vector. Unlike the traditional VAR model, which requires $\bm{e}_t$ to be white noise, model \eqref{eq:vech linear form} accommodates a much broader class of processes by allowing $\bm{e}_t$ to exhibit serial correlation and heteroskedasticity. This flexibility makes it particularly suitable for modeling financial returns. Importantly, the row-wise sparsity of $\bbm{\Omega}$ and $\mathbf{A}_{ik}$ in the original BEKK model is preserved in the parameter vector $\bbm{\omega}$ and matrices $\bm{\Phi}_i$. 
        
        Model \eqref{eq:vech linear form} allows for efficient estimation of $\bbm{\omega}$ and $\bm{\Phi}_i$'s via regularized LSE, a convex optimization problem well-suited for high-dimensional settings. 
        Since the conditional covariance matrix is given by $\bm{\Sigma}_t=\operatorname{vech}^{-1}(\bbm{\omega}+\sum_{i=1}^{p}\bm{\Phi}_i\bm{y}_{t-i})$, then we can estimate $\bm{\Sigma}_t$ by replacing $\bbm{\omega}$ and $\bm{\Phi}_i$'s with their regularized LSEs. 
        Based on $\bbm{\omega}$ and $\bm{\Phi}_i$'s, we can also recover the original BEKK-ARCH parameters $\bbm{\Omega}$ and $\mathbf{A}_{ik}$'s for estimating $\bm{\Sigma}_t$ via model \eqref{eq:BEKK-ARCH} directly, as detailed in Section \ref{sec:vech-to-bekk}. 
        
	\subsection{Estimation}\label{sec:estimation}

	To estimate model \eqref{eq:vech linear form} and the associated conditional covariance matrix $\bm{\Sigma}_t$ in heavy-tailed and high-dimensional settings, this section proposes a robust procedure based on data truncation and regularized least squares estimation (LSE). 
	
	Denote $\mathbf{X}=(\bm{x}_1,\dots,\bm{x}_T)^{\mathrm{T}} \in \mathbb{R}^{T \times (pd+1)}$ with $\bm{x}_t=(1,\bm{y}_{t-1}^{\mathrm{T}},\dots,\bm{y}_{t-p}^{\mathrm{T}})^{\mathrm{T}} \in \mathbb{R}^{pd+1}$, $\mathbf{Y}=(\bm{y}_1,\dots,\bm{y}_T)^{\mathrm{T}} \in \mathbb{R}^{T \times d}$, $\mathbf{E}=(\mathbf{e}_1,\dots,\mathbf{e}_T)^{\mathrm{T}} \in \mathbb{R}^{T \times d}$, and $\bm{\Theta}=(\bbm{\omega},\bm{\Phi}_1,\dots,\bm{\Phi}_p)^{\mathrm{T}} \in \mathbb{R}^{(pd+1) \times d}$. 
	Model \eqref{eq:vech linear form} can then be rewritten into the following matrix form: 
	\begin{align}\label{eq:stacked-VAR}
		\mathbf{Y} = \mathbf{X}\bm{\Theta} + \mathbf{E}.
	\end{align}
	Owing to the row-wise sparsity of $\bm{\Phi}_i$'s, the coefficient matrix $\bm{\Theta}$ is column-wise sparse. 
	Since financial returns $\{\bm{r}_{t}\}$ are typically heavy-tailed, a direct regularized LSE applied to \eqref{eq:stacked-VAR} can be unstable. To address this, we first apply an element-wise truncation to the original return series $\{\bm{r}_{t}\}$ for robustness, and then perform the regularized LSE on the truncated data. This strategy is particularly suitable for volatility modeling, as $\{\bm{y}_t\}$ is generally not mean-zero and its tail behavior is hence more difficult to characterize.

	Specifically, for a given truncation parameter $\tau>0$, define the truncated return component-wise as $r_{t,j}(\tau)=\operatorname{sign}(r_{t,j})\min\{|r_{t,j}|,\tau\}$ with $1 \leq t \leq T$ and $1 \leq j \leq N$, and let $\bm{r}_t(\tau)=(r_{t,1}(\tau),r_{t,2}(\tau),\dots,r_{t,N}(\tau))^{\mathrm{T}}$. 
	Based on this, we construct the truncated response matrix $\mathbf{Y}(\tau)=(\bm{y}_1(\tau),\dots,\bm{y}_T(\tau))^{\mathrm{T}}$ with $\bm{y}_t(\tau)=\operatorname{vech}(\bm{r}_t(\tau)\bm{r}_t^{\mathrm{T}}(\tau))$, and the truncated design matrix $\mathbf{X}(\tau)=(\bm{x}_1(\tau),\dots,\bm{x}_T(\tau))^{\mathrm{T}}$ with $\bm{x}_t(\tau)=(1,\bm{y}_{t-1}^{\mathrm{T}}(\tau),\dots,\bm{y}_{t-p}^{\mathrm{T}}(\tau))^{\mathrm{T}}$. 
	The regularized least squares estimator (LSE) for model \eqref{eq:stacked-VAR} based on these truncated data is then given by:
	\begin{align}\label{eq:loss function}
       \widehat{\bm{\Theta}}(\lambda, \tau)=\bigl(\widehat{\bbm{\omega}},\widehat{\bm{\Phi}}_1,\dots,\widehat{\bm{\Phi}}_p\bigr)^{\mathrm{T}} \in \operatorname{argmin}_{\bm{\Theta} \in \mathbb{R}^{(pd+1) \times d}} \limits \frac{1}{2T}\|\mathbf{Y}(\tau)-\mathbf{X}(\tau)\bm{\Theta}\|_{\mathrm{F}}^2+\lambda\|\bm{\Theta}\|_{1,1},
	\end{align}
	where $\lambda>0$ is a regularization parameter balancing model fit with complexity, and the truncation parameter $\tau$ controls the trade-off between truncation bias and robustness. Specifically, a smaller $\tau$ enhances robustness to heavy tails at the cost of increased truncation bias. The implementation details, including the algorithm for solving the regularized LSE in \eqref{eq:loss function} and the data-driven selection of $\lambda$ and $\tau$, are thoroughly discussed in Section \ref{sec:algorithm and implementation}.
	Simulation results in Section \ref{sec:simulation} indicate that $\widehat{\bm{\Theta}}(\lambda, \tau)$ outperforms the estimator without data truncation if the dimension is large and data are heavier-tailed.

	Using the regularized LSE, the conditional covariance matrix $\bm{\Sigma}_t$ can be estimated as 
	\begin{align}\label{eq:breve_Sigma_t}
       \widebreve{\bm{\Sigma}}_t=\mathcal{P}\!\left(\operatorname{vech}^{-1}\!\big(\widehat{\bm{\Theta}}^{\mathrm T}(\lambda,\tau)\bm x_t\big)\right),
	\end{align}
	where $\mathcal{P}(\cdot)$ denotes the projection operator onto the cone of symmetric positive definite (SPD) matrices. In practice, the estimated matrix $\operatorname{vech}^{-1}(\widehat{\bm{\Theta}}^{\mathrm{T}}\bm{x}_t)$ is typically positive definite under moderate sample sizes, as confirmed by the simulation results in Section S.7 of Supplementary Material. Nonetheless, to maintain methodological rigor, we retain the projection step in our estimation procedure.

 \subsection{Recovering BEKK-ARCH model from its vech form}\label{sec:vech-to-bekk}

        This section studies recovering $\bbm{\Omega}$ and $\mathbf{A}_{ik}$'s of model \eqref{eq:BEKK-ARCH} from $\bbm{\omega}$ and $\bm{\Phi}_i$'s of model \eqref{eq:vech linear form}, and then obtains estimators of $\bbm{\omega}$ and $\mathbf{A}_{ik}$'s using the regularized LSE at \eqref{eq:loss function}. 

        Recall that $\bbm{\omega}=\operatorname{vech}(\bbm{\Omega}) \in \mathbb{R}^d$ and $\bm{\Phi}_i=\mathbf{D}_N^{\dagger}(\sum_{k=1}^{K_i}\mathbf{A}_{ik} \otimes \mathbf{A}_{ik})\mathbf{D}_N\in \mathbb{R}^{d \times d}$, where $\mathbf{D}_N \in \mathbb{R}^{N^2 \times d}$ and $\mathbf{D}_N^{\dagger}=(\mathbf{D}_N^{\mathrm{T}}\mathbf{D}_N)^{-1}\mathbf{D}_N^{\mathrm{T}} \in \mathbb{R}^{d \times N^2}$ are the duplication matrix and elimination matrix, respectively.
        Applying the inverse duplication operation to $\bbm{\omega} = \operatorname{vech}(\bbm{\Omega})$ readily gives $\bbm{\Omega}=\operatorname{vec}^{-1}(\mathbf{D}_N\bbm{\omega})$. However, recovering the original coefficient matrices ${\mathbf{A}_{ik}}$ from $\bm{\Phi}_i$ is a nontrivial and non-unique problem due to two key difficulties: (D1) the mapping $\sum_{k=1}^{K_i}\mathbf{A}_{ik}\otimes\mathbf{A}_{ik}\mapsto\bm{\Phi}_i$ is many-to-one; (D2) the mapping $\{\mathbf{A}_{ik}\}_{k=1}^{K_i}\mapsto\sum_{k=1}^{K_i}\mathbf{A}_{ik}\otimes \mathbf{A}_{ik}$ is also many-to-one. Consequently, without imposing additional structural restrictions, it is challenging to uniquely recover $\sum_{k=1}^{K_i}\mathbf{A}_{ik}\otimes\mathbf{A}_{ik}$ from $\bm{\Phi}_i$, and then to identify $\{\mathbf{A}_{ik}\}_{k=1}^{K_i}$ from this Kronecker sum.

        We begin by examining the cause of (D1) and then present a solution to resolve it. From models \eqref{eq:vec linear form} and \eqref{eq:vech linear form}, it can be observed that $\sum_{k=1}^{K_i}\mathbf{A}_{ik}\otimes \mathbf{A}_{ik}$ corresponds to the coefficients of $\operatorname{vec}(\bm{r}_{t-i}\bm{r}_{t-i}^{\mathrm{T}})$, while $\bm{\Phi}_i$ corresponds to those of $\bm{y}_{t-i}=\operatorname{vech}(\bm{r}_{t-i}\bm{r}_{t-i}^{\mathrm{T}})$. Consequently, certain entries in $\bm{\Phi}_i$ merge coefficients that correspond to different cross-products in $\operatorname{vec}(\bm{r}_{t-i}\bm{r}_{t-i}^{\mathrm{T}})$, namely $r_{t-i,j_1}r_{t-i,j_2}$ and $r_{t-i,j_2}r_{t-i,j_1}$ for $j_1\neq j_2$. For illustration, consider the case where $N=2$ and $K_i=1$, and denote $\mathbf{A}_{i1}=(a_{jl})_{2\times2}$. Then we have $\bm{\Phi}_i=\mathbf{D}_2^{\dagger}\,(\mathbf{A}_{i1}\!\otimes \mathbf{A}_{i1})\,\mathbf{D}_2$,
        \begin{align*}
                \mathbf{A}_{i1}\otimes \mathbf{A}_{i1}=
                \begin{bmatrix}
                        a_{11}^2 & a_{11}a_{12} & a_{12}a_{11} & a_{12}^2\\
                        a_{11}a_{21} & a_{11}a_{22} & a_{12}a_{21} & a_{12}a_{22}\\
                        a_{21}a_{11} & a_{21}a_{12} & a_{22}a_{11} & a_{22}a_{12}\\
                        a_{21}^2 & a_{21}a_{22} & a_{22}a_{21} & a_{22}^2
                \end{bmatrix} \;\;\text{and}\;\; 
                \bm{\Phi}_i
                =\begin{bmatrix}
                        a_{11}^2 & 2a_{11}a_{12} & a_{12}^2\\
                        a_{11}a_{21} & a_{11}a_{22}+a_{12}a_{21} & a_{12}a_{22}\\
                        a_{21}^2 & 2a_{21}a_{22} & a_{22}^2
                \end{bmatrix}.
        \end{align*}
        Here, the entry $a_{11}a_{22}+a_{12}a_{21}$ in $\bm{\Phi}_i$ combines the two distinct terms $a_{11}a_{22}$ and $a_{12}a_{21}$ from $\mathbf{A}_{i1}\otimes \mathbf{A}_{i1}$, making them indistinguishable from their sum alone. In the general case with arbitrary $N$ and $K_i$, there are $g^2$ such coefficient pairs to distinguish, where $g=N(N-1)/2$.
        For each pair, we assign one parameter as the coefficient of $r_{t-i,j_1}r_{t-i,j_2}$, and set the coefficient of $r_{t-i,j_2}r_{t-i,j_1}$ as the corresponding entry in $\bm{\Phi}_i$ minus that parameter. These $g^2$ parameters are collected into an auxiliary matrix $\mathbf{W}_i\in\mathbb{R}^{g\times g}$. We then employ a padding operator $\mathcal{H}(\cdot,\cdot):\mathbb{R}^{d\times d}\times\mathbb{R}^{g\times g}\mapsto \mathbb{R}^{N^2\times N^2}$ to map $\bm{\Phi}_i$ back to $\sum_{k=1}^{K_i}\mathbf{A}_{ik}\otimes \mathbf{A}_{ik}$; see Section S.6.3 of the Supplementary Material for its detailed construction. When the coefficients are correctly assigned, $\mathcal{H}(\cdot,\cdot)$ exactly reconstructs the Kronecker sum:
        \begin{align}\label{eq:allocation operator}
                \mathcal{H}(\bm{\Phi}_i,\mathbf{W}_i) \;=\; \sum_{k=1}^{K_i} \mathbf{A}_{ik} \otimes \mathbf{A}_{ik}.
        \end{align}
        To address (D1), we further introduce a rearrangement operator $\mathcal{R}(\cdot):\mathbb{R}^{N^2\times N^2}\mapsto\mathbb{R}^{N^2\times N^2}$ such that $\mathcal{R}(\mathbf{M}\otimes \mathbf{M}) = \operatorname{vec}(\mathbf{M})\,\operatorname{vec}^\mathrm{T}(\mathbf{M})$ for any $\mathbf{M}\in\mathbb{R}^{N\times N}$. The  construction of $\mathcal{R}(\cdot)$ is provided in Section S.6.2 of Supplementary Material. Applying $\mathcal{R}(\cdot)$ to \eqref{eq:allocation operator}, then we have
        \begin{equation}\label{eq:pop-rearrangement}
                \mathcal{R}\left(\mathcal{H}(\bm{\Phi}_i,\mathbf{W}_i)\right)
                =\sum_{k=1}^{K_i}\operatorname{vec}(\mathbf{A}_{ik})\,\operatorname{vec}^\mathrm{T}(\mathbf{A}_{ik}) \quad\text{with}\quad \mathrm{rank}(\mathcal{R}\left(\mathcal{H}(\bm{\Phi}_i,\mathbf{W}_i)\right))\leq K_i.
        \end{equation}
        This low rank property serves as a criterion for identifying the Kronecker sum $\sum_{k=1}^{K_i}\mathbf{A}_{ik}\otimes \mathbf{A}_{ik}$ in (D1).

        We now address the identifiability issue in (D2). To ensure that the matrices $\mathbf{A}_{ik}$ can be uniquely recovered from their Kronecker sum $\sum_{k=1}^{K_i}\mathbf{A}_{ik}\otimes \mathbf{A}_{ik}$, we impose the constraints given in Assumption \ref{assumption:restricted parameter space}. Specifically, condition (i) in Assumption \ref{assumption:restricted parameter space} corresponds to a reparameterization that leaves the model unchanged, while condition (ii) introduces a mild ordering constraint to ensure uniqueness, which is generally nonrestrictive in practice. The validity of this reparameterization is formally established in Proposition \ref{prop:equivalence of orthohognol forms}. 
        Under Assumption \ref{assumption:restricted parameter space}, the matrix $\mathcal{R}(\mathcal{H}(\bm{\Phi}_i,\mathbf{W}_i))$ is symmetric positive semidefinite and possesses exactly $K_i$ positive eigenvalues. Its nonzero eigenpairs are given by $\big(\|\mathbf{A}_{ik}\|_{\mathrm{F}}^{2}, \operatorname{vec}(\mathbf{A}_{ik})/\|\mathbf{A}_{ik}\|_{\mathrm{F}}\big)$, sorted in decreasing order for $k=1,\dots,K_i$. This spectral structure guarantees a one-to-one correspondence between the Kronecker sum and $\{\mathbf{A}_{ik}\}_{k=1}^{K_i}$, thereby resolving the identifiability issue stated in (D2).

        \begin{assumption}[Restricted Parameter Space]\label{assumption:restricted parameter space}
                For each $1 \leq i \leq p$, the matrices $\mathbf{A}_{ik}$ satisfy: (i) $\langle\mathbf{A}_{ik_1},\mathbf{A}_{ik_2}\rangle = 0$ for all $1 \leq k_1 < k_2 \leq K_i$; (ii) $\mathbf{A}_{ik}$ are arranged in strictly descending order of Frobenius norms, i.e., $\|\mathbf{A}_{i1}\|_\mathrm{F} > \|\mathbf{A}_{i2}\|_\mathrm{F} > \cdots > \|\mathbf{A}_{iK_i}\|_\mathrm{F} > 0$.
        \end{assumption}

        \begin{proposition}[Orthogonalization invariance of BEKK coefficient matrices]\label{prop:equivalence of orthohognol forms}
                Let $\{\mathbf{B}_{ik}\}_{k=1}^{K_i^B} \subset \mathbb{R}^{N\times N}$ be the collection of coefficient matrices at lag $i$ in the BEKK model \eqref{eq:BEKK-ARCH}.
                Then there exists an integer $K_i^A \leq K_i^B$ and a set of matrices $\{\mathbf{A}_{ik}\}_{k=1}^{K_i^A} \subset \mathbb{R}^{N\times N}$ that are pairwise orthogonal, i.e., $\langle \mathbf{A}_{ij}, \mathbf{A}_{i\ell}\rangle = 0$ for all $j \neq \ell$, such that for any $\mathbf{M} \in \mathbb{R}^{N\times N}$,
                \[
                \sum_{k=1}^{K_i^B} \mathbf{B}_{ik} \, \mathbf{M} \, \mathbf{B}_{ik}^{\mathrm{T}}
                \;=\;
                \sum_{k=1}^{K_i^A} \mathbf{A}_{ik} \, \mathbf{M} \, \mathbf{A}_{ik}^{\mathrm{T}} .
                \]
        \end{proposition}

        We are ready to estimate the parameter matrices $\bbm{\Omega}$ and $\mathbf{A}_{ik}$'s of the BEKK model in \eqref{eq:BEKK-ARCH} from the regularized LSEs $\widehat{\bbm{\omega}}$ and $\widehat{\bm{\Phi}}_i$.
        Given that $\bbm{\Omega}=\operatorname{vec}^{-1}(\mathbf{D}_N\bbm{\omega})$, a natural estimator is $\widehat{\bbm{\Omega}}=\mathcal{P}(\operatorname{vec}^{-1}(\mathbf{D}_N\widehat{\bbm{\omega}}))$.
        To estimate $\{\mathbf{A}_{ik}\}$ uniquely, we use the identity in \eqref{eq:pop-rearrangement} together with the minimal rank property. Under Assumption \ref{assumption:restricted parameter space}, the rank of $\mathcal{R}(\mathcal{H}(\bm{\Phi}_i,\mathbf{W}_i))$ attains the lower bound $K_i$ for the true matrices $\bm{\Phi}_i$ and $\mathbf{W}_i$. The key step is therefore to estimate the auxiliary matrix $\mathbf{W}_i$.  
        Given $\widehat{\bm{\Phi}}_i$, a natural convex relaxation for estimating $\mathbf{W}_i$ is 
        \begin{equation}\label{eq:nuclear-argmin}
                \widetilde{\mathbf{W}}_i
                \;\in\;
                \arg\min_{\mathbf{W}}
                \big\|
                \mathcal{R}\big(\mathcal{H}(\widehat{\bm{\Phi}}_i,\mathbf{W})\big)
                \big\|_*,
                \ \ \text{for} \ \ i=1,\dots,p.
        \end{equation}
        Then $\mathbf{A}_{ik}$ can be estimated via eigen-decomposition of $\mathcal{R}(\mathcal{H}(\widehat{\bm{\Phi}}_i,\widetilde{\mathbf{W}}_i))$, and the nuclear-norm-based estimator is given by $\widetilde{\mathbf{A}}_{ik}=\operatorname{vec}^{-1}\!\big(\sqrt{\widetilde{\lambda}_{i,k}^{+}}\,\widetilde{\bm{u}}_{i,k}\big)$ for $k=1,\ldots,K_i$, where $\widetilde{\lambda}_{i,k}^{+}=\max\{\widetilde{\lambda}_{i,k},0\}$ and $(\widetilde{\lambda}_{i,k},\widetilde{\bm{u}}_{i,k})$ are the top $K_i$ eigenpairs of $\mathcal{R}(\mathcal{H}(\widehat{\bm{\Phi}}_i,\widetilde{\mathbf{W}}_i))$.

        However, solving \eqref{eq:nuclear-argmin} incurs a computational cost of $O(N^6)$, which becomes computationally prohibitive for large $N$. Moreover, the nuclear norm objective may not sufficiently separate the leading $K_i$ eigenvalues from the remainder. To address these issues, we instead adopt a top-eigenvalue (TE) loss for estimating $\mathbf W_i$, defined as
        \begin{equation}\label{eq:te-argmin}
                \widehat{\mathbf W}_i
                \;\in\;
                \arg\min_{\mathbf W}
                \;\mathcal{L}_{\mathrm{TE}}\!\left(\mathcal{R}\big(\mathcal{H}(\widehat{\bm{\Phi}}_i,\mathbf W)\big)\right), 
                \ \ \text{for} \ \ i=1,\dots,p,
        \end{equation}
        where the TE loss $\mathcal{L}_{\mathrm{TE}}(\mathbf{M}) = -\sum_{j=1}^{K_i} \lambda_j(\mathbf{M}) + \gamma_{\mathrm{TE}}\sum_{j=K_i+1}^{N^2} \lambda_j^2(\mathbf{M})$ encourages a large gap between the top $K_i$ eigenvalues and the trailing ones, with the tuning parameter $\gamma_{\mathrm{TE}}>0$ controlling the penalization on trailing eigenvalues. In practice, $K_i$ can be estimated using the Ridge-type selector described in Section~\ref{subsec:model order selection}.      
        Crucially, evaluating this loss requires only the Frobenius norm and the top $K_i$ eigenvalues of $\mathcal{R}(\mathcal{H}(\widehat{\bm{\Phi}}_i,\mathbf W))$, reducing the computational cost from $O(N^6)$ to $O(K_iN^4)$.
        Similar to the nuclear-norm-based estimator $\widetilde{\mathbf{A}}_{ik}$, the TE-loss-based estimator of $\mathbf{A}_{ik}$ is then defined as $\widehat{\mathbf{A}}_{ik}=\operatorname{vec}^{-1}\!\big(\sqrt{\widehat{\lambda}_{i,k}^{+}}\,\widehat{\bm{u}}_{i,k}\big)$ for $k=1,\ldots,K_i$, where $\widehat{\lambda}_{i,k}^{+}=\max\{\widehat{\lambda}_{i,k},0\}$ and $(\widehat{\lambda}_{i,k},\widehat{\bm{u}}_{i,k})$ are the top $K_i$ eigenpairs of $\mathcal{R}(\mathcal{H}(\widehat{\bm{\Phi}}_i,\widehat{\mathbf{W}}_i))$.

        Finally, we construct the following plug-in estimators for the conditional covariance matrix based on the BEKK model in \eqref{eq:BEKK-ARCH} and the recovered BEKK coefficient matrices:
        \begin{align}\label{eq:tilde_Sigma_t_and_hat_Sigma_t}
                \widetilde{\bm{\Sigma}}_t
                =\widehat{\bbm{\Omega}}+\sum_{i=1}^{p}\sum_{k=1}^{K_i}
                \widetilde{\mathbf{A}}_{ik}\bm r_{t-i}\bm r_{t-i}^{\mathrm T}\widetilde{\mathbf{A}}_{ik}^{\mathrm T}
                \quad\text{and}\quad
                \widehat{\bm{\Sigma}}_t
                =\widehat{\bbm{\Omega}}+\sum_{i=1}^{p}\sum_{k=1}^{K_i}
                \widehat{\mathbf{A}}_{ik}\bm r_{t-i}\bm r_{t-i}^{\mathrm T}\widehat{\mathbf{A}}_{ik}^{\mathrm T}.
        \end{align}
        These estimators complement $\widebreve{\bm{\Sigma}}_t$ defined in \eqref{eq:breve_Sigma_t}. In Section~\ref{sec:simulation}, we compare their finite-sample performance and find that both $\widehat{\bm{\Sigma}}_t$ and $\widetilde{\bm{\Sigma}}_t$ consistently outperform $\widebreve{\bm{\Sigma}}_t$ in statistical accuracy. Moreover, $\widehat{\bm{\Sigma}}_t$ further outperforms $\widetilde{\bm{\Sigma}}_t$ in both statistical accuracy and computational efficiency. Therefore, we recommend using $\widehat{\bm{\Sigma}}_t$ in practice.

\section{Implementation Issues}\label{sec:algorithm and implementation}

        \subsection{Algorithm}
        To estimate BEKK model \eqref{eq:BEKK-ARCH} and the conditional covariance matrix $\bm{\Sigma}_t$, this subsection presents a two-stage estimation procedure:  
        \begin{enumerate}
                \item[(S1)] Obtain the regularized LSE $\widehat{\bm{\Theta}}$ in \eqref{eq:loss function} for model \eqref{eq:vech linear form}, as described in Algorithm \ref{alg:fista-core};
                \item[(S2)] Use $\widehat{\bm{\Theta}}$ to calculate $\widehat{\bbm{\Omega}}$ and $\{\widehat{\mathbf{A}}_{ik}\}$ or $\{\widetilde{\mathbf{A}}_{ik}\}$ for model \eqref{eq:BEKK-ARCH}, as described in Algorithm \ref{alg:recovery-omega-aik}.
        \end{enumerate}

        For (S1), we apply the fast iterative shrinkage-thresholding algorithm (FISTA) of \cite{beck2009fast} to solve the $\ell_{1,1}$-penalized convex problem in \eqref{eq:loss function}. We have several comments for Algorithm \ref{alg:fista-core}: 
        (i) The selected model order $p$, truncation parameter $\tau>0$ and regularization parameter $\lambda>0$ will be discussed in Sections \ref{subsec:model order selection}--\ref{subsec:tuning parameter selection}; 
        (ii) The tolerance parameter $\varrho>0$ is set to $\varrho=10^{-3}$ for stopping criterion, and the block size $B\in\mathbb{N}$ is set to $B=256$ for faster speed (see (v) for details); 
        (iii) $\mathcal{T}$ denotes the element-wise soft-thresholding operator to a matrix, i.e. $(\mathcal{T}_{\rho}(\mathbf{M}))_{lj}=\operatorname{sign}(m_{lj})\max\{|m_{lj}|-\rho, 0\}$ with $\mathbf{M}=(m_{lj})$. Moreover, $\mathbf{X}^\mathrm{T}(\tau) \mathbf{X}(\tau)$ and $\mathbf{X}^\mathrm{T} (\tau)\mathbf{Y}(\tau)$ can be computed before the loop to reduce computational burden;
        (iv) We simply set $\bm{\Theta}^{(0)}=\mathbf{0}$, as the objective function in \eqref{eq:loss function} is convex and FISTA converges from any starting point with the standard $O(1/n^{2})$ rate, where $n$ is the inner iteration index;  
        (v) All computations are accelerated using GPUs, with on-device memory serving as the limiting constraint. Directly solving \eqref{eq:loss function} without decomposition would require a peak memory footprint of $\mathcal{O}((pd+T)d)$, which becomes infeasible as $N$ increases. To mitigate this, we employ a blockwise strategy that updates $B$ columns at a time, thus reducing the peak device memory to $\mathcal{O}((pd+T)B)$ while preserving high computational throughput. In practice, $B$ is chosen as large as possible to fully utilize available device memory, thereby maximizing speed.

        For (S2), the optimization problems in \eqref{eq:nuclear-argmin} and \eqref{eq:te-argmin} are independent of the sample size and thus do not pose statistical estimation challenges.     
        Hence, we solve the problems in \eqref{eq:nuclear-argmin} and \eqref{eq:te-argmin} using the Adam optimizer during the padding stage of Algorithm \ref{alg:recovery-omega-aik}. 
        For moderate dimensions (e.g., $N \leq 20$), \eqref{eq:nuclear-argmin} or \eqref{eq:te-argmin} can be solved directly via Adam with random initialization. For larger $N$, we exploit the sparsity of $\mathcal{H}(\bm{\Phi}_i,\mathbf{W}_i)$ by truncating small values to zero after each iteration, thereby accelerating computation while preserving solution quality.

        The computational cost of our two-stage estimation method is $O(TN^2+p\sum_{i=1}^{p}K_iN^4)$ when the regularized LSE in \eqref{eq:loss function} is solved by FISTA as described in Algorithm~\ref{alg:fista-core}, and the recovery of BEKK coefficient matrices is performed using \eqref{eq:te-argmin}. In contrast, solving the QMLE for the BEKK model incurs a computational cost of $O(TN^3)$. Under the condition $T \gtrsim N^{1+1/\epsilon}\log(pd+1)$ specified in Theorem~\ref{thm:vech upper bound}, it follows that our method is computationally more efficient than the QMLE, even in high dimensions with large $N$. Moreover, this computational advantage becomes increasingly pronounced as the sample size $T$ grows.

        \algrenewcommand\algorithmicfor{\textbf{For}}
        \algrenewtext{EndFor}{\textbf{End for}}
        \begin{breakablealgorithm}
                \caption{Blockwise FISTA for $\widehat{\bm{\Theta}}$}
                \label{alg:fista-core}
                \begin{algorithmic}[1]
                        \State \textbf{Input:} $p \geq 1$, $\tau>0$, $\lambda>0$, $\varrho>0$, $B\in\mathbb{N}$,
                        $\mathbf{Y}(\tau)\!\in\!\mathbb{R}^{T\times d}$, $\mathbf{X}(\tau)\!\in\!\mathbb{R}^{T\times(pd+1)}$
                        \State \textbf{Initialize:} $\bm{\Theta}^{(0)}=\mathbf{U}^{(0)}=\mathbf{0} \in \mathbb{R}^{(pd+1)\times d}$, $t_0=1$
                        \State \textbf{Set:} step size $\eta=1/L$, where $L=\|\mathbf{X}(\tau)\|_{\mathrm{op}}^2/T$
                        \For{$j=1,\,1+B,\,1+2B,\,\dots$ \textbf{ set } $J \gets \{j,\dots,\min(j+B-1,d)\}$}
                        \For{$n=0,1,2,\dots$}
                        \State $\mathbf{G}^{(n)} \gets \frac{1}{T}\mathbf{X}(\tau)^{\mathrm{T}}(\mathbf{X}(\tau)\mathbf{U}^{(n)}_{\cdot,J}-\mathbf{Y}(\tau)_{\cdot,J})$
                        \State $\bm{\Theta}^{(n+1)}_{\cdot,J} \gets \mathcal{T}_{\lambda\eta}\!\big(\mathbf{U}^{(n)}_{\cdot,J} - \eta\mathbf{G}^{(n)}\big)$
                        \State $\mathbf{U}^{(n+1)}_{\cdot,J}\gets \bm{\Theta}^{(n+1)}_{\cdot,J}+\big(\frac{t_n-1}{t_{n+1}}\big)\big(\bm{\Theta}^{(n+1)}_{\cdot,J}-\bm{\Theta}^{(n)}_{\cdot,J}\big),\;\text{where}\;t_{n+1} = \frac{1+\sqrt{1+4t_n^2}}{2}$
                        \State \textbf{if} $\|\bm{\Theta}^{(n+1)}_{\cdot,J}-\bm{\Theta}^{(n)}_{\cdot,J}\|_\mathrm{F}\big/\big\|\bm{\Theta}^{(n)}_{\cdot,J}\|_\mathrm{F}<\varrho$, \textbf{ then break}
                        \EndFor
                        \State $\widehat{\bm{\Theta}}_{\,\cdot,J}\gets \bm{\Theta}^{(n+1)}_{\cdot,J}$
                        \EndFor
                        \State \textbf{Return:} $\widehat{\bm{\Theta}}$
                \end{algorithmic}
        \end{breakablealgorithm}

        \begin{breakablealgorithm}
                \caption{Recovery of $\widehat{\bbm{\Omega}}$ and $\{\widehat{\mathbf{A}}_{ik},\,\widetilde{\mathbf{A}}_{ik}\}$ from $\widehat{\bm{\Theta}}$}
                \label{alg:recovery-omega-aik}
                \begin{algorithmic}[1]
                        \State \textbf{Input:} $\{K_i\}_{i=1}^p$, $\widehat{\bm{\Theta}}\in\mathbb{R}^{(pd+1) \times d}$, $\mathbf{D}_N$, $\mathcal{P}(\cdot)$, $\mathcal{H}(\cdot,\cdot)$, $\mathcal{R}(\cdot)$, \textbf{method} $\in\{\textsf{nuclear},\textsf{TE}\}$
                        \State \textbf{Estimate $\bbm{\Omega}$ via PSD projection:} $\widehat{\bbm{\Omega}} \gets \mathcal{P}\!\left(\mathrm{vec}^{-1}\big(\mathbf{D}_N\,\widehat{\bbm{\omega}}\big)\right)$, where $\widehat{\bbm{\omega}} = \widehat{\bm{\Theta}}_{1,\cdot}$
                        \State \textbf{Padding:} For each $i=1,\dots,p$, set $\widehat{\bm{\Phi}}_i \gets \big(\widehat{\bm{\Theta}}_{\mathsf{R}_i,\bm{\cdot}}\big)^{\!\mathrm{T}}\in\mathbb{R}^{d\times d}$ with $\mathsf{R}_i=\{\,1+(i-1)d,\dots,id\,\}$. Then solve the optimization problem in \eqref{eq:nuclear-argmin} or \eqref{eq:te-argmin} to obtain $\widetilde{\mathbf{W}}_i$ or $\widehat{\mathbf{W}}_i$.
                        \State \textbf{Estimate $\mathbf{A}_{ik}$ via eigen-decomposition:} For each $i=1,\dots,p$ and $k=1,\dots,K_i$,
                        \If{\textbf{method} $=\textsf{nuclear}$}
                        \State Let $(\widetilde{\lambda}_{i,k},\widetilde{\bm{u}}_{i,k})$ be the top $K_i$ eigenpairs of $\mathcal{R}\big(\mathcal{H}(\widehat{\bm{\Phi}}_i,\widetilde{\mathbf{W}}_i)\big)$, and set
                        \[
                        \widetilde{\mathbf{A}}_{ik}\gets \mathrm{vec}^{-1}\!\Big(\sqrt{\widetilde{\lambda}_{i,k}^{+}}\,\widetilde{\bm{u}}_{i,k}\Big),
                        \qquad \widetilde{\lambda}_{i,k}^{+}=\max\{\widetilde{\lambda}_{i,k},0\}.
                        \]
                        \ElsIf{\textbf{method} $=\textsf{TE}$}
                        \State Let $(\widehat{\lambda}_{i,k},\widehat{\bm{u}}_{i,k})$ be the top $K_i$ eigenpairs of $\mathcal{R}\big(\mathcal{H}(\widehat{\bm{\Phi}}_i,\widehat{\mathbf{W}}_i)\big)$, and set
                        \[
                        \widehat{\mathbf{A}}_{ik}\gets \mathrm{vec}^{-1}\!\Big(\sqrt{\widehat{\lambda}_{i,k}^{+}}\,\widehat{\bm{u}}_{i,k}\Big),
                        \qquad \widehat{\lambda}_{i,k}^{+}=\max\{\widehat{\lambda}_{i,k},0\}.
                        \]	
                        \EndIf
                        \State \textbf{Return:} $\widehat{\bbm{\Omega}}$ and $\{\widetilde{\mathbf{A}}_{ik}\}$ if \textbf{method}$=\textsf{nuclear}$; $\widehat{\bbm{\Omega}}$ and $\{\widehat{\mathbf{A}}_{ik}\}$ if \textbf{method}$=\textsf{TE}$.
                \end{algorithmic}
        \end{breakablealgorithm}

        \subsection{Model selection}\label{subsec:model order selection}

        This section discusses on selecting the true lag order $p^{*}$ and the true numbers of components $K_{i}^{*}$'s. 
        First, we introduce a robust Bayesian information criterion (BIC) based approach to select the lag order $p^{*}$ for the high-dimensional BEKK model in \eqref{eq:BEKK-ARCH}. The BIC is defined as
        \begin{align}\label{eq:BIC}
                \text{BIC}(p) = \log\mathbb{L}(\widehat{\bm{\Theta}}_p) + \iota_d \left(\frac{\log(pd+1)}{T_{\textnormal{eff}}}\right)^{\frac{1+2\epsilon}{1+\epsilon}}\log(T),
        \end{align}
        where $\widehat{\bm{\Theta}}_p$ is the estimator obtained by fitting model \eqref{eq:loss function} with lag order $p$, $\mathbb{L}(\bm{\Theta}_p) = 1/(2T) \|\mathbf{Y}(\tau) - \mathbf{X}_p(\tau)\bm{\Theta}_p\|_\mathrm{F}^2$ with $\mathbf{X}_p(\tau)$ adapted to the lag order $p$, and $T_{\textnormal{eff}}=T/(\log(T))^2$. Here, $\epsilon>0$ and $\iota_d>0$ are tuning parameters of the penalty. For notational simplicity, the dependence of $\mathbb{L}(\cdot)$ on $p$ is suppressed. 
        The tuning parameter $\epsilon$ is introduced to establish selection consistency under the finite ($4 + 4\epsilon$)-th moment condition on data process in Assumption \ref{assumption:process conditions}(i). Meanwhile, $\iota_d$, which may depend on the dimension $d$, is also required to guarantee the consistency of BIC. 
        In practice, we recommend setting $\epsilon = 0.1$ and $\iota_d = 0.05$, as these values yield satisfactory performance in simulation studies of Section \ref{sec:simulation}.
        The lag order is selected by minimizing BIC over a candidate set of values for $p$, that is, $\widehat{p} = \arg\min_{1 \leq p \leq \overline{p}} \text{BIC}(p)$, where $\overline{p}$ is a pre-specified maximum lag order. A small value such as $\overline{p} = 5$ is typically sufficient in practice. The selection consistency of the proposed BIC is established in Theorem \ref{thm:selection consistency of model order}. 

        Next, we introduce Ridge-type estimators for the true numbers of components $K_i^*$ for $1 \leq i \leq p$. Under Assumption \ref{assumption:restricted parameter space}, the rank of $\mathcal{R}(\mathcal{H}(\bm{\Phi}_i^*,\mathbf{W}_i^*))$ equals $K_i^*$, with $\bm{\Phi}_i^*$ and $\mathbf{W}_i^*$ denoting the true coefficient and auxiliary matrices, respectively. This motivates us to estimate $K_i^*$ by the rank of $\mathcal{R}(\mathcal{H}(\bm{\Phi}_i^*,\mathbf{W}_i^*))$, using the estimates $\widehat{\bm{\Phi}}_i$ in \eqref{eq:loss function} and $\widetilde{\mathbf{W}}_i$ in \eqref{eq:nuclear-argmin}.
        Let $\widetilde{\lambda}_{i,1} \geq \widetilde{\lambda}_{i,2} \geq \cdots \geq \widetilde{\lambda}_{i,N^2} $ be the eigenvalues of $\mathcal{R}(\mathcal{H}(\widehat{\bm{\Phi}}_i,\widetilde{\mathbf{W}}_i))$. We estimate $K_i^*$ via the following Ridge-type estimator \citep{xia2015consistently,wang2024highdimtensor}:
        \begin{align}\label{eq:Ridge-type estimator}
                \widehat{K}_i = \operatorname{argmin}_{1 \leq k \leq \overline{K}} \limits \frac{\widetilde{\lambda}_{i,k+1}+c(N,T)}{\widetilde{\lambda}_{i,k}+c(N,T)}, \text{ for } 1 \leq i \leq p,
        \end{align}
        where $c(N,T)>0$ is a constant depending on $N$ and $T$, and $\overline{K}$ is a pre-specified upper bound satisfying $K_i^* \leq \overline{K}$ for all $i$. 
        The consistency of $\widehat{K}_i$, established in Theorem \ref{thm:selection consistency of K_i}, requires appropriate choices of $\overline{K}$ and $c(N, T)$. We recommend choosing $\overline{K}$ sufficiently large to avoid underestimating $K_i^*$. For $c(N,T)$, the form $\alpha N(Np \log(T)/T_\textnormal{eff})^{\epsilon/(1+\epsilon)}$ with $\alpha=10^{-3}$ and $\epsilon=0.1$ performs well in practice, as demonstrated in Section \ref{sec:simulation}.

        \subsection{Tuning parameter selection}\label{subsec:tuning parameter selection}

        To calculate the regularized LSE $\widehat{\bm{\Theta}}$ defined in \eqref{eq:loss function} via Algorithm \ref{alg:fista-core}, appropriate choices of the regularization parameter $\lambda$ and the truncation parameter $\tau$ are required. Although Theorem \ref{thm:vech upper bound} provides optimal values for these parameters, they depend on unknown quantities. Instead, we adopt a rolling forecasting validation scheme to jointly select $\lambda$ and $\tau$. 

        A two-dimensional grid is constructed for the tuning parameters. The grid for $\lambda$ is scaled adaptively according to the data, while the grid for $\tau$ is set to a data-adaptive interval $[\tau_{\min}, \tau_{\max}]$, where $\tau_{\min}$ and $\tau_{\max}$ are empirical quantiles of the observed data. For example, we may take $\tau_{\min} = \operatorname{median}_{t,j}|y_{t,j}|$ and $\tau_{\max} = \max_{t,j}|y_{t,j}|$.
        For each pair $(\lambda,\tau)$, a one-step-ahead forecasting procedure is implemented using an expanding moving window. The forecasting performance is evaluated via the mean squared forecast error (MSFE):
        \begin{align*}
                \operatorname{MSFE}\left(\lambda,\tau\right) = \frac{1}{T_{1}} \sum_{t = T_{0}}^{T_{0} +T_{1}-1} \| \bm{y}_{t+1} - \widehat{\bm{\Theta}}^{(t)\mathrm{T}}(\lambda, \tau) \bm{x}_{t+1}\|_2^2, 
        \end{align*}
        where $T_{0}$ and $T_{1}$ denote the lengths of the training and validation sets, respectively; $\widehat{\bm{\Theta}}^{(t)}(\lambda, \tau)$ is the regularized LSE estimated from the training data up to time $t$, with $t=T_{0},\ldots,T_{0}+T_{1}-1$.
        Consequently, the values of $\lambda$ and $\tau$ are selected by minimizing $\operatorname{MSFE}(\lambda, \tau)$ over the two-dimensional grid. 

        In addition, we also need to choose the tuning parameter $\gamma_{\mathrm{TE}}$ for Algorithm \ref{alg:recovery-omega-aik} when the padding step is solved via \eqref{eq:te-argmin}. The aforementioned rolling forecasting validation scheme can be similarly used to select $\gamma_{\mathrm{TE}}$. Specifically, for each candidate value of $\gamma_{\mathrm{TE}}$, we conduct a one-step-ahead forecasting procedure for the BEKK conditional covariance matrix using an expanding window. The forecasting performance is evaluated by
        \begin{align*}
                \operatorname{MSFE}(\gamma_{\mathrm{TE}})
                =
                \frac{1}{T_1}
                \sum_{t=T_0}^{T_0+T_1-1}
                \Bigl\|
                \bm r_{t+1}\bm r_{t+1}^\top
                -
                \widehat{\bbm \Omega}
                -
                \sum_{i=1}^{p}\sum_{k=1}^{K_i}
                \widehat{\bm A}_{ik}^{(t)}(\gamma_{\mathrm{TE}})
                \bm r_{t+1-i}\bm r_{t+1-i}^{\mathrm{T}}
                \bigl(\widehat{\bm A}_{ik}^{(t)}(\gamma_{\mathrm{TE}})\bigr)^{\mathrm{T}}
                \Bigr\|_{\mathrm{F}}^2,
        \end{align*}
        where $\widehat{\bbm{\Omega}} = \mathcal{P}(\mathrm{vec}^{-1}(\mathbf{D}_N\,\widehat{\bm{\Theta}}_{1,\cdot}))$, and $\widehat{\bm A}_{ik}^{(t)}(\gamma_{\mathrm{TE}})$ denotes the parameter estimates recovered from \eqref{eq:te-argmin} using observations up to time $t$ under the candidate value $\gamma_{\mathrm{TE}}$. The optimal $\gamma_{\mathrm{TE}}$ is selected by minimizing $\operatorname{MSFE}(\gamma_{\mathrm{TE}})$ over the predefined grid for $\gamma_{\mathrm{TE}}$.

\section{Theoretical Analysis}\label{sec:theory}
        In this section, we derive non-asymptotic error bounds for the regularized LSE of model \eqref{eq:vech linear form}, and demonstrate that its convergence rate is minimax optimal. We further extend the analysis to obtain estimation error bounds for the coefficient matrices of BEKK model \eqref{eq:BEKK-ARCH}. Finally, we establish selection consistency for both the BIC criterion and the Ridge-type selector. 
        Throughout this section, the superscript $*$ is used to denote the true values of the coefficient matrices.

        \subsection{Error bounds for the regularized LSE}\label{subsec:vech model}

        Denote by $\mathcal{S}_j$ the support set of the $j$th column of $\bm{\Theta}^*$, with cardinality $|\mathcal{S}_j| = s_j \leq s$. Under the row-wise sparsity condition of $\{\mathbf A_{ik}\}$ imposed in Section~\ref{sec:model}, the resulting column-wise sparsity level satisfies $s\leq 1+s_A^2\sum_{i=1}^p K_i$.
		To derive the error bounds for the regularized LSE $\widehat{\bm{\Theta}}$ at \eqref{eq:loss function}, we first introduce the following conditions on the process $\{\bm{r}_t\}$.

        \begin{assumption}[Process]\label{assumption:process conditions}
                (i) $\{\bm{r}_t\}$ is strictly stationary, and for some $\epsilon \in (0,1]$, it holds that $\mathbb{E}|r_{t,j}|^{4+4\epsilon} \leq M_{4+4\epsilon}$ for all $1 \leq j \leq N$. (ii) $\{\bm{r}_t\}$ is $\alpha$-mixing with coefficients $\alpha(\ell) = O(\zeta^\ell)$, where $\zeta$ possibly depends on $N$ such that $0 \leq \zeta(N) \leq \bar{\zeta}$ for some constant $\bar{\zeta} < 1$.
        \end{assumption}

        Assumption \ref{assumption:process conditions}(i) relaxes the commonly used sub-Gaussian condition in high-dimensional time series analysis, thereby allowing for potentially heavy-tailed distributions. The $\alpha$-mixing condition in Assumption \ref{assumption:process conditions}(ii) is standard in time series modeling for characterizing temporal dependence. As showed by \cite{boussama2011stationarity} under mild conditions, the process defined by BEKK model \eqref{eq:BEKK-ARCH} is geometrically $\beta$-mixing, which implies that the $\alpha$-mixing condition in Assumption \ref{assumption:process conditions}(ii) holds.

        \begin{theorem}[Upper bounds for regularized LSE]\label{thm:vech upper bound}
                Suppose that Assumption \ref{assumption:process conditions} holds. Denote $T_{\textnormal{eff}}=T/(\log(T))^2$, $M = \max\{M_{4+4\epsilon},1\}$, and $\bm{\Gamma}_x = \mathbb{E}(\bm{x}_t\bm{x}_t^{\mathrm{T}})$. If $T \gtrsim N^{1+1/\epsilon}\log(pd+1)$,
                \begin{align*}
                        \tau \asymp \left(\frac{M T_{\textnormal{eff}}}{\log(pd+1)}\right)^{\frac{1}{4+4\epsilon}} \text{ and } \lambda \asymp s\|\bm{\Theta}^*\|_{1,\infty}\left(\frac{M^{1/\epsilon}\log(pd+1)}{T_{\textnormal{eff}}}\right)^{\frac{\epsilon}{1+\epsilon}},
                \end{align*}
                then, with probability at least $1-C\operatorname{exp}[-Cs\log(T)\log(pd+1)]$, we have 
                \begin{align*}
                        \|\widehat{\bm{\Theta}} - \bm{\Theta}^*\|_{\mathrm{F}} \lesssim s^{3/2}N\|\bm{\Theta}^*\|_{1,\infty}\lambda_{\min}^{-1}(\bm{\Gamma}_x)\left(\frac{M^{1/\epsilon}\log(pd+1)}{T_{\textnormal{eff}}}\right)^{\frac{\epsilon}{1+\epsilon}}.
                \end{align*}
        \end{theorem}

        Theorem \ref{thm:vech upper bound} establishes the non-asymptotic error bounds for the regularized LSE under column-wise sparsity. We have several comments on Theorem \ref{thm:vech upper bound}: First, the quantity $T_{\textnormal{eff}}$ can be interpreted as the effective sample size for the underlying $\alpha$-mixing time series. Notably, $T_{\textnormal{eff}}$ is of nearly the same order as $T$. 
        The sample size condition $T \gtrsim N^{1+1/\epsilon}\log(pd+1)$ stems from the derivation of the Frobenius norm error bound. Due to the column-wise sparsity, the total number of parameters is $O(d)$, and controlling the Frobenius error accordingly necessitates a larger sample size. In contrast, if the objective is solely to bound the $\ell_{2,\infty}$-norm error without deriving the results for estimating coefficient matrices in the original BEKK model \eqref{eq:BEKK-ARCH}, a much milder condition $T \gtrsim \log(pd+1)$ suffices. 
        Second, the boundedness of $\|\bm{\Theta}^*\|_{1,\infty}$ follows naturally from the column-wise sparsity of $\bm{\Theta}^*$, implies by the row-wise sparsity of the coefficient matrices in BEKK model \eqref{eq:BEKK-ARCH}. 
        Third, if $\operatorname{Cov}(\widetilde{\bm{x}}_t)$ is positive definite with $\widetilde{\bm{x}}_t = (\bm{y}_{t-1}^{\mathrm{T}},\ldots,\bm{y}_{t-p}^{\mathrm{T}})^{\mathrm{T}}$, then $\bm{\Gamma}_{x} = \mathbb{E}(\bm{x}_t\bm{x}_t^\mathrm{T})$ is also positive definite, with its smallest eigenvalue bounded away from zero, i.e., $\lambda_{\min}(\bm{\Gamma}_x) >0$.
        Notably, it is standard to assume that $\operatorname{Cov}(\widetilde{\bm{x}}_t)$ is positive definite in literature of covariance-stationary sparse VAR models \citep{wu2016performance,wang2023rate}. 
        Finally, the scaling factor $s^{3/2}\|\bm{\Theta}^*\|_{1,\infty}$ arises from the choice of $\lambda$, which ensures that the estimator for each column lies within a localized cone adapted to the time-series setting, differing from the typical scaling factor $s^{1/2}$ for the standard regression context. When $s$, $\|\bm{\Theta}^*\|_{1,\infty}$, $M_{4+4\epsilon}$ and $\lambda_{\min}(\bm{\Gamma}_{x})$ are fixed, the upper bound scales as $(\log(pd+1)/T_{\textnormal{eff}})^{\epsilon/(1+\epsilon)}$ for $\epsilon \in (0,1)$ and $(\log(pd+1)/T_{\textnormal{eff}})^{1/2}$ for $\epsilon \geq 1$. This result aligns with those in robust high-dimensional linear regression \citep{sun2020adaptive,tan2023sparse}, with the effective sample size $T_{\textnormal{eff}}$ accounting for the temporal dependence. 

        Next, we derive the minimax lower bound for the regularized LSE for general data generating processes under heavy-tailed distributions.  
        Let $\mathsf{P}_y(\epsilon,M',\zeta')$ be the class of all joint distributions for $\alpha$-mixing process $\bm{h}_t = (h_{t,1},\ldots,h_{t,d})^\mathrm{T}$ with $\epsilon \in (0,1]$, $M' > 0$ and $\zeta \in (0,1)$, such that $\max_{i}\mathbb{E}|h_{t,i}|^{2+2\epsilon} = M'$ for $1 \leq i \leq d$ and $\alpha(\ell) \leq (\zeta')^\ell$. Let $\mathbb{P}$ be the distribution of $\{\bm{y}_t\}$, and denote by $\bm{\Theta}^*(\mathbb{P})$ and $\bm{\Gamma}(\mathbb{P})$ the true coefficient matrix and Gram matrix under $\mathbb{P}$, respectively. Under model \eqref{eq:stacked-VAR}, we have the minimax lower bound below.

        \begin{theorem}[Minimax lower bound for regularized LSE]\label{thm:vech minimax lower bound}
                For any $\epsilon \in (0,1]$, $M' > 0$, $\zeta' \in (0,1)$ and $T\big/\log_{\zeta^{-2}}^T \geq 2s$, suppose that the joint distribution $\mathbb{P}$ of $\bm{y}_t$ belongs to $\mathsf{P}_y(\epsilon,M',\zeta')$. Then, for any estimator $\widehat{\bm{\Theta}} = \widehat{\bm{\Theta}}(\{\bm{y}_t\})$, we have
                \begin{align*}
                        \inf_{\widehat{\bm{\Theta}}} \sup_{\substack{\mathbb{P} \in 	\mathsf{P}_y(\epsilon,M',\zeta')\\ \|\bm{\Theta}_{\cdot,j}^*\|_0 \leq s, 1\leq j \leq d}} \mathbb{E}\|\widehat{\bm{\Theta}} - \bm{\Theta}^*(\mathbb{P})\|_{\mathrm{F}}^2 \gtrsim sN^2\lambda^{-2}_{\min}(\bm{\Gamma}(\mathbb{P}))\left(\frac{M'^{1/\epsilon}\log(T)}{T}\right)^{\frac{2\epsilon}{1+\epsilon}}.
                \end{align*}
        \end{theorem}

        Notably, when the sparsity level $s$, $\|\bm\Theta^*\|_{1,\infty}$, $M$, and $\lambda_{\min}(\bm\Gamma_x)$ are treated as fixed constants, the upper bound in Theorem \ref{thm:vech upper bound} nearly matches the lower bound rate in Theorem \ref{thm:vech minimax lower bound} with respect to the sample size $T$. Consequently, the proposed regularized LSE is nearly minimax rate-optimal for the vech form in \eqref{eq:vech linear form}.

        \subsection{Error bounds for estimating BEKK coefficient matrices}\label{subsec:bekk model}
        Under Assumption \ref{assumption:restricted parameter space}, we define the minimum adjacent norm gap at lag $i$ as
        \[
        \vartheta_i \;=\; \min_{1 \leq k \leq K_i} 
        \Big\{\, \|\mathbf{A}_{i(k-1)}^*\|_\mathrm{F} - \|\mathbf{A}_{ik}^*\|_\mathrm{F},\;
        \|\mathbf{A}_{ik}^*\|_\mathrm{F} - \|\mathbf{A}_{i(k+1)}^*\|_\mathrm{F}\,\Big\},
        \]
        with the conventions $\|\mathbf{A}_{i0}^*\|_\mathrm{F}=\infty$ and $\|\mathbf{A}_{i(K_i+1)}^*\|_\mathrm{F}=0$ for $1\le i\le p$. To derive the estimation error bounds for estimating the BEKK-ARCH coefficient matrices $\mathbf{A}_{ik}^*$ and $\bbm{\Omega}^*$, and to establish the selection consistency of each $K_i$ in Theorem \ref{thm:selection consistency of K_i}, we introduce the following assumption.

        \begin{assumption}[Padding error]\label{assump:asym_error}
                The padding error induced by \eqref{eq:nuclear-argmin} and \eqref{eq:te-argmin} is controlled by the first-stage estimation error. Specifically, for each $1 \leq i \leq p$,
                \[
                \|\mathcal{H}(\widehat{\bm{\Phi}}_i,\widetilde{\mathbf{W}}_i)-\mathcal{H}(\bm{\Phi}_i^*,\mathbf{W}_i^*)\|_{\mathrm{F}}
                \vee
                \|\mathcal{H}(\widehat{\bm{\Phi}}_i,\widehat{\mathbf{W}}_i)-\mathcal{H}(\bm{\Phi}_i^*,\mathbf{W}_i^*)\|_{\mathrm{F}}
                \lesssim
                \|\widehat{\bm{\Phi}}_i-\bm{\Phi}_i^*\|_{\mathrm{F}}.
                \]
        \end{assumption}

        Assumption \ref{assump:asym_error} is a stability condition for the padding step. It ensures that the perturbation caused by replacing the true $\bm\Phi_i^*$ with its first-stage estimator $\widehat{\bm\Phi}_i$ is not amplified when solving the padding optimization problems at \eqref{eq:nuclear-argmin} and \eqref{eq:te-argmin}. This condition is naturally satisfied when the padding optimization is restricted to a constrained feasible space, such as a closed convex set $\mathcal W_i$ that contains $\mathbf W_i^*$ and is element-wise bounded. Our simulation results indicate that solving the unconstrained versions of \eqref{eq:nuclear-argmin} and \eqref{eq:te-argmin} does not violate this stability condition in practice. Consequently, we employ the unconstrained optimization in both simulations and empirical applications.

        \begin{corollary}[BEKK-ARCH Upper Bound]\label{cor:bekk upper bound}
                Suppose that Assumptions \ref{assumption:restricted parameter space}--\ref{assump:asym_error} and conditions in Theorem \ref{thm:vech upper bound} hold. Then with probability at least $1-C\operatorname{exp}[-Cs\log(T)\log(pd+1)]$,
                \begin{align*}
                        &\|\widehat{\bbm{\Omega}} - \bbm{\Omega}^*\|_{\mathrm{F}} \lesssim s^{3/2}N\|\bm{\Theta}^*\|_{1,\infty}\lambda_{\min}^{-1}(\bm{\Gamma}_x)\left(\frac{M^{1/\epsilon}\log(pd+1)}{T_{\textnormal{eff}}}\right)^{\frac{\epsilon}{1+\epsilon}},
                \end{align*}
                and for all $1 \leq i \leq p$, $1 \leq k \leq K_i$, $\langle\widetilde{\mathbf{A}}_{ik}, \mathbf{A}_{ik}^*\rangle \geq 0$ and $\langle\widehat{\mathbf{A}}_{ik}, \mathbf{A}_{ik}^*\rangle \geq 0$,
                \begin{align*}
                        & \|\widetilde{\mathbf{A}}_{ik} - \mathbf{A}_{ik}^*\|_{\mathrm{F}} \vee \|\widehat{\mathbf{A}}_{ik} - \mathbf{A}_{ik}^*\|_{\mathrm{F}} \lesssim s^{3/2}N\|\bm{\Theta}^*\|_{1,\infty}\vartheta_i^{-1/2}\lambda_{\min}^{-1}(\bm{\Gamma}_x)\left(\frac{M^{1/\epsilon}\log(pd+1)}{T_{\textnormal{eff}}}\right)^{\frac{\epsilon}{1+\epsilon}}.
                \end{align*}
        \end{corollary}
        Corollary \ref{cor:bekk upper bound} shows that the estimators $\widehat{\bbm{\Omega}}$, $\widehat{\mathbf{A}}_{ik}$ and $\widetilde{\mathbf{A}}_{ik}$ achieve the same convergence rate as $\widehat{\bm{\Theta}}$ in Theorem \ref{thm:vech upper bound}. This indicates that the regularized LSE for the vech transformed model \eqref{eq:vech linear form} preserves estimation accuracy for the original BEKK-ARCH coefficient matrices.

        \subsection{Model selection consistency}\label{subsec:model order selection consistency}

        Let $p^*$ denote the true lag order of the BEKK-ARCH model in \eqref{eq:BEKK-ARCH}. For a candidate order $p$, define $\bm{\Theta}_p^\circ = \operatorname{argmin}_{\|\bm{\Theta}_p\|_0 \leq sd} \mathbb{E}\mathbb{L}(\bm{\Theta}_p)$ with $\mathbb{L}(\bm{\Theta}_p) = 1/(2T) \|\mathbf{Y}(\tau) - \mathbf{X}_p(\tau)\bm{\Theta}_p\|_\mathrm{F}^2$.
        For $1 \leq p \leq p^*$, let $\widetilde{\bm{\Theta}}^\circ_p = (\bm{\Theta}_p^{\circ\mathrm{T}}, \mathbf{0}_{d \times (p^* - p)d})^\mathrm{T} \in \mathbb{R}^{(p^*d+1) \times d}$. We quantify the discrepancy between the true model and a misspecified model of order $p < p^*$ by $\varphi_p = \|\widetilde{\bm{\Theta}}^\circ_p - \bm{\Theta}^*\|_{1,1}$, which captures the signal strength of the misspecification. 
        To ensure the selection consistency of the proposed BIC-type criterion for the lag order $p$, we make the following assumptions.

        \begin{assumption}[Penalty parameter]\label{assump:penalty parameter}
                $\iota_d \asymp s^3M^{1/(1+\epsilon)}$.
        \end{assumption}

        \begin{assumption}[Minimum signal strength]\label{assump:minimum signal strength}
                (i). $\min_{p < p^*} \limits \varphi_p \gg sd\lambda_p\log(T)^{1/2}$, where $\lambda_p \asymp s\|\bm{\Theta}^*\|_{1,\infty}\allowbreak(M^{1/\epsilon}\log(pd+1)/T_{\textnormal{eff}})^{\epsilon/(1+\epsilon)}$ and (ii). $\max_{p \neq p^*} \limits |\mathbb{L}(\widehat{\bm{\Theta}}_p) - \mathbb{L}(\bm{\Theta}_p^\circ)| = O_p(\varphi_p^2/sd)$. 
        \end{assumption}

		Assumption \ref{assump:penalty parameter} provides the theoretical guarantee for the BIC to consistently exclude overspecified models, where $\iota_d$ is a positive constant that may depend on $d$. When both $s$ and $M$ are fixed, $\iota_d$ can be chosen as a positive constant.
        Assumption \ref{assump:minimum signal strength} ensures that the misspecification signal $\varphi_p$ is strong enough to be detected by the BIC. 
        Similar minimum signal and estimation error conditions are commonly used in establishing the selection consistency of BIC-type criteria; see, for example, \cite{zheng2025interpretable}. 
        Under the sample size requirement in Theorem \ref{thm:vech upper bound}, the lower bound in Assumption \ref{assump:minimum signal strength}(i) decays to zero as $T \to \infty$. Meanwhile, Assumption \ref{assump:minimum signal strength}(ii) is a mild regularity condition that ensures any fitted misspecified model converges to the population model at a moderate rate. Here, the misspecified model with parameter matrix $\bm{\Theta}_p^\circ$ can be viewed as the best approximation of the process $\{\bm{y}_t\}$ under misspecification. 
        The following theorem establishes the consistency of the BIC-based order selection procedure.

        \begin{theorem}[Selection consistency of $\widehat{p}$]\label{thm:selection consistency of model order}
                Suppose that $p^* \leq \overline{p}$, Assumptions \ref{assump:penalty parameter}--\ref{assump:minimum signal strength} and conditions in Theorem \ref{thm:vech upper bound} hold. Then $\mathbb{P}(\widehat{p} = p^*) \to 1$ as $N,T \to \infty$.
        \end{theorem}

        Let $\mathcal{K}_p=\{K_1,\ldots,K_p\}$ denote the configuration of component numbers at each lag for a BEKK-ARCH model of order $p$. Furthermore, let $\mathcal{K}_p^*=\{K_1^*,\ldots,K_p^*\}$ be the true configuration and $\widehat{\mathcal{K}}_p=\{\widehat{K}_1,\ldots,\widehat{K}_p\}$ its estimator obtained via the Ridge-type selector in \eqref{eq:Ridge-type estimator}. 
        The following theorem establishes the consistency of the Ridge-type selector.

        \begin{theorem}[Selection consistency of $\widehat{\mathcal{K}}_p$]\label{thm:selection consistency of K_i}
                Suppose that $K_i^*\leq\overline{K}$ for all $1\leq i\leq p$, Assumption \ref{assump:asym_error} and conditions in Theorem \ref{thm:vech upper bound} hold. Let $\varsigma=\min_{1\leq i\leq p}\|\mathbf{A}^*_{iK_i^*}\|_{\mathrm{F}}^{2}$. If
                \[s^{3/2}N\|\bm{\Theta}^*\|_{1,\infty} \lambda_{\min}^{-1}(\bm{\Gamma}_x) \left(\dfrac{M^{1/\epsilon}\log(pd+1)}{T_{\mathrm{eff}}}\right)^{\frac{\epsilon}{1+\epsilon}}\ll c(N,T) \ll \varsigma \min_{\substack{1\leq i\leq p,\,1\leq k\leq K_i^*-1}}\dfrac{\|\mathbf{A}^*_{i(k+1)}\|_{\mathrm{F}}^{2}}{\|\mathbf{A}^*_{ik}\|_{\mathrm{F}}^{2}},\]
                then $\mathbb{P}(\widehat{\mathcal{K}}_p=\mathcal{K}_p^*)\to1$ as $N,T\to\infty$.
        \end{theorem}

        The two-sided condition on $c(N,T)$ in Theorem~\ref{thm:selection consistency of K_i} has a clear interpretation.
        The lower bound enforces that the stochastic perturbation of $\mathcal{R}(\mathcal{H}(\widehat{\bm{\Phi}}_i,\widetilde{\mathbf W}_i))$ or $\mathcal{R}(\mathcal{H}(\widehat{\bm{\Phi}}_i,\widehat{\mathbf W}_i))$ (as controlled by Corollary~\ref{cor:bekk upper bound}) is dominated by $c(N,T)$, so that the Ridge-type ratio $(\widehat{\lambda}_{i,k+1}+c(N,T))/(\widehat{\lambda}_{i,k}+c(N,T))$ is not driven by noise-induced tail eigenvalues.
        The upper bound requires $c(N,T)$ to grow more slowly than the population eigenvalue gap among the leading $K_i^*$ components.
        Intuitively, this condition may be violated if the smallest nonzero component $\|\mathbf A^*_{iK_i^*}\|_{\mathrm F}^2$ is too small, or if there is an abrupt drop in the squared Frobenius norms between adjacent components for some $i$ and $k$.
        In either case, the Ridge-type ratio becomes less informative for identifying $K_i^*$.

\section{Simulation}\label{sec:simulation}
        This section evaluates the finite-sample performance of the proposed estimation and model selection methods. 
        Throughout this section, we generate data $\{\bm{r}_t\}$ from the stationary BEKK-ARCH model in \eqref{eq:BEKK-ARCH}, with the lag order fixed at $p=3$. Three component configurations are considered, namely $\mathcal{K}_3 \in \{\{1,1,1\}, \{2,1,1\}, \{1,2,1\}\}$. The innovation term $\{\bbm{\eta}_t\}$ is drawn from one of the three standardized distributions with mean zero and identity covariance matrix: the standard normal, standard Laplace, or standardized Student-$t$ distribution with $4.2$ degrees of freedom, listed in order from light-tailed to heavy-tailed. We consider two settings with dimensions $N=20$ and $N=100$, where each coefficient matrix is row-wise sparse with at most $s=3$ and $s=10$ nonzeros per row, respectively. 

        Within the BEKK-ARCH specification, the diagonal and nonzero off-diagonal entries of $\bbm{\Omega}$ are generated independently from $U(1,2)$ and $U(-0.1,0.1)$, respectively. For $\mathbf{A}_{ik}$ at each lag $i$, we construct different row support sets with $s$ nonzero entries per row, and arrange the matrices $\mathbf{A}_{ik}$ in descending order of Frobenius norm across all $k$. To preserve temporal dependence of each variable, the diagonal elements of matrices $\mathbf{A}_{ik}$ for each $i$ cannot be zero simultaneously across all $k$. Within the prescribed support sets, diagonals are sampled from $U(0.1,0.5)$, while off-diagonals are sampled from $U(-0.1,0.1)$.
        All subsequent simulation results are based on $100$ replications. The tuning parameters $\lambda$ and $\tau$ are selected via the rolling forecasting validation procedure outlined in Section \ref{subsec:tuning parameter selection}.

        We evaluate estimation and forecasting accuracy using the relative root mean squared error (rRMSE). For an estimated parameter matrix $\widehat{\bbm\Upsilon}$ with true value $\bbm\Upsilon^*$, the rRMSE is defined as 
        \begin{equation}\label{eq:rrmse-theta}
        \mathrm{rRMSE}(\widehat{\bbm\Upsilon}) = \frac{\|\widehat{\bbm\Upsilon}-\bbm\Upsilon^*\|_{\mathrm F}}{\|\bbm\Upsilon^*\|_{\mathrm F}}.
        \end{equation}
        For forecast evaluation, we apply the same metric to the one-step-ahead conditional covariance forecast $\widehat{\bm\Sigma}_{T+1}$ relative to its true value $\bm\Sigma_{T+1}$:
        \begin{equation}\label{eq:rrmse-sigma}
        \mathrm{rRMSE}(\widehat{\bm\Sigma}_{T+1})
        =
        \frac{\|\widehat{\bm\Sigma}_{T+1}-\bm\Sigma_{T+1}\|_{\mathrm F}}
        {\|\bm\Sigma_{T+1}\|_{\mathrm F}}.
        \end{equation}
        In the subsequent tables, we report the sample mean and standard deviation of these rRMSEs computed across all replications.
        To save space, we focus the presentation of estimation and prediction errors on the configuration $\mathcal K_3^*=\{2,1,1\}$. The results for the other two configurations are similar and are provided in the Supplementary Material. For model selection performance, we consider all configurations of $\mathcal K_3^*$ to provide a comprehensive assessment.

        \subsection{Vech BEKK model estimation}\label{subsec:vech bekksim}

        Based on the generated data, we fit the Vech-BEKK-ARCH model in \eqref{eq:vech linear form} and estimate $\bbm{\Theta}^*$ using the proposed regularized LSE $\widehat{\bbm \Theta}$ in \eqref{eq:loss function}, together with its non-truncated counterpart $\widehat{\bbm \Theta}^{\star}$, obtained by solving \eqref{eq:loss function} with $\tau=\infty$, for comparison. Table~\ref{tab:theta_error_combined_K211} reports the corresponding results for $(N,s)=(20,3)$ and $(N,s)=(100,10)$. We have two main findings. First, the truncated estimator $\widehat{\bbm \Theta}$ uniformly achieves smaller estimation errors than the non-truncated estimator $\widehat{\bbm \Theta}^{\star}$ across all distributions and sample sizes, showing the clear efficiency gain owing to truncation. Second, the truncated estimator is also more stable, as reflected by its substantially smaller standard deviations. This robustness advantage is especially pronounced under heavier-tailed innovations, such as Laplace and Student's $t$, and becomes more evident in the higher-dimensional setting $N=100$.

        \begin{table}[H]
        \centering
        \setlength{\tabcolsep}{3pt}
        \renewcommand{\arraystretch}{0.6}
        \caption{Comparison of estimation errors for $\bbm \Theta^*$ under $\mathcal{K}_3=\{2,1,1\}$ for $(N,s)=(20,3)$ and $(N,s)=(100,10)$. Entries are means with standard deviations in parentheses. The best values are shown in bold.}
        \label{tab:theta_error_combined_K211}
        \begin{tabular}{llcccc|cccc}
        \toprule \toprule
        &
        & \multicolumn{4}{c|}{$N=20,\ s=3$}
        & \multicolumn{4}{c}{$N=100,\ s=10$} \\
        \cmidrule(lr){3-6} \cmidrule(lr){7-10}
        Distribution $\backslash$ $T$
        &
        & $600$ & $1200$ & $1800$ & $2400$
        & $900$ & $1800$ & $2700$ & $3600$ \\
        \midrule
        \multirow[c]{2.5}{*}{Gaussian}
        & $\widehat{\bbm \Theta}$
        & \makecell{\textbf{1.859}\\\textbf{(0.016)}}
        & \makecell{\textbf{1.408}\\\textbf{(0.012)}}
        & \makecell{\textbf{1.159}\\\textbf{(0.016)}}
        & \makecell{\textbf{1.014}\\\textbf{(0.014)}}
        & \makecell{\textbf{1.360}\\\textbf{(0.006)}}
        & \makecell{\textbf{1.154}\\\textbf{(0.020)}}
        & \makecell{\textbf{1.038}\\\textbf{(0.006)}}
        & \makecell{\textbf{1.051}\\\textbf{(0.004)}} \\
        &
        $\widehat{\bbm \Theta}^{\star}$
        & \makecell{4.070\\(0.221)}
        & \makecell{2.189\\(0.091)}
        & \makecell{1.581\\(0.068)}
        & \makecell{1.309\\(0.042)}
        & \makecell{2.737\\(0.060)}
        & \makecell{2.558\\(0.080)}
        & \makecell{2.343\\(0.077)}
        & \makecell{2.019\\(0.112)} \\
        \midrule
        \multirow[c]{2.5}{*}{Laplace}
        & $\widehat{\bbm \Theta}$
        & \makecell{\textbf{1.425}\\\textbf{(0.016)}}
        & \makecell{\textbf{1.092}\\\textbf{(0.009)}}
        & \makecell{\textbf{1.018}\\\textbf{(0.008)}}
        & \makecell{\textbf{0.918}\\\textbf{(0.007)}}
        & \makecell{\textbf{1.223}\\\textbf{(0.011)}}
        & \makecell{\textbf{1.071}\\\textbf{(0.009)}}
        & \makecell{\textbf{0.989}\\\textbf{(0.008)}}
        & \makecell{\textbf{0.974}\\\textbf{(0.007)}} \\
        &
        $\widehat{\bbm \Theta}^{\star}$
        & \makecell{6.175\\(0.683)}
        & \makecell{2.879\\(0.331)}
        & \makecell{2.079\\(0.263)}
        & \makecell{1.750\\(0.273)}
        & \makecell{2.772\\(0.130)}
        & \makecell{2.587\\(0.095)}
        & \makecell{2.370\\(0.110)}
        & \makecell{2.004\\(0.136)} \\
        \midrule
        \multirow[c]{2.5}{*}{$t_{4.2}$}
        & $\widehat{\bbm \Theta}$
        & \makecell{\textbf{1.772}\\\textbf{(0.021)}}
        & \makecell{\textbf{1.359}\\\textbf{(0.021)}}
        & \makecell{\textbf{1.128}\\\textbf{(0.016)}}
        & \makecell{\textbf{0.968}\\\textbf{(0.016)}}
        & \makecell{\textbf{1.283}\\\textbf{(0.009)}}
        & \makecell{\textbf{1.173}\\\textbf{(0.007)}}
        & \makecell{\textbf{1.131}\\\textbf{(0.007)}}
        & \makecell{\textbf{1.015}\\\textbf{(0.009)}} \\
        &
        $\widehat{\bbm\Theta}^{\star}$
        & \makecell{6.584\\(1.700)}
        & \makecell{3.444\\(0.774)}
        & \makecell{2.636\\(1.114)}
        & \makecell{2.141\\(0.574)}
        & \makecell{2.904\\(0.171)}
        & \makecell{2.622\\(0.194)}
        & \makecell{2.392\\(0.190)}
        & \makecell{2.005\\(0.341)} \\
        \bottomrule
        \end{tabular}
        \end{table}

        \subsection{Conditional covariance matrix prediction}\label{subsec:sigma sim}
        We next evaluate the prediction accuracy of the conditional covariance matrix $\bm{\Sigma}_t$ using four competing estimators: 
        (i) the vech-based estimator $\widebreve{\bm{\Sigma}}_t$ in \eqref{eq:breve_Sigma_t}; 
        (ii) the nuclear-norm-based estimator $\widetilde{\bm{\Sigma}}_t$ in \eqref{eq:tilde_Sigma_t_and_hat_Sigma_t};
        (iii) the TE-loss-based estimator $\widehat{\bm{\Sigma}}_t$ in \eqref{eq:tilde_Sigma_t_and_hat_Sigma_t}; 
        (iv) the vech-based estimator in \eqref{eq:breve_Sigma_t} without truncation,  
        denoted by $\widebreve{\bm{\Sigma}}_t^{\star}$ with $\widebreve{\bm{\Sigma}}_t^{\star}=\mathcal{P}(\operatorname{vech}^{-1}(\widehat{\bm{\Theta}}^{\mathrm{T}}\bm{x}_t))$, where $\widehat{\bm{\Theta}}^{\star}$ is obtained from \eqref{eq:loss function} with $\tau=\infty$;
        (v) the nuclear-norm-based estimator in \eqref{eq:tilde_Sigma_t_and_hat_Sigma_t} without truncation, denoted by $\widetilde{\bm{\Sigma}}_t^{\star}$ with $\widetilde{\bm{\Sigma}}_t^{\star}\allowbreak=\widehat{\bbm{\Omega}}^{\star} \allowbreak+\sum_{i=1}^{p}\sum_{k=1}^{K_i} \allowbreak\widetilde{\mathbf{A}}_{ik}^{\star}\bm{r}_{t-i}\bm{r}_{t-i}^{\mathrm{T}}\widetilde{\mathbf{A}}_{ik}^{\star\mathrm{T}}$, where $\widehat{\bbm{\Omega}}^{\star} = \mathcal{P}(\operatorname{vech}^{-1}(\mathbf{D}_N\widehat{\bm{\Theta}}^{\star\mathrm{T}}_{\cdot,1}))$, and $\widetilde{\mathbf{A}}_{ik}^{\star}$ are obtained from $\widehat{\bm{\Theta}}^{\star}$ based on the nuclear-norm padding in \eqref{eq:nuclear-argmin}; and 
        (vi) the TE-loss-based estimator in \eqref{eq:tilde_Sigma_t_and_hat_Sigma_t} without truncation, denoted by $\widehat{\bm{\Sigma}}_t^{\star}$ with $\widehat{\bm{\Sigma}}_t^{\star}\allowbreak=\widehat{\bbm{\Omega}}^{\star} \allowbreak+\sum_{i=1}^{p}\sum_{k=1}^{K_i} \allowbreak\widehat{\mathbf{A}}_{ik}^{\star}\bm{r}_{t-i}\bm{r}_{t-i}^{\mathrm{T}}\widehat{\mathbf{A}}_{ik}^{\star\mathrm{T}}$, where  $\widehat{\mathbf{A}}_{ik}^{\star}$ are obtained from $\widehat{\bm{\Theta}}^{\star}$ based on the TE-loss padding in \eqref{eq:te-argmin}.

        Table~\ref{tab:sigma_error_combined_K211} reports the corresponding prediction errors under $\mathcal{K}_3=\{2,1,1\}$ for $(N,s)=(20,3)$ and $(N,s)=(100,10)$ across Gaussian, Laplace, and $t_{4.2}$ innovations. Several patterns emerge clearly. First, all estimators become more accurate as the sample size $T$ increases. Second, the two truncated BEKK-based estimators, $\widetilde{\bm{\Sigma}}_t$ and $\widehat{\bm{\Sigma}}_t$, uniformly outperform their non-truncated counterparts $\widetilde{\bm{\Sigma}}_t^{\star}$ and $\widehat{\bm{\Sigma}}_t^{\star}$, showing that truncation substantially improves estimation accuracy for the conditional covariance matrix. Third, in the moderate-dimensional setting $(N,s)=(20,3)$, the direct vech-based estimator $\widebreve{\bm{\Sigma}}_t$ exhibits very small standard deviations, indicating that it is numerically quite stable without padding-back. However, its estimation errors remain noticeably larger than those of the truncated BEKK-based estimators, so this stability does not translate into better accuracy. By contrast, in the high-dimensional setting $(N,s)=(100,10)$, the vech-based estimators $\widebreve{\bm{\Sigma}}_t$ and $\widebreve{\bm{\Sigma}}_t^{\star}$ deteriorate substantially and are clearly dominated by the padding-based estimators in terms of estimation error. This shows that the direct vech-based reconstruction becomes less effective in higher dimensions, whereas the padding-back step is essential for recovering the covariance structure more accurately. Moreover, among all competitors, the TE-loss-based estimator $\widehat{\bm{\Sigma}}_t$ most often attains the smallest error, while the nuclear-norm-based estimator $\widetilde{\bm{\Sigma}}_t$ is typically the second best. This suggests that the TE-loss-based padding step provides an additional gain in recovering the dynamic covariance structure, possibly because it makes use of the structural information implied by $K_i$.

        \setlength{\tabcolsep}{3pt}
        \renewcommand{\arraystretch}{0.6}

        \begin{longtable}{llcccc|cccc}
        \caption{Comparison of prediction errors for $\bm \Sigma_t$ under $\mathcal{K}_3=\{2,1,1\}$ for $(N,s)=(20,3)$ and $(N,s)=(100,10)$. Entries are mean with standard deviation in parentheses. The best values are shown in bold and the second-best are underlined.}
        \label{tab:sigma_error_combined_K211}\\
        \toprule \toprule
        &
        & \multicolumn{4}{c|}{$N=20,\ s=3$}
        & \multicolumn{4}{c}{$N=100,\ s=10$} \\
        \cmidrule(lr){3-6} \cmidrule(lr){7-10}
        Distribution\ \textbackslash \ $T$
        &
        & $600$ & $1200$ & $1800$ & $2400$
        & $900$ & $1800$ & $2700$ & $3600$ \\
        \midrule
        \endfirsthead

        \toprule \toprule
        &
        & \multicolumn{4}{c|}{$N=20,\ s=3$}
        & \multicolumn{4}{c}{$N=100,\ s=10$} \\
        \cmidrule(lr){3-6} \cmidrule(lr){7-10}
        Distribution\ \textbackslash \ $T$
        &
        & $600$ & $1200$ & $1800$ & $2400$
        & $900$ & $1800$ & $2700$ & $3600$ \\
        \midrule
        \endhead

        \midrule
        \multicolumn{10}{r}{Continued on next page}
        \endfoot

        \bottomrule
        \endlastfoot
        \multirow[c]{11}{*}{Gaussian}
        & $\widebreve{\boldsymbol{\Sigma}}_t$
        & \makecell{1.087\\(0.011)}
        & \makecell{0.873\\\textbf{(0.003)}}
        & \makecell{0.807\\\textbf{(0.003)}}
        & \makecell{0.773\\\textbf{(0.003)}}
        & \makecell{3.013\\(0.759)}
        & \makecell{2.843\\(0.217)}
        & \makecell{2.806\\(0.203)}
        & \makecell{2.762\\(0.161)} \\
        &
        $\widetilde{\boldsymbol{\Sigma}}_t$
        & \makecell{\underline{0.978}\\\textbf{(0.008)}}
        & \makecell{\underline{0.883}\\\underline{(0.017)}}
        & \makecell{0.764\\(0.016)}
        & \makecell{0.678\\(0.013)}
        & \makecell{1.424\\(0.426)}
        & \makecell{1.092\\(0.119)}
        & \makecell{1.011\\(0.093)}
        & \makecell{1.003\\(0.058)} \\
        &
        $\widehat{\boldsymbol{\Sigma}}_t$
        & \makecell{\textbf{0.977}\\\underline{(0.009)}}
        & \makecell{\textbf{0.800}\\(0.019)}
        & \makecell{\textbf{0.662}\\\underline{(0.014)}}
        & \makecell{\textbf{0.570}\\\underline{(0.011)}}
        & \makecell{\textbf{0.878}\\(0.048)}
        & \makecell{\textbf{0.812}\\\underline{(0.031)}}
        & \makecell{\textbf{0.777}\\\underline{(0.022)}}
        & \makecell{\textbf{0.753}\\\underline{(0.018)}} \\
        &
        $\widebreve{\boldsymbol{\Sigma}}_t^{\star}$
        & \makecell{2.595\\(0.034)}
        & \makecell{1.841\\(0.029)}
        & \makecell{1.501\\(0.027)}
        & \makecell{1.312\\(0.023)}
        & \makecell{3.424\\\textbf{(0.017)}}
        & \makecell{3.334\\\textbf{(0.013)}}
        & \makecell{3.254\\\textbf{(0.011)}}
        & \makecell{3.191\\\textbf{(0.009)}} \\
        &
        $\widetilde{\boldsymbol{\Sigma}}_t^{\star}$
        & \makecell{1.690\\(0.273)}
        & \makecell{1.096\\(0.125)}
        & \makecell{0.800\\(0.116)}
        & \makecell{0.654\\(0.076)}
        & \makecell{2.222\\(0.193)}
        & \makecell{1.986\\(0.103)}
        & \makecell{1.693\\(0.073)}
        & \makecell{1.398\\(0.046)} \\
        &
        $\widehat{\boldsymbol{\Sigma}}_t^{\star}$
        & \makecell{1.699\\(0.269)}
        & \makecell{1.096\\(0.126)}
        & \makecell{\underline{0.750}\\(0.131)}
        & \makecell{\underline{0.593}\\(0.091)}
        & \makecell{\underline{1.021}\\\underline{(0.042)}}
        & \makecell{\underline{0.990}\\(0.035)}
        & \makecell{\underline{0.949}\\(0.025)}
        & \makecell{\underline{0.934}\\(0.024)} \\
        \midrule
        \multirow[c]{9}{*}{Laplace}
        & $\widebreve{\boldsymbol{\Sigma}}_t$
        & \makecell{1.884\\\textbf{(0.003)}}
        & \makecell{1.844\\\textbf{(0.002)}}
        & \makecell{1.821\\\textbf{(0.003)}}
        & \makecell{1.814\\\textbf{(0.002)}}
        & \makecell{2.897\\(0.398)}
        & \makecell{2.786\\(0.279)}
        & \makecell{2.706\\(0.208)}
        & \makecell{2.675\\(0.199)} \\
        &
        $\widetilde{\boldsymbol{\Sigma}}_t$
        & \makecell{\underline{0.913}\\\underline{(0.009)}}
        & \makecell{\underline{0.804}\\\underline{(0.015)}}
        & \makecell{\underline{0.719}\\(0.014)}
        & \makecell{\underline{0.662}\\(0.016)}
        & \makecell{1.390\\(0.515)}
        & \makecell{1.074\\(0.119)}
        & \makecell{\underline{0.990}\\(0.093)}
        & \makecell{1.105\\(0.129)} \\
        &
        $\widehat{\boldsymbol{\Sigma}}_t$
        & \makecell{\textbf{0.903}\\(0.016)}
        & \makecell{\textbf{0.700}\\(0.016)}
        & \makecell{\textbf{0.608}\\\underline{(0.012)}}
        & \makecell{\textbf{0.566}\\\underline{(0.010)}}
        & \makecell{\textbf{0.826}\\(0.054)}
        & \makecell{\textbf{0.757}\\\underline{(0.030)}}
        & \makecell{\textbf{0.714}\\\underline{(0.019)}}
        & \makecell{\textbf{0.698}\\\underline{(0.021)}} \\
        &
        $\widebreve{\boldsymbol{\Sigma}}_t^{\star}$
        & \makecell{2.766\\(0.072)}
        & \makecell{2.015\\(0.083)}
        & \makecell{1.677\\(0.074)}
        & \makecell{1.497\\(0.109)}
        & \makecell{3.438\\\textbf{(0.024)}}
        & \makecell{3.340\\\textbf{(0.028)}}
        & \makecell{3.270\\\textbf{(0.009)}}
        & \makecell{3.204\\\textbf{(0.012)}} \\
        &
        $\widetilde{\boldsymbol{\Sigma}}_t^{\star}$
        & \makecell{2.489\\(0.569)}
        & \makecell{1.460\\(0.501)}
        & \makecell{1.184\\(0.371)}
        & \makecell{1.071\\(0.429)}
        & \makecell{2.196\\(0.284)}
        & \makecell{2.019\\(0.302)}
        & \makecell{1.770\\(0.141)}
        & \makecell{1.510\\(0.061)} \\
        &
        $\widehat{\boldsymbol{\Sigma}}_t^{\star}$
        & \makecell{2.495\\(0.561)}
        & \makecell{1.464\\(0.499)}
        & \makecell{1.168\\(0.377)}
        & \makecell{1.041\\(0.445)}
        & \makecell{\underline{1.075}\\\underline{(0.042)}}
        & \makecell{\underline{1.023}\\(0.033)}
        & \makecell{0.999\\(0.043)}
        & \makecell{\underline{0.971}\\(0.035)} \\
        \midrule
        \multirow[c]{9}{*}{$t_{4.2}$}
        & $\widebreve{\boldsymbol{\Sigma}}_t$
        & \makecell{1.955\\\textbf{(0.009)}}
        & \makecell{1.887\\\textbf{(0.004)}}
        & \makecell{1.829\\\textbf{(0.003)}}
        & \makecell{1.772\\\textbf{(0.003)}}
        & \makecell{2.871\\(1.134)}
        & \makecell{2.629\\(0.259)}
        & \makecell{2.591\\(0.258)}
        & \makecell{2.575\\(0.186)} \\
        &
        $\widetilde{\boldsymbol{\Sigma}}_t$
        & \makecell{\underline{0.966}\\\underline{(0.009)}}
        & \makecell{\underline{0.828}\\\underline{(0.019)}}
        & \makecell{\underline{0.705}\\(0.018)}
        & \makecell{\underline{0.619}\\(0.011)}
        & \makecell{1.382\\(0.386)}
        & \makecell{\underline{1.123}\\(0.182)}
        & \makecell{\underline{1.107}\\(0.259)}
        & \makecell{\underline{1.100}\\(0.114)} \\
        &
        $\widehat{\boldsymbol{\Sigma}}_t$
        & \makecell{\textbf{0.964}\\(0.010)}
        & \makecell{\textbf{0.747}\\(0.022)}
        & \makecell{\textbf{0.607}\\\underline{(0.015)}}
        & \makecell{\textbf{0.538}\\\underline{(0.010)}}
        & \makecell{\textbf{0.824}\\(0.069)}
        & \makecell{\textbf{0.753}\\\textbf{(0.038)}}
        & \makecell{\textbf{0.710}\\(0.034)}
        & \makecell{\textbf{0.689}\\\underline{(0.021)}} \\
        &
        $\widebreve{\boldsymbol{\Sigma}}_t^{\star}$
        & \makecell{2.667\\(0.125)}
        & \makecell{2.112\\(0.196)}
        & \makecell{1.843\\(0.355)}
        & \makecell{1.626\\(0.214)}
        & \makecell{3.479\\\textbf{(0.066)}}
        & \makecell{3.376\\\underline{(0.039)}}
        & \makecell{3.288\\\textbf{(0.019)}}
        & \makecell{3.219\\(0.025)} \\
        &
        $\widetilde{\boldsymbol{\Sigma}}_t^{\star}$
        & \makecell{3.504\\(1.618)}
        & \makecell{2.108\\(0.889)}
        & \makecell{1.839\\(1.240)}
        & \makecell{1.485\\(0.703)}
        & \makecell{2.301\\(1.328)}
        & \makecell{2.174\\(0.575)}
        & \makecell{1.741\\(0.235)}
        & \makecell{1.580\\(0.224)} \\
        &
        $\widehat{\boldsymbol{\Sigma}}_t^{\star}$
        & \makecell{3.501\\(1.609)}
        & \makecell{2.106\\(0.886)}
        & \makecell{1.833\\(1.237)}
        & \makecell{1.478\\(0.706)}
        & \makecell{\underline{1.309}\\\underline{(0.067)}}
        & \makecell{1.186\\(0.052)}
        & \makecell{1.179\\\underline{(0.021)}}
        & \makecell{1.155\\\textbf{(0.012)}} \\
        \bottomrule
        \end{longtable}

        \subsection{BEKK model estimation}\label{subsec:bekk sim}
        We now further compare the estimation performance for the representative coefficient component $\mathbf{A}_{11}^*$ using both the nuclear-norm-based and TE-loss-based procedures, with and without truncation. 
        Table~\ref{tab:a11_error_combined_K211} reports the corresponding Frobenius-norm errors under $\mathcal{K}_3^*=\{2,1,1\}$ for $(N,s)=(20,3)$ and $(N,s)=(100,10)$ across Gaussian, Laplace, and Student's $t$ innovations. Several patterns are clear from the table. First, for both padding schemes, the truncated estimators $\widetilde{\mathbf{A}}_{11}$ and $\widehat{\mathbf{A}}_{11}$ generally achieve smaller estimation errors than their non-truncated counterparts $\widetilde{\mathbf{A}}_{11}^{\star}$ and $\widehat{\mathbf{A}}_{11}^{\star}$, showing that truncation improves estimation accuracy. Second, the truncated estimators are also substantially more stable, as reflected by the smaller standard deviations in most cases, which indicates stronger robustness to heavy-tailed innovations and sample variability. Third, among the four competitors, the TE-loss-based truncated estimator $\widehat{\mathbf{A}}_{11}$ most often attains the smallest error, and its advantage becomes more pronounced when the dimension increases and the innovations become heavier-tailed.

        \begin{table}[H]
        \setlength{\tabcolsep}{3pt}
        \centering
        \renewcommand{\arraystretch}{0.6}
        \caption{Comparison of estimation errors for $\bm A_{11}^*$ under $\mathcal{K}_3=\{2,1,1\}$ for $(N,s)=(20,3)$ and $(N,s)=(100,10)$. Entries are mean with standard deviation in parentheses. The best values are shown in bold and the second-best are underlined.}
        \label{tab:a11_error_combined_K211}
        \begin{tabular}{llcccc|cccc}
        \toprule \toprule
        &
        & \multicolumn{4}{c|}{$N=20,\ s=3$}
        & \multicolumn{4}{c}{$N=100,\ s=10$} \\
        \cmidrule(lr){3-6} \cmidrule(lr){7-10}
        Distribution\ \textbackslash \ $T$
        &
        & $600$ & $1200$ & $1800$ & $2400$
        & $900$ & $1800$ & $2700$ & $3600$ \\
        \midrule
        \multirow[c]{6}{*}{Gaussian}
        & $\widetilde{\mathbf{A}}_{11}$
        & \makecell{\underline{1.281}\\\underline{(0.045)}}
        & \makecell{\underline{1.050}\\\underline{(0.112)}}
        & \makecell{\underline{0.794}\\\underline{(0.077)}}
        & \makecell{\underline{0.668}\\\underline{(0.081)}}
        & \makecell{1.487\\\textbf{(0.033)}}
        & \makecell{1.325\\\underline{(0.046)}}
        & \makecell{1.241\\(0.068)}
        & \makecell{1.172\\\underline{(0.036)}} \\
        &
        $\widehat{\mathbf{A}}_{11}$
        & \makecell{\textbf{1.280}\\\textbf{(0.045)}}
        & \makecell{\textbf{0.863}\\(0.138)}
        & \makecell{\textbf{0.609}\\\textbf{(0.071)}}
        & \makecell{\textbf{0.535}\\\textbf{(0.069)}}
        & \makecell{\textbf{0.918}\\(0.087)}
        & \makecell{\textbf{0.760}\\(0.063)}
        & \makecell{\textbf{0.707}\\\textbf{(0.022)}}
        & \makecell{\textbf{0.688}\\\textbf{(0.013)}} \\
        &
        $\widetilde{\mathbf{A}}_{11}^{\star}$
        & \makecell{1.720\\(0.090)}
        & \makecell{1.365\\\textbf{(0.102)}}
        & \makecell{1.074\\(0.164)}
        & \makecell{0.896\\(0.164)}
        & \makecell{1.660\\(0.105)}
        & \makecell{1.580\\\textbf{(0.055)}}
        & \makecell{1.443\\\underline{(0.050)}}
        & \makecell{1.288\\(0.045)} \\
        &
        $\widehat{\mathbf{A}}_{11}^{\star}$
        & \makecell{1.734\\(0.088)}
        & \makecell{1.353\\(0.128)}
        & \makecell{0.966\\(0.202)}
        & \makecell{0.806\\(0.187)}
        & \makecell{\underline{0.965}\\\underline{(0.078)}}
        & \makecell{\underline{0.887}\\(0.077)}
        & \makecell{\underline{0.882}\\(0.056)}
        & \makecell{\underline{0.870}\\(0.050)} \\
        \midrule
        \multirow[c]{6}{*}{Laplace}
        & $\widetilde{\mathbf{A}}_{11}$
        & \makecell{\underline{1.174}\\\textbf{(0.060)}}
        & \makecell{\underline{0.864}\\\underline{(0.111)}}
        & \makecell{\underline{0.694}\\\underline{(0.068)}}
        & \makecell{\underline{0.599}\\\underline{(0.053)}}
        & \makecell{1.416\\\textbf{(0.040)}}
        & \makecell{1.301\\\underline{(0.058)}}
        & \makecell{1.222\\\underline{(0.071)}}
        & \makecell{1.160\\\underline{(0.047)}} \\
        &
        $\widehat{\mathbf{A}}_{11}$
        & \makecell{\textbf{1.150}\\\underline{(0.087)}}
        & \makecell{\textbf{0.651}\\\textbf{(0.087)}}
        & \makecell{\textbf{0.546}\\\textbf{(0.063)}}
        & \makecell{\textbf{0.484}\\\textbf{(0.047)}}
        & \makecell{\textbf{0.875}\\(0.120)}
        & \makecell{\textbf{0.721}\\\textbf{(0.037)}}
        & \makecell{\textbf{0.678}\\\textbf{(0.018)}}
        & \makecell{\textbf{0.660}\\\textbf{(0.024)}} \\
        &
        $\widetilde{\mathbf{A}}_{11}^{\star}$
        & \makecell{2.020\\(0.142)}
        & \makecell{1.530\\(0.185)}
        & \makecell{1.353\\(0.231)}
        & \makecell{1.263\\(0.297)}
        & \makecell{1.682\\(0.135)}
        & \makecell{1.616\\(0.146)}
        & \makecell{1.464\\(0.099)}
        & \makecell{1.342\\(0.054)} \\
        &
        $\widehat{\mathbf{A}}_{11}^{\star}$
        & \makecell{2.026\\(0.139)}
        & \makecell{1.538\\(0.183)}
        & \makecell{1.323\\(0.262)}
        & \makecell{1.214\\(0.342)}
        & \makecell{\underline{1.064}\\\underline{(0.089)}}
        & \makecell{\underline{0.943}\\(0.083)}
        & \makecell{\underline{0.949}\\(0.115)}
        & \makecell{\underline{0.912}\\(0.094)} \\
        \midrule
        \multirow[c]{6}{*}{$t_{4.2}$}
        & $\widetilde{\mathbf{A}}_{11}$
        & \makecell{\textbf{1.240}\\\textbf{(0.065)}}
        & \makecell{\underline{0.958}\\\textbf{(0.113)}}
        & \makecell{\underline{0.762}\\\underline{(0.087)}}
        & \makecell{\underline{0.635}\\\underline{(0.073)}}
        & \makecell{1.442\\\textbf{(0.033)}}
        & \makecell{1.309\\\textbf{(0.048)}}
        & \makecell{1.249\\\textbf{(0.035)}}
        & \makecell{1.194\\\textbf{(0.036)}} \\
        &
        $\widehat{\mathbf{A}}_{11}$
        & \makecell{\underline{1.243}\\\underline{(0.075)}}
        & \makecell{\textbf{0.790}\\\underline{(0.125)}}
        & \makecell{\textbf{0.625}\\\textbf{(0.083)}}
        & \makecell{\textbf{0.538}\\\textbf{(0.061)}}
        & \makecell{\textbf{0.905}\\\underline{(0.121)}}
        & \makecell{\textbf{0.751}\\\underline{(0.086)}}
        & \makecell{\textbf{0.693}\\\underline{(0.042)}}
        & \makecell{\textbf{0.674}\\\underline{(0.037)}} \\
        &
        $\widetilde{\mathbf{A}}_{11}^{\star}$
        & \makecell{2.312\\(0.459)}
        & \makecell{1.840\\(0.340)}
        & \makecell{1.702\\(0.460)}
        & \makecell{1.522\\(0.367)}
        & \makecell{1.684\\(0.356)}
        & \makecell{1.678\\(0.164)}
        & \makecell{1.558\\(0.124)}
        & \makecell{1.466\\(0.116)} \\
        &
        $\widehat{\mathbf{A}}_{11}^{\star}$
        & \makecell{2.311\\(0.459)}
        & \makecell{1.839\\(0.348)}
        & \makecell{1.699\\(0.460)}
        & \makecell{1.510\\(0.391)}
        & \makecell{\underline{1.220}\\(0.399)}
        & \makecell{\underline{1.170}\\(0.240)}
        & \makecell{\underline{1.136}\\(0.146)}
        & \makecell{\underline{1.181}\\(0.144)} \\
        \bottomrule
        \end{tabular}
        \end{table}

        \subsection{Model selection}\label{subsec:model order selection sim}

        In this section, we assess the proposed model selection procedures in Section \ref{subsec:model order selection}. Specifically, the lag order $p^*$ is determined via the BIC-based criterion in \eqref{eq:BIC}, while the component configuration $K_i^*$ for each $1 \leq i \leq p$ is selected using the Ridge-type estimator in \eqref{eq:Ridge-type estimator}. We set the maximum lag order to $\overline{p}=5$ and the maximum component number to $\overline{K}=5$. Tables \ref{tab:selection_consistency_for_p} and \ref{tab:selection_consistency_for_K} summarize the percentages of correct selections for $p^*$ and $K_i^*$ for the settings $(N,s)=(20,3)$ and $(N,s)=(100,10)$, respectively. The results demonstrate that both selection procedures perform well, with correct selection rates improving as the sample size $T$ increases. Notably, even under heavy-tailed innovations, the procedures achieve satisfactory performance when the sample size is sufficiently large.

        \begin{table}[H]
                \centering
                \caption{Percentages of correct selection for the lag order $p^*$ by BIC-based procedure \eqref{eq:BIC}.}
                \label{tab:selection_consistency_for_p}
                \renewcommand{\arraystretch}{0.6}
                \setlength{\tabcolsep}{6pt}
                \makebox[\textwidth][c]{
                        \begin{tabular}{c ccc ccc ccc}
                                \toprule \toprule
                                $N=20$\multirow{2}{*} & \multicolumn{3}{c}{$\mathcal{K}_3^* = \{1, 1, 1\}$} & \multicolumn{3}{c}{$\mathcal{K}_3^* = \{2,1,1\}$} & \multicolumn{3}{c}{$\mathcal{K}_3^* = \{1,2,1\}$}\\
                                \cmidrule(lr{4pt}){2-4} \cmidrule(lr{4pt}){5-7} \cmidrule(lr{4pt}){8-10}
                                $T\backslash \bbm{\eta}_t$ & Gaussian & Laplace & $t_{4.2}$ & Gaussian & Laplace & $t_{4.2}$ & Gaussian & Laplace & $t_{4.2}$\\
                                \midrule
                                1200 & 84 & 70 & 59 & 81 & 65 & 67 & 69 & 51 & 54\\
                                1800 & 96 & 98 & 90 & 95 & 96 & 85 & 95 & 96 & 91\\
                                2400 & 99 & 100 & 97  & 99 & 99 & 96 & 100 & 99 & 97\\
                                \midrule
                                $N=100$\multirow{2}{*} & \multicolumn{3}{c}{$\mathcal{K}_3^* = \{1, 1, 1\}$} & \multicolumn{3}{c}{$\mathcal{K}_3^* = \{2,1,1\}$} & \multicolumn{3}{c}{$\mathcal{K}_3^* = \{1,2,1\}$}\\
                                \cmidrule(lr{4pt}){2-4} \cmidrule(lr{4pt}){5-7} \cmidrule(lr{4pt}){8-10}
                                $T\backslash \bbm{\eta}_t$ & Gaussian & Laplace & $t_{4.2}$ & Gaussian & Laplace & $t_{4.2}$ & Gaussian & Laplace & $t_{4.2}$\\
                                \midrule
                                1800 & 41 & 38 & 23 & 51 & 48 & 37 & 61 & 52 & 50\\
                                2700 & 99  & 97 & 72 & 99 & 91 & 88 & 93 & 94 & 93\\
                                3600 & 100  & 100 & 100 & 99 & 100 & 98 & 100 & 100 & 100\\
                                \bottomrule
                        \end{tabular}
                }
        \end{table}

        \begin{table}[H]
                \centering
                \caption{Percentages of correct selection for the component number $\mathcal{K}_3^*$ by the Ridge-type estimator \eqref{eq:Ridge-type estimator}.}
                \label{tab:selection_consistency_for_K}
                \renewcommand{\arraystretch}{0.6}
                \setlength{\tabcolsep}{6pt}
                \makebox[\textwidth][c]{
                        \begin{tabular}{c ccc ccc ccc}
                                \toprule \toprule
                                $N=20$\multirow{2}{*} & \multicolumn{3}{c}{$\mathcal{K}_3^* = \{1, 1, 1\}$} & \multicolumn{3}{c}{$\mathcal{K}_3^* = \{2,1,1\}$} & \multicolumn{3}{c}{$\mathcal{K}_3^* = \{1,2,1\}$}\\
                                \cmidrule(lr{4pt}){2-4} \cmidrule(lr{4pt}){5-7} \cmidrule(lr{4pt}){8-10}
                                $T\backslash \bbm{\eta}_t$ & Gaussian & Laplace & $t_{4.2}$ & Gaussian & Laplace & $t_{4.2}$ & Gaussian & Laplace & $t_{4.2}$\\
                                \midrule
                                1200 & 100 & 100 & 91 & 83 & 76 & 48 & 92 & 69 & 52\\
                                1800 & 100 & 100 & 97 & 96 & 93 & 76 & 99 & 81 & 72\\
                                2400 & 100 & 100 & 99 & 99 & 97 & 89 & 100 & 94 & 83\\
                                \midrule
                                $N=100$\multirow{2}{*} & \multicolumn{3}{c}{$\mathcal{K}_3^* = \{1, 1, 1\}$} & \multicolumn{3}{c}{$\mathcal{K}_3^* = \{2,1,1\}$} & \multicolumn{3}{c}{$\mathcal{K}_3^* = \{1,2,1\}$}\\
                                \cmidrule(lr{4pt}){2-4} \cmidrule(lr{4pt}){5-7} \cmidrule(lr{4pt}){8-10}
                                $T\backslash \bbm{\eta}_t$ & Gaussian & Laplace & $t_{4.2}$ & Gaussian & Laplace & $t_{4.2}$ & Gaussian & Laplace & $t_{4.2}$\\
                                \midrule
                                1800 & 97 & 81 & 65 & 92 & 77 & 81 & 93 & 84 & 78\\
                                2700 & 99 & 98 & 77 & 99 & 89 & 95 & 98 & 98 & 96\\
                                3600 & 100 & 100 & 99 & 100 & 100 & 100 & 100 & 99 & 99\\
                                \bottomrule
                        \end{tabular}
                }
        \end{table}

\section{Real Data Analysis}\label{sec:real data analysis}

        In this section, we consider two empirical applications to illustrate the performance of our proposed methodology. We analyze two financial datasets, the 17 industry portfolios and the 100 portfolios formed on size and investiment, downloaded from Kenneth French's data library (\url{https://mba.tuck.dartmouth.edu/pages/faculty/ken.french/data_library.html}). The two datasets are standard workhorses in empirical portfolio studies \citep{Behr20121414}. It is standard practice to construct aggregated portfolios based on these constituent assets \citep{Qiao2023220}. 
        Both datasets exhibit volatility clustering and heavy-tailedness across all components, which motivates us to analyze these two datasets using the BEKK-ARCH model in \eqref{eq:BEKK-ARCH}, and employ the proposed estimation methods in Section~\ref{sec:methodology} to obtain their conditional covariance matrices for portfolio construction. 
        The model is selected via the procedures in Section~\ref{subsec:model order selection} and all tuning parameters are selected via the rolling forecasting validation procedure described in Section~\ref{subsec:tuning parameter selection}. 
        All computations are implemented in Python with GPU acceleration (NVIDIA RTX 3090).

        \subsection{17 industry portfolios}\label{subsec:17 industry portfolios}

        We analyze the centered daily returns for 17 industry portfolios from January 2, 2020, to June 30, 2025,
        yielding a time series $\{\bm{r}_{t}\}$ with dimension $N = 17$ and sample size $T = 1380$.

        We fit $\{\bm{r}_{t}\}$ with BEKK-ARCH model in \eqref{eq:BEKK-ARCH}. Given the maximum lag order $\overline{p}=10$ and the maximum component number $\overline{K}=5$, the proposed BIC in \eqref{eq:BIC} and Ridge-type estimator in \eqref{eq:Ridge-type estimator} select $\widehat{p}=3$ and $\widehat{\mathcal{K}}_3=\{2,2,1\}$, respectively. Here, we consider six estimators for the conditional covariance matrix $\bm{\Sigma}_t$ based on our method: (i) the vech-based estimator $\widebreve{\bm{\Sigma}}_t$ in \eqref{eq:breve_Sigma_t}; 
        (ii) the nuclear-norm-based estimator $\widetilde{\bm{\Sigma}}_t$ in \eqref{eq:tilde_Sigma_t_and_hat_Sigma_t};
        (iii) the TE-loss-based estimator $\widehat{\bm{\Sigma}}_t$ in \eqref{eq:tilde_Sigma_t_and_hat_Sigma_t}; 
        (iv) the vech-based estimator in \eqref{eq:breve_Sigma_t} without truncation,  
        denoted by $\widebreve{\bm{\Sigma}}_t^{\star}$;
        (v) the nuclear-norm-based estimator in \eqref{eq:tilde_Sigma_t_and_hat_Sigma_t} without truncation, denoted by $\widetilde{\bm{\Sigma}}_t^{\star}$; and 
        (vi) the TE-loss-based estimator in \eqref{eq:tilde_Sigma_t_and_hat_Sigma_t} without truncation, denoted by $\widehat{\bm{\Sigma}}_t^{\star}$.

        We next assess forecasting performance through minimum-variance (MV) portfolio construction. Let $z_t=\bm{w}_t^{\mathrm{T}}\bm{r}_t$ denote the MV portfolio return, where the weight vector is $\bm{w}_t=(\mathbf{1}_N^\mathrm{T}\bm{\Sigma}_t^{-1}\mathbf{1}_N)^{-1}\bm{\Sigma}_t^{-1}\mathbf{1}_N$ with $\bm{\Sigma}_t=\operatorname{var}(\bm{r}_t|\mathcal{F}_{t-1})$.
        Replacing the unknown $\bm{\Sigma}_t$ by its estimators yields the corresponding portfolio returns $\widehat{z}_t$, $\widetilde{z}_t$, $\widebreve{z}_t$, $\widebreve{z}_t^{\star}$, $\widetilde{z}_t^{\star}$, and $\widehat{z}_t^{\star}$. To evaluate out-of-sample performance, we implement a one-step-ahead rolling forecasting procedure with an expanding window, and reserve the last 20\% of the sample for evaluation. Specifically, we fit the BEKK model in \eqref{eq:BEKK-ARCH} with $\widehat{p}=3$ and $\widehat{\mathcal{K}}_3=\{2,2,1\}$ using observations up to $T_0=1103$, and obtain the one-day-ahead forecast at $t_1=T_0+1=1104$. We then move the forecast origin forward by one day, re-estimate the model, and repeat this procedure until the end of the sample, yielding $T-T_0=277$ one-day-ahead MV portfolio forecasts.

        We further compare our methods with several counterparts. In total, eight portfolios are considered: 
        (1) $\mathbf{BEKK}$, the MV portfolio $\widebreve{z}_t$ based on $\widebreve{\bm{\Sigma}}_t$; 
        (2) $\mathbf{BEKK}_{\text{Nuc}}$, the MV portfolio $\widetilde{z}_t$ based on $\widetilde{\bm{\Sigma}}_t$; 
        (3) $\mathbf{BEKK}_{\text{TE}}$, the MV portfolio $\widehat{z}_t$ based on $\widehat{\bm{\Sigma}}_t$; 
        (4) $\mathbf{BEKK}^{\text{NT}}$, the MV portfolio $\widebreve{z}_t^{\star}$ based on $\widebreve{\bm{\Sigma}}_t^{\star}$; 
        (5) $\mathbf{BEKK}_{\text{Nuc}}^{\text{NT}}$, the MV portfolio $\widetilde{z}_t^{\star}$ based on $\widetilde{\bm{\Sigma}}_t^{\star}$; 
        (6) $\mathbf{BEKK}_{\text{TE}}^{\text{NT}}$, the MV portfolio $\widehat{z}_t^{\star}$ based on $\widehat{\bm{\Sigma}}_t^{\star}$;
        (7) $\bm{1/N}$, the naive equal-weighted portfolio; 
        (8) \textbf{CCC}, the MV portfolio based on $\bm{\Sigma}_t=\mathbf{D}_t\mathbf{R}\mathbf{D}_t$, where $\mathbf{D}_t$ is diagonal with each entry fitted by a univariate GARCH$(1,1)$ model via QMLE, and $\mathbf{R}$ is a constant conditional correlation matrix estimated by nonlinear shrinkage \citep{engle2019large}; and 
        (9) \textbf{DCC}, the MV portfolio based on $\bm{\Sigma}_t=\mathbf{D}_t\mathbf{R}_t\mathbf{D}_t$, where $\mathbf{D}_t$ is defined as above and $\mathbf{R}_t=(1-\alpha-\beta)\mathbf{C}+\alpha\,\bm{r}_{t-1}\bm{r}_{t-1}^{\mathrm{T}}+\beta\,\mathbf{R}_{t-1}$, with $\mathbf{C}$ estimated by nonlinear shrinkage \citep{engle2019large} and $\alpha,\beta$ estimated by QMLE.

        To evaluate the out-of-sample performance of each portfolio, we consider three metrics: (i) annualized mean return (AV), defined as the sample mean of portfolio returns $\{z_t\}$ multiplied by $252$; (ii) annualized standard deviation (SD), defined as the sample standard deviation of $\{z_t\}$  multiplied by $252$; (iii) annualized information ratio (IR), computed as the ratio of AV to SD. Notably, the Sharpe ratio reduces to IR when the risk-free rate is assumed to be zero.

        Table~\ref{tab:17 industry portfolios} summarizes the out-of-sample performance and computation time of all competing methods. We have the following findings. First, compared to the naive $1/N$ benchmark, all BEKK-based methods, as well as CCC and DCC, achieve substantially higher annualized returns and lower portfolio risk, leading to marked improvements in the information ratio. Second, while the standard BEKK estimator achieves performance comparable to CCC and DCC, it requires the lowest computational cost among all model-based competitors. Third, the proposed truncated padding-based estimators, $\textnormal{BEKK}_{\text{Nuc}}$ and $\textnormal{BEKK}_{\text{TE}}$, further improve upon BEKK and other benchmarks in terms of return and information ratio, highlighting the benefit of the padding-back step for recovering a more informative covariance structure. Fourth, comparing truncated and non-truncated versions, both $\textnormal{BEKK}$ and $\textnormal{BEKK}_{\text{Nuc}}$ dominate their counterparts $\textnormal{BEKK}^{\text{NT}}$ and $\textnormal{BEKK}_{\text{Nuc}}^{\text{NT}}$, indicating that suitable truncation improves portfolio performance under heavy-tailed data. Finally, in this moderate-dimensional setting, $\textnormal{BEKK}_{\text{Nuc}}$ slightly outperforms $\textnormal{BEKK}_{\text{TE}}$, suggesting that the additional structural information exploited by the TE-loss padding does not yet translate into a clear empirical gain when the eigengap in the nuclear-norm-based padding is already sufficiently pronounced. Overall, these results show that the proposed padding-based BEKK procedures can substantially improve portfolio performance while remaining computationally competitive.

        \begin{table}[H]
        \centering
        \caption{Out-of-sample performance and average computation time for 17 industry portfolios. The best methods in each evaluation metric are marked in bold, and the second-best methods are underlined.}
        \label{tab:17 industry portfolios}
        \renewcommand{\arraystretch}{0.6}
        \setlength{\tabcolsep}{4pt}
        \makebox[\textwidth][c]{
        \begin{tabular}{c ccccccccc}
        \toprule \toprule
        & $\mathbf{BEKK}$ & $\mathbf{BEKK}_{\text{Nuc}}$ & $\mathbf{BEKK}_{\text{TE}}$ & $\mathbf{BEKK}^{\text{NT}}$ & $\mathbf{BEKK}_{\text{Nuc}}^{\text{NT}}$ & $\mathbf{BEKK}_{\text{TE}}^{\text{NT}}$ & $1/N$ & CCC & DCC\\
        \midrule
        AV(\%)    & 11.21 & \textbf{14.61} & \underline{14.04} & 11.45 & 11.89 & 11.63 & 11.45 & 13.84 & 13.85\\
        SD(\%)    & \textbf{12.03} & 14.41 & \underline{12.67} & 12.69 & 12.84 & 12.73 & 17.40 & 15.18 & 15.18\\
        IR        & 0.93  & \underline{1.03} & \textbf{1.10} & 0.90  & 0.93 & 0.91 & 0.66 & 0.91 & 0.91\\
        Time($s$) & \textbf{0.01} & 0.69 & 0.43 & \underline{0.01} & 0.74 & 0.59 & -- & 0.17 & 0.96\\
        \bottomrule
        \end{tabular}
        }
        \end{table}

\subsection{100 portfolios}\label{subsec:100 portfolios}
        In this section, we analyze daily returns of $100$ portfolios formed on size and investment from January 4, 2010, to June 30, 2025. As in Section \ref{subsec:17 industry portfolios}, we calculate the centered time series $\{\bm{r}_t\}$ with $N=100$ and $T=3726$. We fit $\{\bm{r}_{t}\}$ with BEKK-ARCH model in \eqref{eq:BEKK-ARCH}. Given the maximum lag order set to $\overline{p}=10$, and the maximum component number set to $\overline{K}=5$, our proposed BIC and Ridge-type estimator select $\widehat{p}=3$ and $\widehat{\mathcal{K}}_3=\{1,1,1\}$, respectively. Following Section \ref{subsec:17 industry portfolios}, we construct eight portfolios using our methods and existing competitors. To compare their performance, we also conduct one-step-ahead rolling forecasting with an expanding window, reserving the last $10\%$ of observations as the test period (starting at $T_0 = 3354$). Three metrics including AV, SD and IR are calculated for evaluation.

        Table~\ref{tab:100 portfolios} reports the out-of-sample performance and computation time of all competing approaches in the high-dimensional case. The results reveal several clear patterns. First, the proposed BEKK-based methods consistently outperform the benchmarks ($1/N$, CCC, and DCC) in terms of annualized return and information ratio.
        Notably, $\textnormal{BEKK}_{\text{TE}}$ attains the highest information ratio among all methods, while also achieving one of the highest annualized returns and relatively low portfolio risk. 
        Second, the BEKK estimator remains the fastest among all model-based competitors. Moreover, between the two padding-based approaches, the TE-loss-based methods, $\textnormal{BEKK}_{\text{TE}}$ and $\textnormal{BEKK}_{\text{TE}}^{\text{NT}}$, are substantially less time-consuming than their nuclear-norm-based counterparts, $\textnormal{BEKK}_{\text{Nuc}}$ and $\textnormal{BEKK}_{\text{Nuc}}^{\text{NT}}$, highlighting the computational advantage of the TE loss in high dimensions. 
        Third, comparing the truncated and non-truncated versions, both $\textnormal{BEKK}$, $\textnormal{BEKK}_{\text{Nuc}}$ and $\textnormal{BEKK}_{\text{TE}}$ outperform their untruncated counterparts $\textnormal{BEKK}^{\text{NT}}$, $\textnormal{BEKK}_{\text{Nuc}}^{\text{NT}}$ and $\textnormal{BEKK}_{\text{TE}}^{\text{NT}}$ in terms of information ratio, which again highlights the practical value of truncation under heavy-tailed data. By contrast, the non-truncated nuclear-norm-based version does not lead to an improvement, indicating that without truncation the additional flexibility may amplify estimation noise in high dimensions. Overall, the results suggest that, in the high-dimensional setting, $\textnormal{BEKK}_{\text{TE}}$ provides the best overall balance between return, risk, and predictive accuracy, while the BEKK estimator remains an attractive choice when computational cost is the primary concern.

        \begin{table}[H]
        \centering
        \caption{Out-of-sample performance and average computation time for 100 portfolios. The best methods in each evaluation metric are marked in bold, and the second-best methods are underlined.}
        \label{tab:100 portfolios}
        \renewcommand{\arraystretch}{0.6}
        \setlength{\tabcolsep}{4pt}
        \makebox[\textwidth][c]{
        \begin{tabular}{c ccccccccc}
        \toprule \toprule
        & $\mathbf{BEKK}$ & $\mathbf{BEKK}_{\text{Nuc}}$ & $\mathbf{BEKK}_{\text{TE}}$ & $\mathbf{BEKK}^{\text{NT}}$ & $\mathbf{BEKK}_{\text{Nuc}}^{\text{NT}}$ & $\mathbf{BEKK}_{\text{TE}}^{\text{NT}}$ & $1/N$ & CCC & DCC\\
        \midrule
        AV(\%)    & 16.24 & 16.65 & \underline{16.94} & 16.35 & \textbf{17.15} & 17.01 & 13.37 & 13.78 & 10.70\\
        SD(\%)    & 15.38 & 15.97 & \underline{14.93} & 15.76 & 16.89 & 16.23 & 16.32 & 19.02 & \textbf{13.21}\\
        IR        & \underline{1.06} & 1.04 & \textbf{1.13} & 1.04 & 1.01 & 1.05 & 0.82 & 0.72 & 0.81\\
        Time($s$) & \textbf{0.14} & 232.64 & 20.03 & \underline{0.20} & 249.39 & 23.16 & -- & 1.31 & 18.30\\
        \bottomrule
        \end{tabular}
        }
        \end{table}

\section{Conclusion and Discussion}\label{sec:conclusion and discussion}
        This paper proposes a robust and scalable estimation framework for high-dimensional BEKK-ARCH models with heavy-tailed distributions. By imposing row-wise sparsity in the coefficient matrices of the BEKK-ARCH model, the proposed method employs data truncation to handle heavy-tailedness and utilizes an equivalent VAR representation to introduce a regularized least squares estimation (LSE) for efficient computation. We establish that the resulting regularized LSE achieves minimax optimality. 
        Building on the vech-based estimation step, we further develop a padding-back procedure to recover the original BEKK coefficient matrices. In particular, we propose both nuclear-norm-based and TE-loss-based padding schemes, which not only provide stable recovery of the bilinear BEKK structure but also extend naturally to other models with similar bilinear parameterizations.
        Furthermore, we propose a robust BIC for selecting the model order and a Ridge-type estimator for determining the number of BEKK components, establishing their selection consistency under heavy-tailed settings. 
        Finally, the practical value of our approach is illustrated through two financial datasets of 17 industry portfolios and 100 portfolios, respectively. Empirical results confirm that our method outperforms existing alternatives in both computational speed and forecasting accuracy.

        The proposed approach can be extended and improved in the following directions. First, the robust estimation framework developed for BEKK-ARCH models could be adapted to other multivariate volatility models, such as the DCC class, under heavy-tailed distributions.
        Second, our theoretical analysis currently relies on a finite $(4+4\epsilon)$-th moment condition on data process. Relaxing this requirement would broaden the applicability of the method to more extreme heavy-tailed settings.
        Third, while we impose a row-wise sparsity structure on the coefficient matrices to manage high-dimensionality, it would be valuable to weaken this assumption to a weak sparsity condition, thereby accommodating more flexible model structures.
        We leave these topics for future research.

\setlength{\bibsep}{1pt}
\bibliography{mybib.bib}

\startsupplement[Supplementary material for ``A robust and scalable estimation for high-dimensional volatility models'']
\begin{abstract}
This supplementary material provides technical details, proofs and additional simulation results supporting the main paper. Section \ref{append:notation and auxiliary definition} introduces the essential notation and an auxiliary definition. Section \ref{append:Primary lemmas} presents primary lemmas that form the theoretical foundation for our main results. Section \ref{append:proofs of main results} contains proofs of all theorems and propositions. Section \ref{append:proofs of primary lemmas} provides detailed proofs of the primary lemmas from Section \ref{append:Primary lemmas}. Section \ref{append:technical lemmas} collects technical lemmas from the literature. Section \ref{append:Specification} specifies the construction of auxiliary matrices and operators. Finally, Section \ref{append:additional simulation results} presents additional simulation results.
\end{abstract}

\vspace{5mm}
	\section{Notation and an Auxiliary Definition}\label{append:notation and auxiliary definition}
	\renewcommand{\thedefinition}{A.\arabic{definition}}
	\setcounter{definition}{0}
	We begin by establishing notations used throughout the supplement. Scalars are denoted by roman letters (e.g., $a,\tau$), vectors by bold lowercase (e.g., $\bm{u},\bm{\eta}$), and matrices by bold uppercase letters (e.g., $\mathbf{A},\bm{\Theta}$). For a vector $\bm{u}\in\mathbb{R}^N$ and $q\in[0,\infty]$, define the $\ell_q$-norm as $\|\bm{u}\|_q=(\sum_{i=1}^N |u_i|^q)^{1/q}$ (with $\|\bm{u}\|_\infty=\max_i |u_i|$) and the sparsity count $\|\bm{u}\|_0=\sum_i \mathds{1}(u_i\neq 0)$. For a matrix $\mathbf{A}=(a_{ij})\in\mathbb{R}^{N_1\times N_2}$, let $\operatorname{vec}(\mathbf{A})$ be the column-stacking vectorization and $\operatorname{vech}(\mathbf{A})$ the half-vectorization. Their inverse operations are denoted by $\operatorname{vec}^{-1}(\cdot)$ and $\operatorname{vech}^{-1}(\cdot)$, respectively. For matrix entries, we use $\mathbf{M}[i,j]$ to denote the element at position $(i,j)$. The $j$-th column of $\mathbf{A}$ and the submatrix consisting of columns indexed by a set $J$ are denoted by $\mathbf{A}_{\cdot,j}$ and $\mathbf{A}_{\cdot,J}$, respectively. We define the Frobenius inner product as $\langle \mathbf{A},\mathbf{B}\rangle=\mathrm{tr}(\mathbf{A}^\mathrm{T} \mathbf{B})$, and the Frobenius norm $\|\mathbf{A}\|_{\mathrm{F}}=\langle \mathbf{A},\mathbf{A}\rangle^{1/2}$. The operator norm is given by $\|\mathbf{A}\|_{\mathrm{op}}=\sigma_1(\mathbf{A})$, and the nuclear norm by $\|\mathbf{A}\|_*=\sum_i \sigma_i(\mathbf{A})$, where $\sigma_i(\mathbf{A})$ denotes the $i$-th singular value of $\mathbf{A}$. The elementwise max norm is defined as $\|\mathbf{A}\|_{\infty,\infty}=\max_{i,j}|a_{ij}|$. For grouped (columnwise) norms, we write $\|\mathbf{A}\|_{p,q}=\big(\sum_{j=1}^{N_2}\|\mathbf{A}_{\boldsymbol{\cdot},j}\|_p^{\,q}\big)^{1/q}$. The Kronecker product is denoted by $\otimes$. Let $C$ represents a generic positive constant that may vary across different contexts. For sequences $\{a_n\}$ and $\{b_n\}$, we write $a_n\gtrsim b_n$ if $a_n\ge C b_n$ for all $n$, and $a_n\asymp b_n$ if both $a_n\gtrsim b_n$ and $b_n\gtrsim a_n$ hold. 

	The empirical quadratic loss is given by $\mathbb{L}(\bm{\Theta})=1/(2T)\|\mathbf{Y}(\tau)-\mathbf{X}(\tau)\bm{\Theta}\|_\mathrm{F}^2$. Denote
	\[
		\nabla\mathbb{L}(\bm{\Theta})=\frac{1}{T}\mathbf{X}(\tau)^\mathrm{T}\left\{\mathbf{X}(\tau)\bm{\Theta}-\mathbf{Y}(\tau)\right\}\in\mathbb{R}^{(1+pd)\times d}
	\]
	as the gradient of the loss with respect to \(\bm{\Theta}\). For \(1\leq j\leq d\), let
	\[
		\nabla^2\mathbb{L}(\bm{\Theta}_{\cdot,j})=\frac{1}{T}\mathbf{X}(\tau)^\mathrm{T}\mathbf{X}(\tau)\in\mathbb{R}^{(1+pd)\times(1+pd)}
	\]
	be the Hessian matrix of \(\mathbb{L}(\bm{\Theta})\) with respect to $\bm{\Theta}_{\cdot,j}$ for $1 \leq j \leq d$.
	Let $\mathcal{F}_{t_1}^{t_2}$ be the $\sigma$-algebra generated by $\{\bm{r}_{t_1},\dots,\bm{r}_{t_2}\}$. The $\alpha$-mixing coefficient is defined as
	\[
		\alpha(\ell) = \sup_{A \in \mathcal{F}_{-\infty}^{t}, B \in \mathcal{F}_{t+\ell}^{\infty}} |\mathbb{P}(A \cap B) - \mathbb{P}(A)\mathbb{P}(B)|.
	\]

	To mitigate the singularity issues of the Hessian matrix in high-dimensional settings, we study the property in a localized sense. Similar to \cite{sun2020adaptive}, we introduce the localized restricted eigenvalue of $\nabla^2 \mathbb{L}(\bm{\Theta}_{\bm{\cdot},j})$ as follows.

	\begin{definition}[Localized Restricted Eigenvalue]\label{def:localized restricted eigenvalue}
		The minimum localized restricted eigenvalue for $\nabla^2 \mathbb{L}(\bm{\Theta}_{\bm{\cdot},j})$ is defined as
		\begin{align*}
			\nu(\nabla^2 \mathbb{L}(\bm{\Theta}_{\bm{\cdot},j}),\mathcal{S}_j, \gamma) = \inf_{\mathbf{u}}\left\{\frac{\mathbf{u}^{\mathrm{T}} \nabla^2 \mathbb{L}(\bm{\Theta}_{\bm{\cdot},j}) \mathbf{u}}{\|\mathbf{u}\|_2^2}: \mathbf{u} \in \mathcal{C}(\mathcal{S}_j, \gamma)\right\},
		\end{align*}
		where $\mathcal{C}(\mathcal{S}_j, \gamma)=\{\mathbf{u} \in \mathbb{R}^{pd+1} : \mathbf{u} \neq \mathbf{0}, \|\mathbf{u}_{\mathcal{S}_j^c}\|_1 \leq \gamma \|\mathbf{u}_{\mathcal{S}_j}\|_1\}$.
	\end{definition}
	
	\section{Primary Lemmas}\label{append:Primary lemmas}
	\renewcommand{\thelemma}{B.\arabic{lemma}}
	\setcounter{lemma}{0}
	The proof for Theorem \ref{thm:vech upper bound} is based on Lemma \ref{lemma:first order derivative} -- \ref{lemma:restricted eigenvalue condition}. Specifically, Lemma \ref{lemma:first order derivative} establishes an upper bound on the first-order derivative of the loss function, and its proof is based on Lemma \ref{lemma:covariance matrix bound} which provides an upper bound for the empirical covariance matrix $\mathbf{S}=\frac{1}{T}\mathbf{X}(\tau)^{\mathrm{T}}\mathbf{X}(\tau)$ of the transformed vector $\boldsymbol{x}_t(\tau)$ in terms of matrix $\boldsymbol{\Gamma}_x = \mathbb{E}(\boldsymbol{x}_t\boldsymbol{x}_t^\mathrm{T})$, where $\mathbf{X}(\tau)=(\bbm{x}_1(\tau),\dots,\bbm{x}_T(\tau))^{\mathrm{T}}$. Lemma \ref{lemma:ell1 cone} shows that each column of $\widehat{\bm{\Theta}}$ falls into the $\ell_1$ cone under some regular conditions. Lemma~\ref{lemma:frobenius upper bound for symmetric bergman divergence} provides an upper bound of $\langle\nabla \mathbb{L}(\widehat{\bm{\Theta}}_{\bm{\cdot},j}) - \nabla\mathbb{L}(\bm{\Theta}_{\bm{\cdot},j}^*), \widehat{\bm{\Theta}}_{\bm{\cdot},j} - \bm{\Theta}_{\bm{\cdot},j}^*\rangle$ by the estimation error, scaled by $\sqrt{s}\lambda$. Lemma \ref{lemma:localized restricted eigenvalues} ensures that the minimum localized restricted eigenvalue is bounded away from zero with high probability.
    Let $M = \max\{M_{4+4\epsilon}, 1\}$. 

		\begin{lemma}\label{lemma:covariance matrix bound}
		Suppose Assumption \ref{assumption:process conditions} hold. If $T \gtrsim \log(pd+1)$ and $$\tau \asymp \biggl(\frac{M T_{\textnormal{eff}}}{\log(pd+1)}\biggr)^{\frac{1}{4+4\epsilon}},$$ then, for $\epsilon \in (0,1]$, with probability at least $1-C\operatorname{exp}[-C\log(T)\log(pd+1)]$,
		\begin{align*}
			\|\mathbf{S} - \boldsymbol{\Gamma}_x\|_{\infty,\infty} \lesssim \biggl(\frac{M^{1/\epsilon} \log(pd+1)}{T_{\textnormal{eff}}}\biggr)^{\frac{\epsilon}{1+\epsilon}}.
		\end{align*}
		\end{lemma}

		\begin{lemma}\label{lemma:first order derivative}
			Suppose Assumption \ref{assumption:process conditions} hold. If $T \gtrsim \log(pd+1)$ and $$\tau \asymp \left(\frac{M T_{\textnormal{eff}}}{\log(pd+1)}\right)^{\frac{1}{4+4\epsilon}},$$ then, for $\epsilon \in (0,1]$, with probability at least $1-C\operatorname{exp}[-Cs\log(T)\log(pd+1)]$, 
			\begin{align*}
				\|\nabla\mathbb{L}(\bm{\Theta}^*)\|_{\infty,\infty} \lesssim s\|\bm{\Theta}^*\|_{1,\infty}\left(\frac{M^{1/\epsilon}\log(pd+1)}{T_{\textnormal{eff}}}\right)^{\frac{\epsilon}{1+\epsilon}}.
			\end{align*}
		\end{lemma}

		\begin{lemma}[$\ell_1$-Cone]\label{lemma:ell1 cone}
			Assume that $\|\nabla\mathbb{L}(\bm{\Theta}^*)\|_{\infty,\infty} \leq \lambda/2$ for any $\lambda>0$. Let $\widehat{\bm{\Theta}}$ be a solution to \eqref{eq:loss function} and $\mathcal{S}_j$ be the support set of $\bm{\Theta}_{\bm{\cdot},j}^*$. Then, for each $1 \leq j \leq d$, $\widehat{\bm{\Theta}}_{\bm{\cdot},j}$ lies in the following $\ell_1$ cone:
			\begin{align*}
				\|(\widehat{\bm{\Theta}}_{\bm{\cdot},j} - \bm{\Theta}^*_{\bm{\cdot},j} )_{\mathcal{S}_j^c}\|_1 \leq 3\|(\widehat{\bm{\Theta}}_{\bm{\cdot},j} - \bm{\Theta}^*_{\bm{\cdot},j} )_{\mathcal{S}_j}\|_1.
			\end{align*}
		\end{lemma}

		\begin{lemma}\label{lemma:frobenius upper bound for symmetric bergman divergence}
			Under the same conditions as in Lemma \ref{lemma:ell1 cone}, we have, for each $1 \leq j \leq d$,
			\begin{align*}
				\langle\nabla \mathbb{L}(\widehat{\bm{\Theta}}_{\bm{\cdot},j}) - \nabla\mathbb{L}(\bm{\Theta}_{\bm{\cdot},j}^*), \widehat{\bm{\Theta}}_{\bm{\cdot},j} - \bm{\Theta}_{\bm{\cdot},j}^*\rangle \leq 6\lambda\sqrt{s}\|\widehat{\bm{\Theta}}_{\bm{\cdot},j} - \bm{\Theta}_{\bm{\cdot},j}^*\|_2.
			\end{align*}
		\end{lemma}

		\begin{lemma}\label{lemma:localized restricted eigenvalues}
			Suppose Assumption \ref{assumption:process conditions} hold. If $T \gtrsim \log(pd+1)$ and  $$\tau \asymp \left(\frac{M T_{\textnormal{eff}}}{\log(pd+1)}\right)^{\frac{1}{4+4\epsilon}},$$ then, for $1 \leq j \leq d$, with probability at least $1-C\operatorname{exp}[-C\log(T)\log(pd+1)]$, the minimum localized restricted eigenvalue of $\nabla^2 \mathbb{L}(\bm{\Theta}_{\bm{\cdot},j})$ satisfies $$\nu(\nabla^2 \mathbb{L}(\bm{\Theta}_{\bm{\cdot},j}),\mathcal{S}_j, \gamma) \geq \frac12\lambda_{\min}(\boldsymbol{\Gamma}_x).$$
		\end{lemma}

		\begin{lemma}[Restricted Eigenvalue Condition]\label{lemma:restricted eigenvalue condition}
			Assume that $\|\nabla\mathbb{L}(\bm{\Theta}^*)\|_{\infty,\infty} \leq \lambda/2$ for any $\lambda>0$. Under the same conditions as in Lemma \ref{lemma:localized restricted eigenvalues}, for any $\bm{\Theta}_{\bm{\cdot},j} - \bm{\Theta}_{\bm{\cdot},j}^* \in \mathcal{C}(\mathcal{S}_j,\gamma)$, with probability at least $1-C\operatorname{exp}\bigl[-C\log(T)\log(pd+1)\bigr]$,
			\begin{align*}
				\langle\nabla \mathbb{L}(\widehat{\bm{\Theta}}_{\bm{\cdot},j}) - \nabla\mathbb{L}(\bm{\Theta}_{\bm{\cdot},j}^*), \widehat{\bm{\Theta}}_{\bm{\cdot},j} - \bm{\Theta}_{\bm{\cdot},j}^*\rangle \geq \frac{1}{2}\lambda_{\textnormal{min}}(\boldsymbol{\Gamma}_x)\|\bm{\Theta}_{\bm{\cdot},j} - \bm{\Theta}_{\bm{\cdot},j}^*\|_2^2.
			\end{align*}
		\end{lemma}

	\section{Proofs of Main Results}\label{append:proofs of main results}
	\renewcommand{\thedefinition}{C.\arabic{definition}}
	\setcounter{definition}{0}
	\renewcommand{\theequation}{C.\arabic{equation}}
	\setcounter{lemma}{0}
	\renewcommand{\thelemma}{C.\arabic{lemma}}
	\setcounter{equation}{0}

	\begin{proof}[\textbf{Proof of Proposition \ref{prop:equivalence of orthohognol forms}}]
	By the vectorization identity $\mathrm{vec}(\mathbf{B}\mathbf{M}\mathbf{B}^\mathrm{T})=(\mathbf{B} \otimes \mathbf{B})\mathrm{vec}(\mathbf{M})$, we have
	\begin{align}\label{eq:vec-step}
	\mathrm{vec}\left(\sum_{k=1}^{K_i^B}\mathbf{B}_{ik}\mathbf{M}\mathbf{B}_{ik}^{\mathrm{T}}\right)
	\;=\;
	\left(\sum_{k=1}^{K_i^B}\mathbf{B}_{ik}\otimes \mathbf{B}_{ik}\right)\mathrm{vec}(\mathbf{M}).
	\end{align}
	Note that $\mathcal{R}(\cdot)$ is the rearrangement operator satisfying
	$\mathcal{R}(\mathbf{U}\otimes \mathbf{U})=\mathrm{vec}(\mathbf{U})\mathrm{vec}(\mathbf{U})^\mathrm{T}$, and it is a linear bijection. Applying $\mathcal{R}(\cdot)$ to the sum of Kronecker products on the right-hand side of \eqref{eq:vec-step}, we have
	\begin{align}\label{eq:permutation-operation}
	\mathcal{R}\left(\sum_{k=1}^{K_i^B}\mathbf{B}_{ik}\otimes \mathbf{B}_{ik}\right)
	\;=\;
	\sum_{k=1}^{K_i^B}\mathrm{vec}(\mathbf{B}_{ik})\,\mathrm{vec}(\mathbf{B}_{ik})^\mathrm{T}
	\;=:\;\mathbf{G}.
	\end{align}
	Therefore, the matrix $\mathbf{G}$ is positive semi-definite with $\operatorname{rank}(\mathbf{G})=K_i^A\le K_i^B$.
	By the spectral decomposition, we have
	\begin{align}\label{eq:spectral}
	\mathbf{G}=\sum_{k=1}^{K_i^A}\lambda_{i,k}\bm{u}_{i,k}\bm{u}_{i,k}^\mathrm{T},
	\quad \text{with}\;
	\lambda_{i,k}>0, \{\bm{u}_{i,k}\}_{k=1}^{K_i^A}\ \text{orthonormal}.
	\end{align}
	Let $\mathbf{A}_{ik} = \sqrt{\lambda_{i,k}}\,\mathrm{vec}^{-1}(\bm{u}_{i,k})$ for $k=1,\ldots,K_i^A$. Then, we have, for $j\neq l$, $\langle \mathbf{A}_{ij},\mathbf{A}_{il}\rangle = \sqrt{\lambda_{i,j}\lambda_{i,l}}\bm{u}_{i,j}^\mathrm{T} \bm{u}_{i,l}=0$ and 
	\begin{align*}
		\mathcal{R}\left(\sum_{k=1}^{K_i^B}\mathbf{B}_{ik}\otimes \mathbf{B}_{ik}\right) = \mathcal{R}\left(\sum_{k=1}^{K_i^A}\mathbf{A}_{ik}\otimes \mathbf{A}_{ik}\right).
	\end{align*}
	Since $\mathcal{R}(\cdot)$ is bijective, $\mathcal{R}^{-1}(\cdot)$ exists and is also bijective. Applying $\mathcal{R}^{-1}(\cdot)$ on the above equation and multiplying both sides by $\mathrm{vec}(\mathbf{M})$ yields
	\begin{align}\label{eq:kronecker-equality}
		\left(\sum_{k=1}^{K_i^B}\mathbf{B}_{ik}\otimes \mathbf{B}_{ik}\right)\mathrm{vec}(\mathbf{M})
		\;=\;
		\left(\sum_{k=1}^{K_i^A}\mathbf{A}_{ik}\otimes \mathbf{A}_{ik}\right)\mathrm{vec}(\mathbf{M}).
	\end{align}
	Applying the vectorization identity to \eqref{eq:kronecker-equality} and then unvectorizing both sides, we have
	\begin{equation*}
		\sum_{k=1}^{K_i^B}\mathbf{B}_{ik}\mathbf{M}\mathbf{B}^\mathrm{T}_{ik}=\sum_{k=1}^{K_i^A}\mathbf{A}_{ik}\mathbf{M}\mathbf{A}^\mathrm{T}_{ik}. \hfill \qedhere
	\end{equation*}
	\end{proof}

	\begin{proof}[\textbf{Proof of Theorem \ref{thm:vech upper bound}}]
		Under the assumptions in Theorem \ref{thm:vech upper bound}, Lemma \ref{lemma:first order derivative} implies that with probability at least $1-C\exp[-Cs\log(T)\log(pd+1)]$, the gradient satisfies $\|\nabla\mathbb{L}(\bm{\Theta}^*)\|_{\infty,\infty} \lesssim \lambda$. Then by Lemma \ref{lemma:ell1 cone}, for each $1 \leq j \leq d$, $\widehat{\bm{\Theta}}_{\bm{\cdot},j}$ lies in the following $\ell_1$ cone:
		\begin{align*}
			\|(\widehat{\bm{\Theta}}_{\bm{\cdot},j} - \bm{\Theta}_{\bm{\cdot},j}^*)_{\mathcal{S}_j^c}\|_1 \leq 3\|(\widehat{\bm{\Theta}}_{\bm{\cdot},j} - \bm{\Theta}_{\bm{\cdot},j}^*)_{\mathcal{S}_j}\|_1,
		\end{align*}
		with probability at least $1-C\exp[-Cs\log(T)\log(pd+1)]$. 
		By Lemma \ref{lemma:restricted eigenvalue condition}, we have
		\begin{align}\label{eq:re}
			\langle\nabla \mathbb{L}(\widehat{\bm{\Theta}}_{\bm{\cdot},j}) - \nabla\mathbb{L}(\bm{\Theta}_{\bm{\cdot},j}^*), \widehat{\bm{\Theta}}_{\bm{\cdot},j} - \bm{\Theta}_{\bm{\cdot},j}^*\rangle \geq \frac{1}{2}\lambda_{\min}(\boldsymbol{\Gamma}_x)\|\widehat{\bm{\Theta}}_{\bm{\cdot},j} - \bm{\Theta}_{\bm{\cdot},j}^*\|_2^2,
		\end{align}
		with probability at least $1-C\operatorname{exp}[-Cs\log(T)\log(pd+1)]$. Besides, by Lemma \ref{lemma:frobenius upper bound for symmetric bergman divergence}, we have
		\begin{align}\label{re:bd}
			\langle\nabla \mathbb{L}(\widehat{\bm{\Theta}}_{\bm{\cdot},j}) - \nabla\mathbb{L}(\bm{\Theta}_{\bm{\cdot},j}^*), \widehat{\bm{\Theta}}_{\bm{\cdot},j} - \bm{\Theta}_{\bm{\cdot},j}^*\rangle \leq 6\lambda\sqrt{s}\|\widehat{\bm{\Theta}}_{\bm{\cdot},j} - \bm{\Theta}_{\bm{\cdot},j}^*\|_2.
		\end{align}
		Combining \eqref{eq:re} and \eqref{re:bd}, for each $1 \leq j \leq d$, we have 
		\begin{align}\label{eq:rowwise upper bound}
			\|\widehat{\bm{\Theta}}_{\bm{\cdot},j} - \bm{\Theta}_{\bm{\cdot},j}^*\|_2 \leq 12\lambda\sqrt{s}\lambda_{\min}^{-1}(\boldsymbol{\Gamma}_x) \asymp s^{3/2}\|\bm{\Theta}^*\|_{1,\infty}\lambda_{\min}^{-1}(\boldsymbol{\Gamma}_x)\left(\frac{M^{1/\epsilon}\log(pd+1)}{T_{\textnormal{eff}}}\right)^{\frac{\epsilon}{1+\epsilon}},
		\end{align}
		with probability at least $1-C\operatorname{exp}[-Cs\log(T)\log(pd+1)]$. 
		Thus, we have
		\begin{align*}
			\|\widehat{\bm{\Theta}} - \bm{\Theta}^*\|_{2,\infty} \lesssim s^{3/2}\|\bm{\Theta}^*\|_{1,\infty}\lambda_{\min}^{-1}(\boldsymbol{\Gamma}_x)\left(\frac{M^{1/\epsilon}\log(pd+1)}{T_{\textnormal{eff}}}\right)^{\frac{\epsilon}{1+\epsilon}}
		\end{align*}
		and
		\begin{equation*}
			\|\widehat{\bm{\Theta}} - \bm{\Theta}^*\|_{\mathrm{F}} = \left(\sum_{j=1}^{d}\|\widehat{\bm{\Theta}}_{\bm{\cdot},j} - \bm{\Theta}_{\bm{\cdot},j}^*\|_2^2\right)^{\frac{1}{2}} \lesssim s^{3/2}N\|\bm{\Theta}^*\|_{1,\infty}\lambda_{\min}^{-1}(\boldsymbol{\Gamma}_x)\left(\frac{M^{1/\epsilon}\log(pd+1)}{T_{\textnormal{eff}}}\right)^{\frac{\epsilon}{1+\epsilon}},
		\end{equation*}
		with probability at least $1-C\operatorname{exp}[-Cs\log(T)\log(pd+1)]$.
	\end{proof}

	\begin{proof}[\textbf{Proof of Theorem \ref{thm:vech minimax lower bound}}]
	We begin the proof by considering the case of $p=1$ and no intercept term. Let $\mathcal{S}=\bigcup_{j=1}^{d}\{(i,j):i\in\mathcal{S}_j\}$ with $\mathcal{S}_j\subseteq\{1,\ldots,d\}$ and $|\mathcal{S}_j|=s$. Then, $|\mathcal{S}|=sd=:s_B$, and any coefficient matrix supported on $\mathcal{S}$ satisfies $\|\bm{\Theta}_{\cdot,j}^*\|_0\leq s$ for $1\leq j\leq d$. We enumerate the elements of $\mathcal{S}$ as $\mathcal{S}=\{\kappa_1,\ldots,\kappa_{s_B}\}$.

	For each $\bbm{\omega}=(\omega_1,\ldots,\omega_{s_B})\in\{-1,1\}^{s_B}$, define a $d\times d$ matrix $\bm A_{\bbm{\omega}}$ such that
	\[
		(A_{\bbm{\omega}})_{i,j}
		=
		\begin{cases}
		\omega_k, & (i,j)=\kappa_k\in\mathcal S,\\
		0, & (i,j)\notin\mathcal S.
		\end{cases}
	\]
	Therefore, $\|(\bm A_{\bbm{\omega}})_{\cdot,j}\|_0=s$ for $1\leq j\leq d$.

	Let $\widetilde T=\left\lceil T/L\right\rceil$ with $L=\left\lfloor\log_{\zeta^{-2}}T\right\rfloor$, and divide $\{1,\ldots,T\}$ into $\widetilde T$ consecutive index groups $\mathcal G_1,\ldots,\mathcal G_{\widetilde T}$, where each group contains at most $L$ consecutive indices. For each $1\leq g\leq\widetilde T$, let $B_g\sim\operatorname{Bernoulli}(\gamma)$ and let $\bm R_g=(R_{g,1},\ldots,R_{g,d})^\mathrm T$ be a vector of independent Rademacher random variables satisfying
	\[
		\mathbb P(R_{g,i}=1)=\mathbb P(R_{g,i}=-1)=\frac12,
		\qquad
		1\leq i\leq d.
	\]
	In addition, let $\bm U_t=(U_{t,1},\ldots,U_{t,d})^\mathrm T$ for $t=1,\ldots,T$, be independent Rademacher random vectors whose coordinates are mutually independent. We assume that $\{\bm U_t\}_{t=1}^T$, $\{B_g,\bm R_g\}_{g=1}^{\widetilde T}$ are mutually independent.

	For $t\in\mathcal G_g$, define $\bm X_t=b\bm U_t+aB_g\bm R_g$ and $\bm Y_{t+1}^{(\bbm{\omega})}=cB_g\bm A_{\bbm{\omega}}^\mathrm T\bm R_g$. Thus, for each $1\leq i,j\leq d$, $X_{t,i}=bU_{t,i}+aB_gR_{g,i}$, and $Y_{t+1,j}^{(\bbm{\omega})}=cB_g\sum_{i\in\mathcal S_j}(A_{\bbm{\omega}})_{i,j}R_{g,i}$. It follows that $\mathbb E X_{t,i}=0$. Moreover, since $U_{t,i}$, $B_g$, and $R_{g,i}$ are mutually independent and centered, $\mathbb E X_{t,i}^2=b^2+\gamma a^2$. For $i\neq i'$, we also have $\mathbb E X_{t,i}X_{t,i'}=0$. Hence, $\boldsymbol{\Gamma}_x:=\mathbb E\left[\bm X_t\bm X_t^\mathrm T\right]=(b^2+\gamma a^2)\bm I_d$.

	Let $r_T:=M^{\frac{1}{1+\epsilon}}\gamma^{\frac{\epsilon}{1+\epsilon}}$, and fix $\lambda_x>0$ such that $\lambda_x\geq r_T$. We take $a=c=(M/\gamma)^{\frac{1}{2+2\epsilon}}$ and $b^2=\lambda_x-\gamma a^2=\lambda_x-r_T$. Then, $\boldsymbol{\Gamma}_x=\lambda_x\bm I_d$, and consequently $\lambda_{\min}(\boldsymbol{\Gamma}_x)=\lambda_x$.

	We next verify the moment conditions. Using $|x+y|^q \leq 2^{q-1}(|x|^q+|y|^q)$, we have
	\begin{align*}
		\mathbb E|X_{t,i}|^{2+2\epsilon}=\mathbb E\left|bU_{t,i}+aB_gR_{g,i}\right|^{2+2\epsilon}\leq2^{1+2\epsilon}\left(b^{2+2\epsilon}+\gamma a^{2+2\epsilon}\right)=2^{1+2\epsilon}\left(b^{2+2\epsilon}+M\right)\leq 2^{1+2\epsilon}\left(\lambda_x^{1+\epsilon}+M\right).
	\end{align*}
	Therefore, the required moment condition for $\bm X_t$ is satisfied provided that $2^{1+2\epsilon}(\lambda_x^{1+\epsilon}+M) \leq M'$.
	For the response coordinate, $Y_{t+1,j}^{(\bbm{\omega})}=cB_g\sum_{i\in\mathcal S_j}(A_{\bbm{\omega}})_{i,j}R_{g,i}$. By the Khintchine inequality,
	\[
		\mathbb E\left|\sum_{i\in\mathcal S_j}(A_{\bbm{\omega}})_{i,j}R_{g,i}\right|^{2+2\epsilon} \leq C_\epsilon\left(\sum_{i\in\mathcal S_j}(A_{\bbm{\omega}})_{i,j}^2\right)^{1+\epsilon}.
	\]
	Since $\sum_{i\in\mathcal S_j}(A_{\bbm{\omega}})_{i,j}^2=s$, we obtain
	\begin{align*}
		\mathbb E\left|Y_{t+1,j}^{(\bbm{\omega})}\right|^{2+2\epsilon}=\gamma c^{2+2\epsilon}\mathbb E\left|\sum_{i\in\mathcal S_j}(A_{\bbm{\omega}})_{i,j}R_{g,i}\right|^{2+2\epsilon}\leq C_\epsilon s^{1+\epsilon}\gamma c^{2+2\epsilon}=C_\epsilon s^{1+\epsilon}M.
	\end{align*}
	Thus, the required $(2+2\epsilon)$-th moment conditions are satisfied uniformly over $\bbm{\omega}$ provided that $C_\epsilon s^{1+\epsilon}M \leq M'$. We next calculate the lag-one cross moments. Since $\bm U_t$ is independent of $B_g$, $\bm R_g$, and $\bm Y_{t+1}^{(\bbm{\omega})}$, and $\mathbb E\bm U_t=\bm 0$, we have $\mathbb E[b\bm U_t(\bm Y_{t+1}^{(\bbm{\omega})})^\mathrm T]=\bm 0$. Moreover, $\mathbb E[\bm R_g\bm R_g^\mathrm T]=\bm I_d$. Therefore,
	\begin{align*}
		\boldsymbol{\Gamma}_{xy}^{(\bbm{\omega})}:=\mathbb E\left[\bm X_t\left(\bm Y_{t+1}^{(\bbm{\omega})}\right)^\mathrm T\right]=ac\,\mathbb E\left[B_g^2\bm R_g\bm R_g^\mathrm T\right]\bm A_{\bbm{\omega}}=\gamma ac\bm A_{\bbm{\omega}}=r_T\bm A_{\bbm{\omega}}.
	\end{align*}
	In particular, for each $\kappa_k=(i_k,j_k)\in\mathcal S$, $\mathbb E[X_{t,i_k}Y_{t+1,j_k}^{(\bbm{\omega})}]=\omega_k r_T$. We now verify the dependence condition. For $t_1<t_2$ belonging to different groups, since the random variables associated with different groups are independent, we have
	\[
		\alpha\left(\sigma(\{\bm Y_t:t\leq t_1\}),\sigma(\{\bm Y_t:t\geq t_2\})\right)=0.
	\]

	Next, consider $t_1<t_2$ belonging to the same group $\mathcal G_g$. By the definition of the $\alpha$-mixing coefficient,
	\begin{align*}
		\alpha\left(\sigma(\{\bm Y_t:t\leq t_1\}),\sigma(\{\bm Y_t:t\geq t_2\})\right)=\sup_{\substack{A_1\in\sigma(\{\bm Y_t:t\leq t_1\}),\\A_2\in\sigma(\{\bm Y_t:t\geq t_2\})}}\left|\mathbb P(A_1\cap A_2)-\mathbb P(A_1)\mathbb P(A_2)\right|.
	\end{align*}
	For any such $A_1$ and $A_2$, let $\eta_l=\mathbb P(A_l\mid B_g=0)\in\{0,1\}$, $p_l=\mathbb P(A_l\mid B_g=1)$ for $l=1,2$, and $p_{12}=\mathbb P(A_1\cap A_2\mid B_g=1)$. Since all the random components in $\mathcal G_g$ reduce to deterministic zero vectors conditional on $B_g=0$, we have $\mathbb P(A_1\cap A_2\mid B_g=0)=\eta_1\eta_2$. Therefore, $\mathbb P(A_l)=(1-\gamma)\eta_l+\gamma p_l$ for $l=1,2$, and $\mathbb P(A_1\cap A_2)=(1-\gamma)\eta_1\eta_2+\gamma p_{12}$.

	We consider the four possible values of $(\eta_1,\eta_2)$ separately.

	If $\eta_1=\eta_2=0$, then $\mathbb P(A_1\cap A_2)-\mathbb P(A_1)\mathbb P(A_2)=\gamma p_{12}-\gamma^2p_1p_2=\gamma\left(p_{12}-\gamma p_1p_2\right)$.
	Since $0\leq p_{12}\leq1$ and $0\leq\gamma p_1p_2\leq1$, we have $-1\leq p_{12}-\gamma p_1p_2 \leq 1$. Hence, $\left|\mathbb P(A_1\cap A_2)-\mathbb P(A_1)\mathbb P(A_2)\right|\leq \gamma$.

	If $\eta_1=0$ and $\eta_2=1$, then
	\begin{align*}
		\mathbb P(A_1\cap A_2)-\mathbb P(A_1)\mathbb P(A_2)=\gamma p_{12}-\gamma p_1(1-\gamma+\gamma p_2)=\gamma\left[p_{12}-p_1+\gamma p_1(1-p_2)\right].
	\end{align*}
	Since $0\leq p_{12}\leq p_1\leq1$, we have $-1 \leq p_{12}-p_1 \leq0$. Moreover, $0\leq \gamma p_1(1-p_2)\leq1$. More precisely, the whole expression inside the brackets satisfies $p_{12}-p_1+\gamma p_1(1-p_2) \geq -p_1 \geq -1$, while $p_{12}-p_1 + \gamma p_1(1-p_2)\leq \gamma p_1(1-p_2) \leq 1$. Therefore, $\left|p_{12}-p_1+\gamma p_1(1-p_2)\right|\leq1$, and hence $\left|\mathbb P(A_1\cap A_2)-\mathbb P(A_1)\mathbb P(A_2)\right|\leq\gamma$.

	The case $\eta_1=1$ and $\eta_2=0$ follows analogously. Indeed, $\mathbb P(A_1\cap A_2)-\mathbb P(A_1)\mathbb P(A_2)=\gamma\left[p_{12}-p_2+\gamma p_2(1-p_1)\right]$, and, since $0\leq p_{12}\leq p_2\leq1$, we similarly obtain $-1 \leq p_{12}-p_2+\gamma p_2(1-p_1) \leq 1$. Thus, $\left|\mathbb P(A_1\cap A_2)-\mathbb P(A_1)\mathbb P(A_2)\right|\leq\gamma$.

	Finally, if $\eta_1=\eta_2=1$, then
	\begin{align*}
		\mathbb P(A_1\cap A_2)-\mathbb P(A_1)\mathbb P(A_2)&=1-\gamma+\gamma p_{12}-(1-\gamma+\gamma p_1)(1-\gamma+\gamma p_2)\\
		&=\gamma\left[1-p_1-p_2+p_{12}-\gamma(1-p_1)(1-p_2)\right].
	\end{align*}
	Note that $1-p_1-p_2+p_{12}=\mathbb P(A_1^c\cap A_2^c\mid B_g=1)\in[0,1]$, and $\gamma(1-p_1)(1-p_2)\in[0,1]$. Therefore, $-1\leq 1-p_1-p_2+p_{12}-\gamma(1-p_1)(1-p_2)\leq 1$, which implies $\left|\mathbb P(A_1\cap A_2)-\mathbb P(A_1)\mathbb P(A_2)\right| \leq \gamma$.

	Combining the four cases, for any admissible $A_1$ and $A_2$, $\left|\mathbb P(A_1\cap A_2)-\mathbb P(A_1)\mathbb P(A_2)\right|\leq\gamma$. Taking the supremum over $A_1$ and $A_2$ yields
	\[
		\alpha(t_2-t_1) \leq \gamma.
	\]

	We take $\gamma = c_\gamma/\widetilde T$ for some sufficiently small constant $c_\gamma>0$. Since $\zeta^{\log_{\zeta^{-2}}T}=T^{-1/2}$ and $L=\left\lfloor\log_{\zeta^{-2}}T\right\rfloor$, we have, for $0<\zeta<1$, $T^{-1/2}=\zeta^{\log_{\zeta^{-2}}T}\leq\zeta^L$.
	Moreover, $\log T/T\lesssim T^{-1/2}$ for sufficiently large $T$.
	Therefore, for $1\leq t_2-t_1\leq L$,
	\[
		\alpha(t_2-t_1) \leq \gamma \lesssim \frac{\log T}{T} \lesssim T^{-1/2} \leq \zeta^L \leq \zeta^{t_2-t_1}.
	\]

	For $t_2-t_1>L$, observations belong to different groups and hence the corresponding mixing coefficient is zero. Therefore, the constructed processes are $\alpha$-mixing with geometric mixing coefficients. Together with the moment conditions above, $\mathbb P_{\bbm{\omega}} \in \mathsf P_y(\epsilon,M',\zeta')$ for $\bbm{\omega}\in\{-1,1\}^{s_B}$.

	Next, we derive the corresponding coefficient matrices. By the Yule-Walker equation, $\bm{\Theta}^{*(\bbm{\omega})}=\boldsymbol{\Gamma}_x^{-1}\boldsymbol{\Gamma}_{xy}^{(\bbm{\omega})}$. Since $\boldsymbol{\Gamma}_x = \lambda_x\bm I_d$ and $\boldsymbol{\Gamma}_{xy}^{(\bbm{\omega})}=r_T\bm A_{\bbm{\omega}}$, we have
	\[
		\bm{\Theta}^{*(\bbm{\omega})}=\frac{r_T}{\lambda_x}\bm A_{\bbm{\omega}}.
	\]
	Therefore, for each $\kappa_k\in\mathcal S$, $\Theta_{\kappa_k}^{*(\bbm{\omega})}=\omega_k\delta_T$, where $\delta_T:=r_T/\lambda_x$. Since $\lambda_x=\lambda_{\min}(\boldsymbol{\Gamma}_x)$, we further have
	\begin{align*}
		\delta_T=\lambda_{\min}^{-1}\left(\boldsymbol{\Gamma}_x\right)M^{\frac{1}{1+\epsilon}}\gamma^{\frac{\epsilon}{1+\epsilon}}\asymp\lambda_{\min}^{-1}\left(\boldsymbol{\Gamma}_x\right)\left(\frac{M^{1/\epsilon}\log T}{T}\right)^{\frac{\epsilon}{1+\epsilon}}.
	\end{align*}
	For $(i,j)\notin\mathcal S$, we have $\Theta_{i,j}^{*(\bbm{\omega})}=0$. Hence, $\operatorname{supp}(\bm{\Theta}^{*(\bbm{\omega})}) \subseteq \mathcal S$ and $\|\bm{\Theta}_{\cdot,j}^{*(\bbm{\omega})}\|_0 \leq s$ for $1\leq j\leq d$.

	For each $1\leq k\leq s_B$, let $\bbm{\omega}^{(k)}$ be obtained from $\bbm{\omega}$ by changing only its $k$-th component, namely,
	\[
		\omega_l^{(k)}
		=
		\begin{cases}
			-\omega_k, & l=k,\\
			\omega_l, & l\neq k.
		\end{cases}
	\]
	Then, $\|\bm{\Theta}^{*(\bbm{\omega})}-\bm{\Theta}^{*(\bbm{\omega}^{(k)})}\|_{\mathrm F}^2=4\delta_T^2$.

	We next show that two neighboring models cannot be distinguished with arbitrarily high probability. Consider $\mathbb P_{\bbm{\omega}}$ and $\mathbb P_{\bbm{\omega}^{(k)}}$. We construct a coupling of the two processes by using the same $\{\bm U_t\}_{t=1}^T$ and $\{B_g,\bm R_g\}_{g=1}^{\widetilde T}$ under the two models. Under this coupling, the two observed samples are identical if no group is activated. Hence,
	\begin{align*}
		\mathbb P\left(\{(\bm X_t,\bm Y_{t+1}^{(\bbm{\omega})})\}_{t=1}^T=\{(\bm X_t,\bm Y_{t+1}^{(\bbm{\omega}^{(k)})})\}_{t=1}^T\right)\geq \mathbb P\left(B_1=\cdots=B_{\widetilde T}=0\right)=(1-\gamma)^{\widetilde T}.
	\end{align*}
	By the standard coupling inequality \citep[Lemma~3]{peluchetti2013strong}, for any coupling $(Z,Z')$ with marginal distributions $P$ and $Q$, $\operatorname{TV}(P,Q) \leq \mathbb P(Z\neq Z')$, or equivalently, $1-\operatorname{TV}(P,Q) \geq \mathbb P(Z=Z')$. Therefore, $1-\operatorname{TV}(\mathbb P_{\bbm{\omega}},\mathbb P_{\bbm{\omega}^{(k)}})\geq(1-\gamma)^{\widetilde T}$. Since $\gamma=c_\gamma/\widetilde T$, we have $\gamma\leq1/2$ for sufficiently large $T$. Using $\log(1-x)\geq-x/(1-x)$ for $0<x<1$, we obtain
	\begin{align*}
		(1-\gamma)^{\widetilde T}=\exp\left\{\widetilde T\log(1-\gamma)\right\}\geq\exp\left(-\frac{\gamma\widetilde T}{1-\gamma}\right)\geq\exp(-2\gamma\widetilde T)=\exp(-2c_\gamma)=:c_0>0.
	\end{align*}
	Consequently, $1-\operatorname{TV}(\mathbb P_{\bbm{\omega}},\mathbb P_{\bbm{\omega}^{(k)}})\geq c_0$ uniformly over $\bbm{\omega}$ and $1\leq k\leq s_B$.

	Let $\widehat{\bm{\Theta}}$ be any estimator of $\bm{\Theta}^{*(\bbm{\omega})}$. Since $\Theta_{\kappa_k}^{*(\bbm{\omega})}=\omega_k\delta_T$ for $\omega_k\in\{-1,1\}$, we define the induced estimator of $\omega_k$ by $\widehat{\omega}_k=\operatorname{sign}(\widehat{\Theta}_{\kappa_k})$, where the value at zero can be assigned arbitrarily. If $\widehat{\omega}_k\neq\omega_k$, then $|\widehat{\Theta}_{\kappa_k}-\Theta_{\kappa_k}^{*(\bbm{\omega})}| \geq \delta_T$. Consequently, $|\widehat{\Theta}_{\kappa_k}-\Theta_{\kappa_k}^{*(\bbm{\omega})}|^2 \geq \delta_T^2 \mathds 1 \left\{\widehat{\omega}_k\neq\omega_k\right\}$.	Summing over $k=1,\ldots,s_B$, we obtain $\|\widehat{\bm{\Theta}}-\bm{\Theta}^{*(\bbm{\omega})}\|_{\mathrm F}^2\geq\delta_T^2\sum_{k=1}^{s_B}\mathds 1\left\{\widehat{\omega}_k\neq\omega_k\right\}$. Taking expectation under $\mathbb P_{\bbm{\omega}}$ gives
	\[
		\mathbb E_{\bbm{\omega}}\|\widehat{\bm{\Theta}}-\bm{\Theta}^{*(\bbm{\omega})}\|_{\mathrm F}^2\geq \delta_T^2\sum_{k=1}^{s_B}\mathbb P_{\bbm{\omega}}\left(\widehat{\omega}_k\neq\omega_k\right).
	\]
	Since the supremum over $\bbm{\omega}\in\{-1,1\}^{s_B}$ is no smaller than the average over all $2^{s_B}$ possible values of $\bbm{\omega}$, we further have
	\begin{align*}
		\sup_{\bbm{\omega}\in\{-1,1\}^{s_B}}\mathbb E_{\bbm{\omega}}\|\widehat{\bm{\Theta}}-\bm{\Theta}^{*(\bbm{\omega})}\|_{\mathrm F}^2\geq\frac{\delta_T^2}{2^{s_B}}\sum_{\bbm{\omega}\in\{-1,1\}^{s_B}}\sum_{k=1}^{s_B}\mathbb P_{\bbm{\omega}}\left(\widehat{\omega}_k\neq\omega_k\right).
	\end{align*}

	For each $1\leq k\leq s_B$, pair $\bbm{\omega}$ with $\bbm{\omega}^{(k)}$. Define $\mathcal A_k=\left\{\widehat{\omega}_k=\omega_k\right\}$. Since $\omega_k^{(k)}=-\omega_k$, we have $\left\{\widehat{\omega}_k\neq\omega_k^{(k)}\right\}=\mathcal A_k$ under $\bbm{\omega}^{(k)}$, whereas $\left\{\widehat{\omega}_k\neq\omega_k\right\}=\mathcal A_k^c$ under $\bbm{\omega}$. Therefore,
	\begin{align*}
		\mathbb P_{\bbm{\omega}}\left(\widehat{\omega}_k\neq\omega_k\right)+\mathbb P_{\bbm{\omega}^{(k)}}\left(\widehat{\omega}_k\neq\omega_k^{(k)}\right)=1-\left[\mathbb P_{\bbm{\omega}}(\mathcal A_k)-\mathbb P_{\bbm{\omega}^{(k)}}(\mathcal A_k)\right]\geq 1-\operatorname{TV}\left(\mathbb P_{\bbm{\omega}},\mathbb P_{\bbm{\omega}^{(k)}}\right)\geq c_0.
	\end{align*}

	For each fixed $k$, the hypercube $\{-1,1\}^{s_B}$ contains $2^{s_B-1}$ disjoint pairs $(\bbm{\omega},\bbm{\omega}^{(k)})$. Hence,
	\[
		\sum_{\bbm{\omega}\in\{-1,1\}^{s_B}}\mathbb P_{\bbm{\omega}}\left(\widehat{\omega}_k\neq\omega_k\right) \geq 2^{s_B-1}c_0.
	\]
	Summing over $k=1,\ldots,s_B$, we obtain
	\begin{align*}
		\sup_{\bbm{\omega}\in\{-1,1\}^{s_B}}\mathbb E_{\bbm{\omega}}\|\widehat{\bm{\Theta}}-\bm{\Theta}^{*(\bbm{\omega})}\|_{\mathrm F}^2\geq \frac{\delta_T^2}{2^{s_B}}\sum_{k=1}^{s_B}2^{s_B-1}c_0=\frac{c_0}{2}s_B\delta_T^2.
	\end{align*}
	Since the above inequality holds for any estimator $\widehat{\bm{\Theta}}$, we further obtain
	\begin{align*}
		\inf_{\widehat{\bm{\Theta}}}\sup_{\bbm{\omega}\in\{-1,1\}^{s_B}}\mathbb E_{\bbm{\omega}}\|\widehat{\bm{\Theta}}-\bm{\Theta}^{*(\bbm{\omega})}\|_{\mathrm F}^2\gtrsim s_B\lambda_{\min}^{-2}\left(\boldsymbol{\Gamma}_x\right)\left(\frac{M^{1/\epsilon}\log T}{T}\right)^{\frac{2\epsilon}{1+\epsilon}}.
	\end{align*}

	Since $\left\{\mathbb P_{\bbm{\omega}}:\bbm{\omega}\in\{-1,1\}^{s_B}\right\}\subseteq \mathsf P_y(\epsilon,M',\zeta')$ and $\|\bm{\Theta}_{\cdot,j}^{*(\bbm{\omega})}\|_0 \leq s$ for $1\leq j\leq d$, we conclude that
	\begin{align*}
		\inf_{\widehat{\bm{\Theta}}}\sup_{\substack{\mathbb P\in\mathsf P_y(\epsilon,M',\zeta')\\\|\bm{\Theta}_{\cdot,j}^*\|_0\leq s,\;1\leq j\leq d}}\mathbb E\|\widehat{\bm{\Theta}}-\bm{\Theta}^*\|_{\mathrm F}^2\gtrsim sd\lambda_{\min}^{-2}\left(\boldsymbol{\Gamma}_x\right)\left(\frac{M^{1/\epsilon}\log T}{T}\right)^{\frac{2\epsilon}{1+\epsilon}}\asymp sN^2\lambda_{\min}^{-2}\left(\boldsymbol{\Gamma}_x\right)\left(\frac{M^{1/\epsilon}\log T}{T}\right)^{\frac{2\epsilon}{1+\epsilon}}. \qedhere
	\end{align*}

	Finally, we extend the result to the case of $p>1$ and with an intercept term. Consider the subclass of the VAR$(p)$ model satisfying $\bbm{\mu}=\bm 0$ and $\bm A_2=\cdots=\bm A_p=\bm 0$, while $\bm A_1$ follows the construction above. This subclass reduces to the no-intercept VAR$(1)$ model considered previously. Since the minimax risk over the entire parameter space is no smaller than that over any of its subclasses,
	\[
		\inf_{\widehat{\bm{\Theta}}}\sup_{\bm{\Theta}^*\in\mathcal F_p}\mathbb E\left\|\widehat{\bm{\Theta}}-\bm{\Theta}^*\right\|_{\mathrm F}^2\geq \inf_{\widehat{\bm{\Theta}}}\sup_{\bm{\Theta}^*\in\mathcal F_1}\mathbb E\left\|\widehat{\bm{\Theta}}-\bm{\Theta}^*\right\|_{\mathrm F}^2.
	\]
	Therefore, the same minimax lower bound continues to hold for the VAR$(p)$ model with an intercept term.
	\end{proof}

	\begin{proof}[\textbf{Proof of Corollary \ref{cor:bekk upper bound}}]
		We begin by analyzing the error bound for $\widehat{\bbm{\omega}}$, and then derive the error bound for $\widehat{\mathbf{A}}_{ik}$.
		
		\noindent\textbf{Error Bound for $\widehat{\bbm{\omega}}$}: Given the construction of $\mathbf{D}_N$, it holds that $\|\mathbf{D}_N\bm{\beta}\|_\mathrm{2} \leq \sqrt{2}\|\bm{\beta}\|_2$ for any $\bm{\beta} \in \mathbb{R}^d$. Then, note that $\bbm{\omega}^*=\operatorname{vec}^{-1}(\mathbf{D}_N\bbm{\omega}^*)$ and $\widehat{\bbm{\omega}}=\mathcal{P}(\operatorname{vec}^{-1}(\mathbf{D}_N\widehat{\bbm{\omega}}))$, it follows from the nonexpansiveness of the projection operator that
		\begin{align}\label{eq:projection}
			\|\widehat{\boldsymbol{\Omega}} - \boldsymbol{\Omega}^*\|_{\mathrm{F}} &= \|\mathcal{P}(\operatorname{vec}^{-1}(\mathbf{D}_N\widehat{\bbm{\omega}})) - \boldsymbol{\Omega}^*\|_{\mathrm{F}} \leq \|\operatorname{vec}^{-1}(\mathbf{D}_N\widehat{\bbm{\omega}}) - \boldsymbol{\Omega}^*\|_{\mathrm{F}}\notag \\
			&= \|\mathbf{D}_N(\widehat{\boldsymbol{\omega}} - \boldsymbol{\omega}^*)\|_2 \leq \sqrt{2}\|\widehat{\boldsymbol{\omega}} - \boldsymbol{\omega}^*\|_{2}.
		\end{align}
		By Theorem \ref{thm:vech upper bound}, we have the probabilistic upper bound of $\|\widehat{\boldsymbol{\omega}} - \boldsymbol{\omega}^*\|_{2}$ as
		\begin{align}\label{eq:upper bound of omega}
			\|\widehat{\boldsymbol{\omega}} - \boldsymbol{\omega}^*\|_{2} \leq \|\widehat{\bm{\Theta}} - \bm{\Theta}^*\|_{\mathrm{F}} \lesssim s^{3/2}N\|\bm{\Theta}^*\|_{1,\infty}\lambda_{\min}^{-1}(\boldsymbol{\Gamma}_x)\left(\frac{M^{1/\epsilon}\log(pd+1)}{T_{\textnormal{eff}}}\right)^{\frac{\epsilon}{1+\epsilon}},
		\end{align}
		with probability at least $1-C\operatorname{exp}[-Cs\log(T)\log(pd+1)]$. Thus, combining \eqref{eq:projection} with the probabilistic upper bound in \eqref{eq:upper bound of omega}, we have
		\begin{align*}
			\|\widehat{\boldsymbol{\Omega}} - \boldsymbol{\Omega}^*\|_{\mathrm{F}} \lesssim s^{3/2}N\|\bm{\Theta}^*\|_{1,\infty}\lambda_{\min}^{-1}(\boldsymbol{\Gamma}_x)\left(\frac{M^{1/\epsilon}\log(pd+1)}{T_{\textnormal{eff}}}\right)^{\frac{\epsilon}{1+\epsilon}},
		\end{align*}
		with probability at least $1-C\operatorname{exp}[-Cs\log(T)\log(pd+1)]$.
		
		\noindent\textbf{Error Bound for $\widehat{\mathbf{A}}_{ik}$}: In the proof, we fix an index $i$ and derive the corresponding results. We begin by establishing an upper bound for the error term $\|\widehat{\mathbf{A}}_{ik} - \mathbf{A}_{ik}^*\|_{\mathrm{F}}$ in terms of $\|\mathcal{H}(\widehat{\bm{\Phi}}_i,\widehat{\mathbf{W}}_i) - \mathcal{H}(\bm{\Phi}_i^*,\mathbf{W}^*_i)\|_\mathrm{F}$. As noted in the Proof of Theorem \ref{thm:selection consistency of K_i}, matrix $\mathcal{R}(\mathcal{H}(\widehat{\bm{\Phi}}_i,\widehat{\mathbf{W}}_i))$ is symmetric. Since the rearrangement operator $\mathcal{R}(\cdot)$ preserves the Frobenius norm, it follows that
		\begin{align*}
			\left\|\sum_{k=1}^{N^2}\widehat{\lambda}_{i,k}\widehat{\boldsymbol{u}}_{i,k}\widehat{\boldsymbol{u}}_{i,k}^{\mathrm{T}} - \sum_{k=1}^{K_i}\lambda_{i,k}^*\boldsymbol{u}_{i,k}^*\boldsymbol{u}_{i,k}^{*\mathrm{T}}\right\|_{\mathrm{F}} &= \left\|\mathcal{R}(\mathcal{H}(\widehat{\bm{\Phi}}_i,\widehat{\mathbf{W}})) - \mathcal{R}(\mathcal{H}(\bm{\Phi}_i^*,\mathbf{W}^*_i))\right\|_\mathrm{F}\\
			&= \left\|\mathcal{H}(\widehat{\bm{\Phi}}_i,\widehat{\mathbf{W}}_i) - \mathcal{H}(\bm{\Phi}_i^*,\mathbf{W}^*_i)\right\|_\mathrm{F},
		\end{align*}
		where $\lambda_{i,k}^* = \|\mathbf{A}^*_{ik}\|_{\mathrm{F}}^2$, $\boldsymbol{u}_{i,k}^* = \operatorname{vec}(\mathbf{A}^*_{ik})/\|\mathbf{A}_{ik}^*\|_{\mathrm{F}}$, and $\sum_{k=1}^{N^2}\widehat{\lambda}_{i,k}\widehat{\boldsymbol{u}}_{i,k}\widehat{\boldsymbol{u}}_{i,k}^{\mathrm{T}}$ is the eigenvalue decomposition of $\mathcal{R}(\mathcal{H}(\widehat{\bm{\Phi}}_i,\widehat{\mathbf{W}}_i))$. Recall that $\widehat{\mathbf{A}}_{ik}=\operatorname{vec}^{-1}((\widehat{\lambda}_{i,k}^{+})^{1/2}\,\widehat{\bm{u}}_{i,k})$, we have
		\begin{align}\label{eq:upper bound of Aik}
			\|\widehat{\mathbf{A}}_{ik} - \mathbf{A}_{ik}^*\|_{\mathrm{F}} &= \left\|(\widehat{\lambda}_{i,k}^{+})^{1/2}\widehat{\bm{u}}_{i,k} - \lambda_{i,k}^{*1/2}\boldsymbol{u}_{i,k}^*\right\|_{2} = \left\|((\widehat{\lambda}_{i,k}^{+})^{1/2} - \lambda_{i,k}^{*1/2})\widehat{\bm{u}}_{i,k} + \lambda_{i,k}^{*1/2}(\widehat{\bm{u}}_{i,k} - \boldsymbol{u}_{i,k}^*)\right\|_{2} \notag\\
			&\leq \left|(\widehat{\lambda}_{i,k}^{+})^{1/2} - \lambda_{i,k}^{*1/2}\right|\cdot\left\|\widehat{\boldsymbol{u}}_{i,k}\right\|_2 + \lambda_{i,k}^{*1/2}\left\|\widehat{\boldsymbol{u}}_{i,k} - \boldsymbol{u}_{i,k}^*\right\|_2\\
			& = \underbrace{\left|(\widehat{\lambda}_{i,k}^{+})^{1/2} - \lambda_{i,k}^{*1/2}\right|}_{P_1} + \lambda_{i,k}^{*1/2}\underbrace{\left\|\widehat{\boldsymbol{u}}_{i,k} - \boldsymbol{u}_{i,k}^*\right\|_2}_{P_2}.\notag
		\end{align}
		Next, we derive the upper bound of $P_1$ and $P_2$. For $P_1$, by the nonexpansiveness of projection of $\widehat{\lambda}_{i,k}^{+}$ and the Weyl's inequality in Lemma \ref{lemma:weyl inequality}, we have
		\begin{align}\label{eq:upper bound of eigenvalue}
			P_1&=\left|(\widehat{\lambda}_{i,k}^{+})^{1/2} - \lambda_{i,k}^{*1/2}\right| = \frac{|\widehat{\lambda}_{i,k}^{+} - \lambda_{i,k}^*|}{(\widehat{\lambda}_{i,k}^{+})^{1/2} + \lambda_{i,k}^{*1/2}} \leq \frac{|\widehat{\lambda}_{i,k} - \lambda_{i,k}^*|}{\lambda_{i,k}^{*1/2}}\notag\\
			&\leq \frac{\left\|\sum_{k=1}^{N^2}\widehat{\lambda}_{i,k}\widehat{\boldsymbol{u}}_{i,k}\widehat{\boldsymbol{u}}_{i,k}^{\mathrm{T}} - \sum_{k=1}^{K_i}\lambda_{i,k}^*\boldsymbol{u}_{i,k}^*\boldsymbol{u}_{i,k}^{*\mathrm{T}}\right\|_{\textnormal{op}}}{\lambda_{i,k}^{*1/2}} \leq \frac{\left\|\mathcal{H}(\widehat{\bm{\Phi}}_i,\widehat{\mathbf{W}}_i) - \mathcal{H}(\bm{\Phi}_i^*,\mathbf{W}^*_i)\right\|_\mathrm{F}}{\lambda_{i,k}^{*1/2}}.
		\end{align}
		The minimum adjacent norm gap can be represented as $\vartheta_i = \min_{1 \leq k \leq K_i^*}\{\lambda_{i,k-1}^{*1/2} - \lambda_{i,k}^{*1/2},\lambda_{i,k}^{*1/2} - \lambda_{i,k+1}^{*1/2}\}$ with $\lambda_{i,0}^* = \infty$ and $\lambda_{i,K_i^*+1}^* = 0$ for $1 \leq i \leq p$ under Assumption \ref{assumption:restricted parameter space}(ii). Then, by the variant of Davis Kahan Theorem in Lemma \ref{lemma:davis-kahan theorem}, for $\langle\widehat{\boldsymbol{u}}_{i,k},\boldsymbol{u}_{i,k}^*\rangle \geq 0$, we have
		\begin{align}\label{eq:upper bound of eigenvector}
			P_2&=\|\widehat{\boldsymbol{u}}_{i,k} - \boldsymbol{u}_{i,k}^*\|_2 \leq 2^{3/2}\vartheta_i^{-1}\left\|\sum_{k=1}^{N^2}\widehat{\lambda}_{i,k}\widehat{\boldsymbol{u}}_{i,k}\widehat{\boldsymbol{u}}_{i,k}^{\mathrm{T}} - \sum_{k=1}^{K_i}\lambda_{i,k}^*\boldsymbol{u}_{i,k}^*\boldsymbol{u}_{i,k}^{*\mathrm{T}}\right\|_{\textnormal{op}} \notag\\
			&\leq 2^{3/2}\vartheta_i^{-1}\left\|\mathcal{H}(\widehat{\bm{\Phi}}_i,\widehat{\mathbf{W}}_i) - \mathcal{H}(\bm{\Phi}_i^*,\mathbf{W}^*_i)\right\|_\mathrm{F}.
		\end{align}
		Plugging \eqref{eq:upper bound of eigenvalue} and \eqref{eq:upper bound of eigenvector} into \eqref{eq:upper bound of Aik}, we have
		\begin{align}\label{eq:upper bound of Aik 2}
			\|\widehat{\mathbf{A}}_{ik} - \mathbf{A}_{ik}^*\|_{\mathrm{F}} &\leq \frac{\left\|\mathcal{H}(\widehat{\bm{\Phi}}_i,\widehat{\mathbf{W}}_i) - \mathcal{H}(\bm{\Phi}_i^*,\mathbf{W}^*_i)\right\|_\mathrm{F}}{\lambda_{i,k}^{*1/2}} + 2^{3/2}\vartheta_i^{-1}\lambda_{i,k}^{*1/2}\left\|\mathcal{H}(\widehat{\bm{\Phi}}_i,\widehat{\mathbf{W}}_i) - \mathcal{H}(\bm{\Phi}_i^*,\mathbf{W}^*_i)\right\|_\mathrm{F} \notag\\
			&\leq 2^{7/4}\vartheta_i^{-1/2}\left\|\mathcal{H}(\widehat{\bm{\Phi}}_i,\widehat{\mathbf{W}}_i) - \mathcal{H}(\bm{\Phi}_i^*,\mathbf{W}^*_i)\right\|_\mathrm{F}.
		\end{align}
		Then by Assumption \ref{assump:asym_error} and Theorem \ref{thm:vech upper bound}, for each $1 \leq k \leq K_i$, we have 
		\begin{align*}
			\|\widehat{\mathbf{A}}_{ik} - \mathbf{A}_{ik}^*\|_{\mathrm{F}} &\leq 2^{7/4}\vartheta_i^{-1/2}\left\|\mathcal{H}(\widehat{\bm{\Phi}}_i,\widehat{\mathbf{W}}_i) - \mathcal{H}(\bm{\Phi}_i^*,\mathbf{W}^*_i)\right\|_\mathrm{F} \\
			&\lesssim \vartheta_i^{-1/2}\|\widehat{\bm{\Phi}}_i - \bm{\Phi}_i^*\|_\mathrm{F}\leq \vartheta_i^{-1/2}\|\widehat{\bm{\Theta}} - \bm{\Theta}^*\|_{\mathrm{F}}\\
			&\lesssim s^{3/2}N\|\bm{\Theta}^*\|_{1,\infty}\vartheta_i^{-1/2}\lambda_{\min}^{-1}(\boldsymbol{\Gamma}_x)\left(\frac{M^{1/\epsilon}\log(pd+1)}{T_{\textnormal{eff}}}\right)^{\frac{\epsilon}{1+\epsilon}},
		\end{align*}
		with probability at least $1-C\exp[-Cs\log(T)\log(pd+1)]$. Since the preceding arguments do not depend on the specific choice of lag $i$, the results hold uniformly for all $1 \leq i \leq p$. The proof for $\widetilde{\bm A}_{ik}$ is analogous to that for $\widehat{\bm A}_{ik}$ and is therefore omitted. This completes the proof.
	\end{proof}

	\begin{proof}[\textbf{Proof of Thoerem \ref{thm:selection consistency of model order}}]
		To prove this thoerem, it suffices to show that for all $p \neq p^*$ with $p,p^* \leq \overline{p}$, as $T \to \infty$, $\mathbb{P}(\text{BIC}(p) > \text{BIC}(p^*)) \to 1$. Given $p$, $\bm{\Theta}_p^\circ$ is defined by
		\begin{align*}
			\bm{\Theta}_p^\circ = \operatorname{argmin}_{\bm{\Theta}_p \in \mathbb{R}^{(pd+1) \times d}, \|\bm{\Theta}_p\|_0 \leq sd} \limits \mathbb{E}\mathbb{L}(\bm{\Theta}_p),
		\end{align*}
		where $\mathbb{L}(\bm{\Theta}_p) = 1/(2T)\|\mathbf{Y} - \mathbf{X}_p(\tau)\bm{\Theta}_p\|_\mathrm{F}^2$ with $\mathbf{X}_p(\tau)$ adapted to lag order $p$. Note that when $p = p^*$, $\bm{\Theta}_p^\circ = \bm{\Theta}^*$ because $\bm{\Theta}^*$ satisfies the sparsity constraint $\|\bm{\Theta}^*\|_0 \leq sd$ and is a global minimizer of $\mathbb{E} \mathbb{L}(\cdot)$. Therefore, it must also be a minimizer over the restricted parameter space. Analogously, $\widehat{\bm{\Theta}}_p = \widehat{\bm{\Theta}}$ when $p = p^*$.
		Define 
		\begin{align*}
			\varphi_p = \iota_d \left(\frac{\log(pd+1)}{T_\text{eff}}\right)^{\frac{1+2\epsilon}{1+\epsilon}}.
		\end{align*}
		Then the BIC takes the form of $\text{BIC}(p) = \log(\mathbb{L}(\widehat{\bm{\Theta}}_p)) + \varphi_p \log(T)$.
		Note that $\widetilde{\bm{\Theta}}^\circ_p = (\bm{\Theta}_p^{\circ\mathrm{T}}, \mathbf{0}_{d \times (p^* - p)d})^\mathrm{T} \in \mathbb{R}^{(p^*d+1) \times d}$ for $p \leq p^*$. Then we have 
		\begin{align*}
			\text{BIC}(p) - \text{BIC}(p^*) = \log\left(1+\frac{D_p}{\mathbb{L}(\widehat{\bm{\Theta}})}\right) + (\varphi_p - \varphi_{p^*}) \log(T),
		\end{align*}
		where for $p \leq p^*$, 
		\begin{align*}
			D_p = \mathbb{L}(\widehat{\bm{\Theta}}_p) - \mathbb{L}(\widehat{\bm{\Theta}}) =& \underbrace{\mathbb{L}(\bm{\Theta}_p^\circ) - \mathbb{L}(\bm{\Theta}^*) - \frac12\operatorname{vec}^\mathrm{T}(\widetilde{\bm{\Theta}}_p^\circ - \bm{\Theta}^*)(\mathbf{I}_d \otimes \boldsymbol{\Gamma}_x)\operatorname{vec}(\widetilde{\bm{\Theta}}_p^\circ - \bm{\Theta}^*)}_{D_{p,1}}\\
			&+ \underbrace{\frac12\operatorname{vec}^\mathrm{T}(\widetilde{\bm{\Theta}}_p^\circ - \bm{\Theta}^*)(\mathbf{I}_d \otimes \boldsymbol{\Gamma}_x)\operatorname{vec}(\widetilde{\bm{\Theta}}_p^\circ - \bm{\Theta}^*)}_{D_{p,2}} \\ 
			&+ \underbrace{\mathbb{L}(\widehat{\bm{\Theta}}_p) - \mathbb{L}(\bm{\Theta}_p^\circ)}_{D_{p,3}} - \underbrace{[\mathbb{L}(\widehat{\bm{\Theta}}) - \mathbb{L}(\bm{\Theta}^*)]}_{D_{p^*}},
		\end{align*}
		and for $p > p^*$, 
		\begin{align*}
			D_p &= \mathbb{L}(\widehat{\bm{\Theta}}_p) - \mathbb{L}(\widehat{\bm{\Theta}}) = \underbrace{\mathbb{L}(\bm{\Theta}_p^\circ) - \mathbb{L}(\bm{\Theta}^*)}_{D_{p,1}'} + \underbrace{\mathbb{L}(\widehat{\bm{\Theta}}_p) - \mathbb{L}(\bm{\Theta}_p^\circ)}_{D_{p,2}'} - \underbrace{[\mathbb{L}(\widehat{\bm{\Theta}}) - \mathbb{L}(\bm{\Theta}^*)]}_{D_{p^*}}.
		\end{align*}
		We first establish the probabilistic order of $D_{p^*}$. By the optimality of $\widehat{\bm{\Theta}}$, we have
		\begin{align*}
			\mathbb{L}(\widehat{\bm{\Theta}}) + \lambda\|\widehat{\bm{\Theta}}\|_{1,1} \leq \mathbb{L}(\bm{\Theta}^*) + \lambda\|\bm{\Theta}^*\|_{1,1}.
		\end{align*}
		Under the conditions of Theorem \ref{thm:vech upper bound}, Lemmas \ref{lemma:first order derivative} and \ref{lemma:ell1 cone} imply that, with probability at least $1 - C\exp[-Cs\log(T)\log(p^*d+1)]$, we have $\|(\widehat{\bm{\Theta}}_{\bm{\cdot},j} - \bm{\Theta}^*_{\bm{\cdot},j})_{\mathcal{S}_j^c}\|_1 \leq 3\|(\widehat{\bm{\Theta}}_{\bm{\cdot},j} - \bm{\Theta}^*_{\bm{\cdot},j})_{\mathcal{S}_j}\|_1$ for all $1 \leq j \leq d$. Let $\mathcal{S}$ denote the support set of $\bm{\Theta}^*$. Then, it follows that $\|(\widehat{\bm{\Theta}} - \bm{\Theta}^*)_{\mathcal{S}^c}\|_{1,1} \leq 3\|(\widehat{\bm{\Theta}} - \bm{\Theta}^*)_\mathcal{S}\|_{1,1}$ with probability at least $1 - C\exp[-Cs\log(T)\log(p^*d+1)]$. Therefore, as $T \to \infty$, with probability tending to 1, we have
		\begin{align*}
			D_{p^*} = \mathbb{L}(\widehat{\bm{\Theta}}) - \mathbb{L}(\bm{\Theta}^*) &\leq \lambda_{p^*}\left(\|\bm{\Theta}^*\|_{1,1} - \|\widehat{\bm{\Theta}}\|_{1,1}\right) \\ 
			&\leq \lambda_{p^*}\|\widehat{\bm{\Theta}} - \bm{\Theta}^*\|_{1,1}
			= \lambda_{p^*}\left(\|(\widehat{\bm{\Theta}} - \bm{\Theta}^*)_\mathcal{S}\|_{1,1} + \|(\widehat{\bm{\Theta}} - \bm{\Theta}^*)_{\mathcal{S}^c}\|_{1,1}\right) \\
			&\leq 4\lambda_{p^*}\|(\widehat{\bm{\Theta}} - \bm{\Theta}^*)_\mathcal{S}\|_{1,1} \leq 4\sqrt{sd}\lambda_{p^*}\|\widehat{\bm{\Theta}} - \bm{\Theta}^*\|_\mathrm{F}.
		\end{align*}
		Combining with the probabilistic upper bound under Frobenius norm in Theorem \ref{thm:vech upper bound}, we have $D_{p^*} = O_p((s^{1/2}N\lambda_{p^*})^2)$. Next, we discuss on the under- and over-specified models separately.
		
		\noindent\textbf{Under-specified models}: For $p < p^*$, the key analysis is to derive the probabilistic lower bound for $D_{p,2}$ and show that $D_{p,2}$ dominates the other terms in $D_p$.

		Let $\varphi_p = \|\widetilde{\bm{\Theta}}_p^\circ - \bm{\Theta}^*\|_{1,1}$. Since $\boldsymbol{\Gamma}_x$ is positive definite, we have
		\begin{align*}
			D_{p,2} &= \frac12\operatorname{vec}^\mathrm{T}(\widetilde{\bm{\Theta}}_p^\circ - \bm{\Theta}^*)\left(\mathbf{I}_d \otimes \boldsymbol{\Gamma}_x\right)\operatorname{vec}(\widetilde{\bm{\Theta}}_p^\circ - \bm{\Theta}^*) \\ 
			&\geq \frac12\lambda_{\min}(\boldsymbol{\Gamma}_x)\|\widetilde{\bm{\Theta}}_p^\circ - \bm{\Theta}^*\|_\mathrm{F}^2\\
			&\geq \frac12 \cdot (2sd)^{-1}\lambda_{\min}(\boldsymbol{\Gamma}_x)\|\widetilde{\bm{\Theta}}_p^\circ - \bm{\Theta}^*\|_{1,1}^2 = (4sd)^{-1}\lambda_{\min}(\boldsymbol{\Gamma}_x)\varphi_p^2 = O(\varphi_p^2/sd).
			\end{align*}
		By $\|\mathbf{M}_1\|_\mathrm{F}^2 - \|\mathbf{M}_2\|_\mathrm{F}^2= \|\mathbf{M}_1 - \mathbf{M}_2\|_\mathrm{F}^2 + 2\langle \mathbf{M}_1 - \mathbf{M}_2, \mathbf{M}_2\rangle$, we have
		\begin{align*}
			D_{p,1} &= \mathbb{L}(\bm{\Theta}_p^\circ) - \mathbb{L}(\bm{\Theta}^*) - \frac12\operatorname{vec}^\mathrm{T}(\widetilde{\bm{\Theta}}_p^\circ - \bm{\Theta}^*)(\mathbf{I}_d \otimes \boldsymbol{\Gamma}_x)\operatorname{vec}(\widetilde{\bm{\Theta}}_p^\circ - \bm{\Theta}^*)\\
			&= \frac{1}{2T}\|\mathbf{Y}(\tau) - \mathbf{X}_p(\tau)\bm{\Theta}_p^\circ\|_\mathrm{F}^2 - \frac{1}{2T}\|\mathbf{Y}(\tau) - \mathbf{X}(\tau)\bm{\Theta}^*\|_\mathrm{F}^2 - \frac12\operatorname{vec}^\mathrm{T}(\widetilde{\bm{\Theta}}_p^\circ - \bm{\Theta}^*)(\mathbf{I}_d \otimes \boldsymbol{\Gamma}_x)\operatorname{vec}(\widetilde{\bm{\Theta}}_p^\circ - \bm{\Theta}^*)\\
			&= \frac{1}{2T}\|\mathbf{X}(\tau)(\bm{\Theta}^* - \widetilde{\bm{\Theta}}_p^\circ)\|_\mathrm{F}^2 + \frac{1}{T}\langle \mathbf{X}(\tau)(\bm{\Theta}^* - \widetilde{\bm{\Theta}}_p^\circ), \mathbf{Y}(\tau) - \mathbf{X}(\tau)\bm{\Theta}^* \rangle\\
			&\quad - \frac12\operatorname{vec}^\mathrm{T}(\widetilde{\bm{\Theta}}_p^\circ - \bm{\Theta}^*)(\mathbf{I}_d \otimes \boldsymbol{\Gamma}_x)\operatorname{vec}(\widetilde{\bm{\Theta}}_p^\circ - \bm{\Theta}^*)\\
			&= \frac12\operatorname{vec}^\mathrm{T}(\widetilde{\bm{\Theta}}_p^\circ - \bm{\Theta}^*)\left(\mathbf{I}_d \otimes (\mathbf{S} - \bm{\Gamma}_x)\right)\operatorname{vec}(\widetilde{\bm{\Theta}}_p^\circ - \bm{\Theta}^*) + \langle \nabla\mathbb{L}(\bm{\Theta}^*), \bm{\Theta}^* - \widetilde{\bm{\Theta}}_p^\circ \rangle\\
			& \leq \frac12\|\mathbf{S} - \bm{\Gamma}_x\|_{\infty,\infty}\|\widetilde{\bm{\Theta}}_p^\circ - \bm{\Theta}^*\|_{1,1}^2 + \|\nabla\mathbb{L}(\bm{\Theta}^*)\|_{\infty,\infty} \|\widetilde{\bm{\Theta}}_p^\circ - \bm{\Theta}^*\|_{1,1},
		\end{align*}
		where $\mathbf{X}_p(\tau)$ is adapted to lag order $p$, $\mathbf{S} = \frac1T\mathbf{X}^{\mathrm{T}}(\tau)\mathbf{X}(\tau)$ and $\nabla\mathbb{L}(\bm{\Theta}^*) = \frac{1}{T}\mathbf{X}^{\mathrm{T}}(\tau)(\mathbf{Y}(\tau) - \mathbf{X}(\tau)\bm{\Theta}^*)$. Under the conditions of Theorem \ref{thm:vech upper bound}, by Lemmas \ref{lemma:covariance matrix bound} and \ref{lemma:first order derivative}, we have $\|\mathbf{S} - \bm{\Gamma}_x\|_{\infty,\infty} \lesssim \lambda_{p^*}$ and $\|\nabla\mathbb{L}(\bm{\Theta}^*)\|_{\infty,\infty} \lesssim \lambda_{p^*}$ with probability at least $1 - C\exp[-Cs\log(T)\log(p^*d+1)]$. Therefore, with probability tending to 1, we have
		\begin{align*}
			D_{p,1} \lesssim \lambda_{p^*}\|\bm{\Theta}_p^\circ - \bm{\Theta}^*\|_{1,1}^2 + \lambda_{p^*}\|\bm{\Theta}_p^\circ - \bm{\Theta}^*\|_{1,1} = \lambda_{p^*}(\varphi_p^2 + \varphi_p),
		\end{align*}
		and thus $D_{p,1} = O_p(\lambda_{p^*}(\varphi_p^2 + \varphi_p)) = o_p(\varphi_p^2/sd)$. Note that $D_{p,2} = O(\varphi_p^2/sd)$ and $D_{p^*} = O_p((s^{1/2}N\lambda_{p^*})^2) = o_p(\varphi_p^2/sd)$. In addition, $\max_{p \neq p^*} \limits \mathbb{L}(\widehat{\bm{\Theta}}_p) - \mathbb{L}(\bm{\Theta}_p^\circ) = O_p(\varphi_p^2/sd)$ in Assumption \ref{assump:minimum signal strength}(ii) implies that $D_{p,3} = O_p(\varphi_p^2/sd)$. Combining the above results, we have $D_p = D_{p,1} + D_{p,2} + D_{p,3} - D_{p^*} = O_p(\varphi_p^2/sd)$. As $p$ is fixed and $\min_{p < p^*} \limits \varphi_p \gg sd\lambda_p\log(T)^{1/2}$ by Assumption \ref{assump:minimum signal strength}(i), we have $\varphi_p^2/sd^2\tau^4\|\bm{\Theta}^*\|_{1,\infty}^2 \gg s\lambda_{p^*}^2\log(T)/\tau^4\|\bm{\Theta}^*\|_{1,\infty}^2$. This together with the boundedness of $\mathbf{X}(\tau)$, implies that
		\begin{align*}
			&\mathbb{L}(\bm{\Theta}^*) = \frac{1}{2T}\sum_{t=1}^{T}\sum_{j=1}^{d}(y_{t,j}(\tau) - \bbm{x}_t^\mathrm{T}(\tau)\bm{\Theta}_{\bm{\cdot},j}^*)^2 \leq \frac{1}{T}\sum_{t=1}^{T}\sum_{j=1}^{d}\left(y_{t,j}^2(\tau) + \langle\bbm{x}_t(\tau),\bm{\Theta}_{\bm{\cdot},j}^*\rangle^2\right)\\
			&\leq \frac{1}{T}\sum_{t=1}^{T}\sum_{j=1}^{d}\left(y_{t,j}^2(\tau) + \|\bbm{x}_t(\tau)\|_\infty^2\|\bm{\Theta}_{\bm{\cdot},j}^*\|_1^2\right) \leq \frac{1}{T}\sum_{t=1}^{T}\sum_{j=1}^{d}\left(\tau^4 + \tau^4\|\bm{\Theta}^*\|_{1,\infty}^2\right)
			= O(d\tau^4\|\bm{\Theta}^*\|_{1,\infty}^2).
		\end{align*}
		Then, we have $\mathbb{L}(\widehat{\bm{\Theta}}) = D_{p^*} + \mathbb{L}(\bm{\Theta}^*) = O_p((s^{1/2}N\lambda_{p^*})^2) + O(d\tau^4\|\bm{\Theta}^*\|_{1,\infty}^2) = O_p(d\tau^4\|\bm{\Theta}^*\|_{1,\infty}^2)$. By Assumption \ref{assump:penalty parameter}, we have 
		\begin{align*}
			\varphi_{p^*}\log(T) &= \iota_d \left(\frac{\log(p^*d+1)}{T_{\text{eff}}}\right)^{\frac{1+2\epsilon}{1+\epsilon}}\log(T) \asymp s^3M^{1/(1+\epsilon)}\left(\frac{\log(p^*d+1)}{T_{\text{eff}}}\right)^{\frac{1+2\epsilon}{1+\epsilon}}\log(T)\\
			&\asymp s\lambda_{p^*}^2\log(T)/\tau^4\|\bm{\Theta}^*\|_{1,\infty}^2 \ll \varphi_p^2/sd^2\tau^4\|\bm{\Theta}^*\|_{1,\infty}^2.
		\end{align*}
		This together with the fact that $D_p/\mathbb{L}(\widehat{\bm{\Theta}}) = O_p(\varphi_p^2/sd^2\tau^4\|\bm{\Theta}^*\|_{1,\infty}^2)$, $D_{p,2}/\mathbb{L}(\widehat{\bm{\Theta}})>0$ is the dominant term in $D_p/\mathbb{L}(\widehat{\bm{\Theta}})$ and $\log(1+x) \geq \min(0.5x,\log2)$ for any $x > 0$, we have
		\begin{align*}
			\text{BIC}(p) - \text{BIC}(p^*) \geq \min\left(\frac{D_p}{2\mathbb{L}(\widehat{\bm{\Theta}})}, \log2\right) + (\varphi_p - \varphi_{p^*}) \log(T) > 0
		\end{align*}
		with probability tending to 1 as $T \to \infty$.
		
		\noindent\textbf{Over-specified models}: Let $p > p^*$. Since $\bm{\Theta}^*$ is a global minimizer of $\mathbb{E}\widetilde{\mathbb{L}}(\bm{\Theta}_{p^*})$ and $\|\bm{\Theta}^*\|_{0} \leq sd$, it follows that $\bm{\Theta}_{p^*}^\circ = \bm{\Theta}^*$. Consequently, the minimum is attained at $\bm{\Theta}_p^\circ = (\bm{\Theta}^{*\mathrm{T}},\mathbf{0}_{(p-p^*)d \times d})^\mathrm{T}$, i.e.,
		\begin{align*}
			\bm{\Theta}_p^\circ = (\bm{\Theta}^{*\mathrm{T}},\mathbf{0}_{(p-p^*)d \times d})^\mathrm{T} = \operatorname{argmin}_{\bm{\Theta}_p \in \mathbb{R}^{(pd+1) \times d}, \|\bm{\Theta}_p\|_0 \leq sd} \limits \mathbb{E}\widetilde{\mathbb{L}}(\bm{\Theta}_p).
		\end{align*}
		Thus, it follows that
		\begin{align*}
			D_{p,1}' = \mathbb{L}(\bm{\Theta}_p^\circ) - \mathbb{L}(\bm{\Theta}^*) = \frac{1}{2T}\|\mathbf{E}_t\|_\mathrm{F}^2 - \frac{1}{2T}\|\mathbf{E}_t\|_\mathrm{F}^2 = 0.
		\end{align*}
		Following the same arguments of $D_{p,3}$, we have $D_{p,2}' = O_p((s^{1/2}N\lambda_p)^2)$. As mentioned before, $D_{p^*} = O_p((s^{1/2}N\lambda_{p^*})^2)$ and $\mathbb{L}(\widehat{\bm{\Theta}}) = O_p(d\tau^4\|\bm{\Theta}^*\|_{1,\infty}^2)$. Thus, we have
		\begin{align*}
			\frac{D_p}{\mathbb{L}(\widehat{\bm{\Theta}})} &= \frac{D_{p,1}' + D_{p,2}' - D_{p^*}}{\mathbb{L}(\widehat{\bm{\Theta}})} = (s\lambda_p^2/\tau^4\|\bm{\Theta}^*\|_{1,\infty}^2 - s\lambda_{p^*}^2/\tau^4\|\bm{\Theta}^*\|_{1,\infty}^2)O_p(1) = (\varphi_p - \varphi_{p^*})O_p(1).
		\end{align*}
		The last equality holds because $\varphi_p \asymp s\lambda_p^2/\tau^4\|\bm{\Theta}^*\|_{1,\infty}^2$ for any $p \leq \overline{p}$. Noting that $\varphi_p - \varphi_{p^*} > 0$ when $p > p^*$, and applying the inequaility $\log(1+x) \leq x$ for any $x > 0$, we have
		\begin{align*}
			\log\left(1+\frac{D_p}{\mathbb{L}(\widehat{\bm{\Theta}})}\right) \geq -\frac{D_p}{\mathbb{L}(\widehat{\bm{\Theta}})}
		\end{align*}
		with probability tending to 1 as $T \to \infty$. Therefore, we have
		\begin{align*}
			\text{BIC}(p) - \text{BIC}(p^*) = \log\left(1+\frac{D_p}{\mathbb{L}(\widehat{\bm{\Theta}})}\right) + (\varphi_p - \varphi_{p^*})\log(T) \geq -\frac{D_p}{\mathbb{L}(\widehat{\bm{\Theta}})} + (\varphi_p - \varphi_{p^*})\log(T) > 0
		\end{align*}
		with probability tending to 1 as $T \to \infty$. This concludes the proof.
	\end{proof}

	\begin{proof}[\textbf{Proof of Theorem \ref{thm:selection consistency of K_i}}]
		In the proof of this theorem, we fix an index $i$ and derive the corresponding results. Under the conditions in Theorem \ref{thm:vech upper bound}, we have
		\begin{align*}
			\|\widehat{\bm{\Phi}}_i - \bm{\Phi}^*_i\|_{\mathrm{F}} \leq \|\widehat{\bm{\Theta}} - \bm{\Theta}^*\|_{\mathrm{F}} \lesssim s^{3/2}N\|\bm{\Theta}^*\|_{1,\infty}\lambda_{\min}^{-1}(\boldsymbol{\Gamma}_x)\left(\frac{M^{1/\epsilon}\log(pd+1)}{T_{\textnormal{eff}}}\right)^{\frac{\epsilon}{1+\epsilon}}
		\end{align*}
		with probability at least $1-C\exp[-Cs\log(T)\log(pd+1)]$. By Assumption \ref{assump:asym_error}, we have
		\begin{align*}
			\|\mathcal{H}(\widehat{\bm{\Phi}}_i,\widetilde{\mathbf{W}}_i) - \mathcal{H}(\bm{\Phi}_i^*,\mathbf{W}^*_i)\|_\mathrm{F} \lesssim \|\widehat{\bm{\Phi}}_i - \bm{\Phi}_i^*\|_\mathrm{F} \lesssim s^{3/2}N\|\bm{\Theta}^*\|_{1,\infty}\lambda_{\min}^{-1}(\boldsymbol{\Gamma}_x)\left(\frac{M^{1/\epsilon}\log(pd+1)}{T_{\textnormal{eff}}}\right)^{\frac{\epsilon}{1+\epsilon}}
		\end{align*}
		with $\mathcal{H}(\bm{\Phi}_i^*,\mathbf{W}^*_i) = \sum_{k=1}^{K_i}\mathbf{A}_{ik}^* \otimes \mathbf{A}_{ik}^*$. Since the rearrangement operator $\mathcal{R}(\cdot)$ does not change the Frobenius norm, we have
		\begin{align*}
			\|\mathcal{R}(\mathcal{H}(\widehat{\bm{\Phi}}_i,\widetilde{\mathbf{W}}_i)) - \mathcal{R}(\mathcal{H}(\bm{\Phi}_i^*,\mathbf{W}^*_i))\|_\mathrm{F} \lesssim s^{3/2}N\|\bm{\Theta}^*\|_{1,\infty}\lambda_{\min}^{-1}(\boldsymbol{\Gamma}_x)\left(\frac{M^{1/\epsilon}\log(pd+1)}{T_{\textnormal{eff}}}\right)^{\frac{\epsilon}{1+\epsilon}}
		\end{align*}
		with probability at least $1-C\exp[-Cs\log(T)\log(pd+1)]$, where $\mathcal{R}(\mathcal{H}(\bm{\Phi}_i^*,\mathbf{W}^*_i)) = \sum_{k=1}^{K_i^*}\lambda_{i,k}^* \bm{u}_{i,k}^*\bm{u}_{i,k}^{*\mathrm{T}}$ with $\lambda_{i,k}^*=\|\mathbf{A}_{ik}\|_\mathrm{F}^2$ and $\bm{u}_{i,k}^*=\operatorname{vec}(\mathbf{A}_{ik}^*)/\|\mathbf{A}_{ik}^*\|_\mathrm{F}$. Under Assumption \ref{assumption:restricted parameter space}, the right-hand side constitutes the eigendecomposition of $\mathcal{R}(\mathcal{H}(\bm{\Phi}_i^*,\mathbf{W}^*_i))$, with eigenvalues arranged in decreasing order. It is worth noting that the matrices $\mathcal{H}(\widehat{\bm{\Phi}}_i,\widetilde{\mathbf{W}}_i)$ are padded while preserving the underlying Kronecker product structure. Consequently, their rearranged forms $\mathcal{R}(\mathcal{H}(\widehat{\bm{\Phi}}_i,\widetilde{\mathbf{W}}_i))$ are symmetric. Since $\mathcal{R}(\mathcal{H}(\widehat{\bm{\Phi}}_i,\widetilde{\mathbf{W}}_i))$ and $\mathcal{R}(\mathcal{H}(\bm{\Phi}_i^*,\mathbf{W}^*_i))$ are symmetric matrices, let $\widehat{\lambda}_{i,k}$ and $\lambda_{i,k}^*$ denote the $k$-th eigenvalues of $\mathcal{R}(\mathcal{H}(\widehat{\bm{\Phi}}_i,\widetilde{\mathbf{W}}_i))$ and $\mathcal{R}(\mathcal{H}(\bm{\Phi}_i^*,\mathbf{W}^*_i))$, respectively. It follows that $\lambda_{i,k}^* = \|\mathbf{A}_{ik}^*\|_\mathrm{F}^2$ for $1 \leq k \leq K_i^*$, and $\lambda_{i,k}^* = 0$ for $k > K_i^*$. Then, by the Weyl's inequality in Lemma \ref{lemma:weyl inequality}, we have√
		\begin{align*}
			\max_{1 \leq k \leq K_i^*} \limits |\widetilde{\lambda}_{i,k} - \lambda_{i,k}^*| &\leq \|\mathcal{R}(\mathcal{H}(\widehat{\bm{\Phi}}_i,\widetilde{\mathbf{W}}_i)) - \mathcal{R}(\mathcal{H}(\bm{\Phi}_i^*,\mathbf{W}^*_i))\|_{\mathrm{op}} \leq \|\mathcal{R}(\mathcal{H}(\widehat{\bm{\Phi}}_i,\widetilde{\mathbf{W}}_i)) - \mathcal{R}(\mathcal{H}(\bm{\Phi}_i^*,\mathbf{W}^*_i))\|_\mathrm{F}\\ 
			&\lesssim s^{3/2}N\|\bm{\Theta}^*\|_{1,\infty}\lambda_{\min}^{-1}(\boldsymbol{\Gamma}_x)\left(\frac{M^{1/\epsilon}\log(pd+1)}{T_{\textnormal{eff}}}\right)^{\frac{\epsilon}{1+\epsilon}} = o(c(N,T))
		\end{align*}
		with probability approaching 1 as $T \to \infty$ under the conditions in Theorem \ref{thm:selection consistency of K_i}. That is, $\widetilde{\lambda}_{i,k} - \lambda_{i,k}^* = o_p(c(N,T))$ uniformly for $1 \leq k \leq K_i^*$.

		Since $c(N,T) \ll \varsigma \min_{1 \leq i \leq p,\ 1 \leq k \leq K_i^* - 1} \lambda_{i,k+1}^*/\lambda_{i,k}^*$ with $\varsigma = \min_{1 \leq i \leq p} \lambda_{i,K_i^*}^*$, we have 
		$$
		c(N,T) \ll \min_{1 \leq i \leq p,\ 1 \leq k \leq K_i^* - 1} \frac{\lambda_{i,K_i^*}^* \lambda_{i,k+1}^*}{\lambda_{i,k}^*}
		\leq \min_{1 \leq i \leq p,\ 1 \leq k \leq K_i^* - 1} \lambda_{i,k+1}^*
		= o(\lambda_{i,k}^*).
		$$
		Hence, $c(N,T) = o(\lambda_{i,k}^*)$ uniformly for $1 \le k \le K_i^*$. Note that $\widetilde{\lambda}_{i,k} + c(N,T)
		= \lambda_{i,k}^* + \big(\widetilde{\lambda}_{i,k} - \lambda_{i,k}^*\big) + c(N,T)$. For $k > K_i^*$, since $\lambda_{i,k}^* = 0$ and $\widetilde{\lambda}_{i,k} - \lambda_{i,k}^* = o_p(c(N,T))$, the term $c(N,T)$ dominates in $\widetilde{\lambda}_{i,k} + c(N,T)$, i.e., $\widetilde{\lambda}_{i,k} + c(N,T) = O_p(c(N,T))$. For $k \le K_i^*$, since $\widetilde{\lambda}_{i,k} - \lambda_{i,k}^* = o_p(c(N,T))$ and $c(N,T) = o(\lambda_{i,k}^*)$, the term $\lambda_{i,k}^*$ dominates in $\widetilde{\lambda}_{i,k} + c(N,T)$, i.e., $\widetilde{\lambda}_{i,k} + c(N,T) = O_p(\lambda_{i,k}^*)$.\\
		Hence, for $k > K_i^*$, as $T \to \infty$,
		\begin{align*}
			\frac{\widehat{\lambda}_{i,k+1}+c(N,T)}{\widetilde{\lambda}_{i,k}+c(N,T)} \overset{P}{\to} \frac{c(N,T)}{c(N,T)} = 1.
		\end{align*}
		For $k < K_i^*$, as $T \to \infty$,
		\begin{align*}
			\frac{\widehat{\lambda}_{i,k+1}+c(N,T)}{\widetilde{\lambda}_{i,k}+c(N,T)} \overset{P}{\rightarrow} \frac{\lambda_{i,k+1}^*}{\lambda_{i,k}^*}.
		\end{align*}
		Note that $c(N,T) \ll \varsigma \min_{1 \leq i \leq p,\ 1 \leq k \leq K_i^* - 1} \lambda_{i,k+1}^*/\lambda_{i,k}^*$ where $\varsigma = \min_{1 \leq i \leq p} \lambda_{i,K_i^*}^*$. Thus, for $k = K_i^*$, as $T \to \infty$,
		\begin{align*}
			\frac{\widehat{\lambda}_{i,k+1}+c(N,T)}{\widetilde{\lambda}_{i,k}+c(N,T)} \overset{P}{\rightarrow} \frac{c(N,T)}{\lambda_{i,K_i^*}^*} \leq \frac{c(N,T)}{\varsigma} = o\left(\min_{1 \leq k \leq K_i}\limits\frac{\lambda_{i,k+1}^*}{\lambda_{i,k}^*}\right).
		\end{align*}
		It follows that, with probability approaching 1 as $T \to \infty$, the ridge-type ratio $[\widehat{\lambda}_{i,K_i^*+1}+c(N,T)]/[\widehat{\lambda}_{i,K_i^*}+c(N,T)]$ achieves its minimum. Since the above arguments do not depend on the specific choice of $i$, the results hold uniformly for all $1 \leq i \leq p$.
		Note that
		\[
		\{\widehat{\mathcal{K}}_p^* \neq \mathcal{K}_p^*\}
		= \bigcup_{i=1}^p \{\widehat{K}_i \neq K_i^*\}.
		\]
		Then, by Boole's inequality,
		\[
		\mathbb{P}(\widehat{\mathcal{K}}_p \neq \mathcal{K}_p^*)
		\;\le\; \sum_{i=1}^{p} \mathbb{P}(\widehat{K}_i \neq K_i^*).
		\]
		Since $p^*$ is fixed and $\mathbb{P}(\widehat{K}_i \neq K_i^*) \to 0$ for each $i$,
		the right-hand side tends to $0$. Hence
		$\mathbb{P}(\widehat{\mathcal{K}}_p \neq \mathcal{K}_p^*) \to 0$, which implies
		$\mathbb{P}(\widehat{\mathcal{K}}_p=\mathcal{K}_p^*) \to 1$. This concludes the proof.
	\end{proof}

	\section{Proofs of Primary Lemmas}\label{append:proofs of primary lemmas}
	\renewcommand{\theequation}{D.\arabic{equation}}
	\setcounter{equation}{0}

	\begin{proof}[\textbf{Proof of Lemma \ref{lemma:covariance matrix bound}}]
	Denote $\boldsymbol{\Gamma}_x(\tau)=\mathbb{E}[\boldsymbol{x}_t(\tau)\boldsymbol{x}_t^\mathrm{T}(\tau)]$. 
	Recall that $\boldsymbol{\Gamma}_x=\mathbb{E}[\boldsymbol{x}_t\boldsymbol{x}_t^\mathrm{T}]$ and $\mathbf{S}=1/T\sum_{t=1}^T \bbm{x}_t(\tau)\bbm{x}_t^\mathrm{T}(\tau)$, where 
	$\boldsymbol{x}_t=(1,\boldsymbol{y}_t^\mathrm{T},\ldots,\boldsymbol{y}_{t-p}^\mathrm{T})^\mathrm{T}$ and 
	$\boldsymbol{x}_t(\tau)=(1,\boldsymbol{y}_t^\mathrm{T}(\tau),\ldots,\boldsymbol{y}_{t-p}^\mathrm{T}(\tau))^\mathrm{T}$with $\boldsymbol{y}_t\in\mathbb{R}^d$ and $\boldsymbol{x}_t\in\mathbb{R}^{pd+1}$. The matrices $\boldsymbol{\Gamma}_x(\tau), \boldsymbol{\Gamma}_x$ and $\mathbf{S}$ share the same block structure: the upper-left entry equals $1$; the intercept–lag-$l_1$ block equals the mean of $\bbm{y}_{t-l_1}(\tau)$ (i.e., $\mathbb{E}[\,\bbm{y}_{t-l_1}(\tau)\,]$ for $\bm{\Gamma}_x(\tau)$ and $\frac{1}{T}\sum_{t=1}^T \bbm{y}_{t-l_1}(\tau)$ for $\mathbf{S}$); and the lag–lag $(l_1,l_2)$ block equals the cross second moment of the corresponding lagged vectors (i.e., $\mathbb{E}[\,\bbm{y}_{t-l_1}(\tau)\,\bbm{y}_{t-l_2}^\mathrm{T}(\tau)]$ for $\bm{\Gamma}_x(\tau)$ and $\tfrac{1}{T}\sum_{t=1}^T \bbm{y}_{t-l_1}(\tau)\,\bbm{y}_{t-l_2}^\mathrm{T}(\tau)$ for $\mathbf{S}$). Replacing $\bbm{y}(\tau)$ by $\bbm{y}$ yields the untruncated counterparts.

	The proof proceeds in two steps. First, we establish a deterministic bound for $\|\boldsymbol{\Gamma}_x(\tau)-\boldsymbol{\Gamma}_x\|_{\infty,\infty}$; second, we derive a high-probability bound for $\|\mathbf{S}-\boldsymbol{\Gamma}_x(\tau)\|_{\infty,\infty}$. Specifically, leveraging the block structure of $\boldsymbol{\Gamma}_x(\tau), \boldsymbol{\Gamma}_x$ and $\mathbf{S}$ described above, entrywise differences can be decomposed into three types: the intercept–intercept entry (which vanishes identically), the intercept–lag blocks (differences of means of $\bbm{y}_{t-l_1}(\tau)$), and the lag–lag blocks (differences of second moments of lagged vectors). Consequently, it suffices to show uniform bounds, over the last two blocks, for $\mathbb{E}[\bbm{y}_{t-l_1}(\tau)]-\mathbb{E}[\,\bbm{y}_{t-l_1}]$ and $\mathbb{E}[\bbm{y}_{t-l_1}(\tau)\bbm{y}_{t-l_2}^\mathrm{T}(\tau)]-\mathbb{E}[\bbm{y}_{t-l_1}\bbm{y}_{t-l_2}^\mathrm{T}]$, as well as for $\tfrac{1}{T}\sum_{t=1}^T \bbm{y}_{t-l_1}(\tau)-\mathbb{E}[\,\bbm{y}_{t-l_1}(\tau)\,]$ and $\tfrac{1}{T}\sum_{t=1}^T \bbm{y}_{t-l_1}(\tau)\bbm{y}_{t-l_2}^\mathrm{T}(\tau)-\mathbb{E}[\,\bbm{y}_{t-l_1}(\tau)\bbm{y}_{t-l_2}^\mathrm{T}(\tau)\,]$.
	
	\noindent\textbf{Deterministic bound for $\|\boldsymbol{\Gamma}_x(\tau) - \boldsymbol{\Gamma}_x\|_{\infty,\infty}$}:
	For $1 \leq i,j \leq d$ and $1 \leq l_1,l_2 \leq p$, without loss of generality, we assume $y_{t-l_1,i} = r_{t-l_1,i_1}r_{t-l_1,i_2}$ and $y_{t-l_1,i}(\tau) = r_{t-l_1,i_1}(\tau)r_{t-l_1,i_2}(\tau)$. Similarly, $y_{t-l_2,j}$ and $y_{t-l_2,j}(\tau)$ are defined in the same manner. By the assumption that $\tau \asymp (MT_{\textnormal{eff}}/\log(pd+1))^{1/(4+4\epsilon)}$ and $T \gtrsim \log(pd+1)$, we have $\tau^{-2}\to 0$ as $T\to\infty$, and thus there exists a finite constant $C>0$ such that $\tau^{-2}\le C$ uniformly in $T$.  
	Then, we have
	\begin{align*}
		&\left|\mathbb{E}\Bigl(y_{t-l_1,i} y_{t-l_2,j}\Bigr) - \mathbb{E}\Bigl(y_{t-l_1,i}(\tau) y_{t-l_2,j}(\tau)\Bigr)\right| \\
		=& \Bigl|\mathbb{E}\bigl(r_{t-l_1,i_1} r_{t-l_1,i_2} 
		r_{t-l_2,j_1} r_{t-l_2,j_2} \bigr) - \mathbb{E}\bigl(r_{t-l_1,i_1}(\tau) r_{t-l_1,i_2}(\tau) r_{t-l_2,j_1}(\tau) r_{t-l_2,j_2}(\tau)\bigr) \Bigr| \\
		\leq& \mathbb{E} \Bigl[\Bigl| r_{t-l_1,i_1} r_{t-l_1,i_2} r_{t-l_2,j_1} r_{t-l_2,j_2} \Bigr| \Bigl( \mathds{1}(|r_{t-l_1,i_1}| \geq \tau) + \mathds{1}(|r_{t-l_1,i_2}| \geq \tau)\Bigr. \\
		&+ \mathds{1}(|r_{t-l_2,j_1}| \geq \tau) + \mathds{1}(|r_{t-l_2,j_2}| \geq \tau) \Bigr)\Bigr]\\
		\leq& \Bigl(\mathbb{E}\Bigl| r_{t-l_1,i_1} r_{t-l_1,i_2} r_{t-l_2,j_1} r_{t-l_2,j_2} \Bigr|^{1+\epsilon}\Bigr)^{\frac{1}{1+\epsilon}} \bigl[\mathbb{P}(|r_{t-l_1,i_1}| \geq \tau)^{\frac{\epsilon}{1+\epsilon}} + \mathbb{P}(|r_{t-l_1,i_2}| \geq \tau)^{\frac{\epsilon}{1+\epsilon}} \\
		&+ \mathbb{P}(|r_{t-l_2,j_1}| \geq \tau)^{\frac{\epsilon}{1+\epsilon}} + \mathbb{P}(|r_{t-l_2,j_2}| \geq \tau)^{\frac{\epsilon}{1+\epsilon}} \bigr]\\
		\leq& \Bigl(\mathbb{E}\left|r_{t-l_1,i_1}\right|^{4+4\epsilon} \mathbb{E}\left|r_{t-l_1,i_2}\right|^{4+4\epsilon} \mathbb{E}\left|r_{t-l_2,j_1}\right|^{4+4\epsilon} \mathbb{E}\left|r_{t-l_2,j_2}\right|^{4+4\epsilon}\Bigr)^{\frac{1}{4+4\epsilon}} \biggl[\Bigl(\frac{\mathbb{E}\left|r_{t-l_1,i_1}\right|^{4+4\epsilon}}{\tau^{4+4\epsilon}}\Bigr)^\frac{\epsilon}{1+\epsilon} \\
		& + \Bigl(\frac{\mathbb{E}\left|r_{t-l_1,i_2}\right|^{4+4\epsilon}}{\tau^{4+4\epsilon}}\Bigr)^\frac{\epsilon}{1+\epsilon} + \Bigl(\frac{\mathbb{E}\left|r_{t-l_2,j_1}\right|^{4+4\epsilon}}{\tau^{4+4\epsilon}}\Bigr)^\frac{\epsilon}{1+\epsilon} + \Bigl(\frac{\mathbb{E}\left|r_{t-l_2,j_2}\right|^{4+4\epsilon}}{\tau^{4+4\epsilon}}\Bigr)^\frac{\epsilon}{1+\epsilon} \biggr]\\
		\leq& M_{4+4\epsilon}^{\frac{1}{1+\epsilon}} \cdot 4\biggl(\frac{M_{4+4\epsilon}}{\tau^{4+4\epsilon}}\biggr)^{\frac{\epsilon}{1+\epsilon}} \leq 4\frac{M}{\tau^{4\epsilon}} \asymp \biggl(\frac{M^{1/\epsilon} \log(pd+1)}{T_{\textnormal{eff}}}\biggr)^{\frac{\epsilon}{1+\epsilon}},
	\end{align*}
	and
	\begin{align*}
		&\left|\mathbb{E}y_{t-l_1,i} - \mathbb{E}y_{t-l_1,i}(\tau)\right| = \Bigl|\mathbb{E}\bigl(r_{t-l_1,i_1} r_{t-l_1,i_2}\bigr)
		- \mathbb{E}\bigl(r_{t-l_1,i_1}(\tau) r_{t-l_1,i_2}(\tau)\bigr)\Bigr| \\
		\leq& \mathbb{E} \Bigl[\Bigl| r_{t-l_1,i_1} r_{t-l_1,i_2} \Bigr| \Bigl( \mathds{1}(|r_{t-l_1,i_1}| \geq \tau) + \mathds{1}(|r_{t-l_1,i_2}| \geq \tau) \Bigr)\Bigr]\\
		\leq& \Bigl(\mathbb{E}\Bigl| r_{t-l_1,i_1} r_{t-l_1,i_2} \Bigr|^{2+2\epsilon}\Bigr)^{\frac{1}{2+2\epsilon}} \bigl[ \mathbb{P}(|r_{t-l_1,i_1}| \geq \tau)^{\frac{1+2\epsilon}{2+2\epsilon}} + \mathbb{P}(|r_{t-l_1,i_2}| \geq \tau)^{\frac{1+2\epsilon}{2+2\epsilon}} \bigr]\\
		\leq& \Bigl(\mathbb{E}\Bigl| r_{t-l_1,i_1} r_{t-l_1,i_2} \Bigr|^{2+2\epsilon}\Bigr)^{\frac{1}{2+2\epsilon}} \biggl[\Bigl(\frac{\mathbb{E}\left|r_{t-l_1,i_1}\right|^{4+4\epsilon}}{\tau^{4+4\epsilon}}\Bigr)^\frac{1+2\epsilon}{2+2\epsilon} + \Bigl(\frac{\mathbb{E}\left|r_{t-l_1,i_2}\right|^{4+4\epsilon}}{\tau^{4+4\epsilon}}\Bigr)^\frac{1+2\epsilon}{2+2\epsilon} \biggr]\\
		\leq& M_{4+4\epsilon}^{\frac{1}{2+2\epsilon}} \cdot 2\biggl(\frac{M_{4+4\epsilon}}{\tau^{4+4\epsilon}}\biggr)^{\frac{1+2\epsilon}{2+2\epsilon}} = 2\frac{M_{4+4\epsilon}}{\tau^{2+4\epsilon}} \\
		\leq& \frac{2CM}{\tau^{4\epsilon}} \asymp \biggl(\frac{M^{1/\epsilon} \log(pd+1)}{T_{\textnormal{eff}}}\biggr)^{\frac{\epsilon}{1+\epsilon}}.
	\end{align*}
	As a result, we have
	\begin{equation}\label{eq:det_bound}
		\|\boldsymbol{\Gamma}_x(\tau) - \boldsymbol{\Gamma}_x\|_{\infty,\infty} \lesssim \biggl(\frac{M^{1/\epsilon} \log(pd+1)}{T_{\textnormal{eff}}}\biggr)^{\frac{\epsilon}{1+\epsilon}}.
	\end{equation}

	\textbf{Probabilistic bound for $\|\mathbf{S} - \boldsymbol{\Gamma}_x(\tau)\|_{\infty,\infty}$}: For the truncated $y_{t-l_1,i}(\tau)$ and $y_{t-l_2,j}(\tau)$, it can be verified that
	\begin{align*}
		& \mathbb{E}|y_{t-l_1,i}(\tau) y_{t-l_2,j}(\tau)|^2 = \mathbb{E}|r_{t-l_1,i_1}(\tau) r_{t-l_1,i_2}(\tau) r_{t-l_2,j_1}(\tau) r_{t-l_2,j_2}(\tau)|^2\\
		\leq& \tau^{4-4\epsilon}\mathbb{E}|r_{t-l_1,i_1}(\tau) r_{t-l_1,i_2}(\tau) r_{t-l_2,j_1}(\tau) r_{t-l_2,j_2}(\tau)|^{1+\epsilon}\\
		\leq& \tau^{4-4\epsilon} \bigl(\mathbb{E}\left|r_{t-l_1,i_1}(\tau)\right|^{4+4\epsilon} \mathbb{E}\left|r_{t-l_1,i_2}(\tau)\right|^{4+4\epsilon} \mathbb{E}\left|r_{t-l_2,j_1}(\tau)\right|^{4+4\epsilon} \mathbb{E}\left|r_{t-l_2,j_2}(\tau)\right|^{4+4\epsilon}\bigr)^{\frac{1}{4}}\\
		\leq& \tau^{4-4\epsilon} M_{4+4\epsilon} \leq \tau^{4-4\epsilon} M,
	\end{align*}
	and
	\begin{align*}
		& \mathbb{E}|y_{t-l_1,i}(\tau)|^2 = \mathbb{E}|r_{t-l_1,i_1}(\tau) r_{t-l_1,i_2}(\tau)|^2 
		\leq \tau^{2-2\epsilon}\mathbb{E}|r_{t-l_1,i_1}(\tau) r_{t-l_1,i_2}(\tau)|^{1+\epsilon}\\
		\leq& \tau^{2-2\epsilon} \bigl(\mathbb{E}\left|r_{t-l_1,i_1}(\tau)\right|^{4+4\epsilon} \mathbb{E}\left|r_{t-l_1,i_2}(\tau)\right|^{4+4\epsilon}\bigr)^{\frac{1}{4}} \leq \tau^{2-2\epsilon} M_{4+4\epsilon}^{1/2} \overset{(i)}{\leq} C\tau^{4-4\epsilon} M,
	\end{align*}
	where $M=\max\{M_{4+4\epsilon},1\}$. Here the last inequality $(i)$ is due to the assumption that $\tau \asymp (MT_{\textnormal{eff}}/\log(pd+1))^{1/(4+4\epsilon)}$ and $T \gtrsim \log(pd+1)$, such that $\tau^{2-2\epsilon} \to \infty$ for any $\epsilon \in (0,1]$, and thus there exists a finite constant $C>0$ such that $\tau^{2-2\epsilon} \ge 1/C$ uniformly in $T$.  
	Similarly, there exists a finite constant $C>0$ such that $\tau^{2(n-2)}\ge C^{-1}$ uniformly in $T$ for any fixed $n>2$. Then for any $n>2$, it holds that
	\begin{align*}
		&\mathbb{E}|y_{t-l_1,i}(\tau) y_{t-l_2,j}(\tau)|^n = \mathbb{E}|r_{t-l_1,i_1}(\tau) r_{t-l_1,i_2}(\tau) 
		r_{t-l_2,j_1}(\tau) r_{t-l_2,j_2}(\tau)|^n\\
		\leq& \tau^{4(n-2)}\mathbb{E}|r_{t-l_1,i_1}(\tau) r_{t-l_1,i_2}(\tau) 
		r_{t-l_2,j_1}(\tau) r_{t-l_2,j_2}(\tau)|^2 \leq \tau^{4(n-2)} \cdot \tau^{4-4\epsilon} M,
	\end{align*}
	and
	\begin{align*}
		&\mathbb{E}|y_{t-l_1,i}(\tau)|^n = \mathbb{E}|r_{t-l_1,i_1}(\tau) r_{t-l_1,i_2}(\tau)|^n \leq \tau^{2(n-2)}\mathbb{E}|r_{t-l_1,i_1}(\tau) r_{t-l_1,i_2}(\tau)|^2\\
		\leq& C\tau^{4(n-2)} \mathbb{E}|r_{t-l_1,i_1}(\tau) r_{t-l_1,i_2}(\tau)|^2 \leq C\tau^{4(n-2)} \cdot \tau^{4-4\epsilon} M.
	\end{align*}

	Since the operations of hard thresholding, vector outer product and lower-triangular vectorization preserves measurability. Consequently, by Lemma~\ref{lemma:measurable map} and Assumption \ref{assumption:process conditions}$(ii)$ that \(\{\boldsymbol{r}_t\}\) is \(\alpha\)-mixing with geometric mixing coefficients \(\alpha(\ell)=O(\zeta^\ell)\), the process \(\{\boldsymbol{y}_t(\tau)\}\) is also \(\alpha\)-mixing with geometric mixing coefficients. Furthermore, by Lemma \ref{lemma:lag process mixing}, $\boldsymbol{x}_t(\tau)=(1,\boldsymbol{y}_{t-1}^{\mathrm{T}}(\tau),\dots,\boldsymbol{y}_{t-p}^{\mathrm{T}}(\tau))^{\mathrm{T}}$ is also $\alpha$-mixing with geometric mixing coefficients as the lag order $p$ is fixed.
	Let $q \asymp T/\log(T)$ with $q \in [1,T/2]$ and $\delta \asymp (M^{1/\epsilon}\log(pd+1)/T_{\textnormal{eff}})^{\epsilon/(1+\epsilon)}$. Then, $\mu(\delta) = \delta^2/(\tau^4\delta+\tau^{4-4\epsilon}M) \asymp \log(pd+1)/T_{\textnormal{eff}} \ll \log(T)$. Note that Assumption \ref{assumption:process conditions}(ii) implies that $y_{t-l_1,i}(\tau) y_{t-l_2,j}(\tau)$ and $y_{t-l_1,i}(\tau)$ are $\alpha$-mixing with geometric mixing coefficients $\alpha(\ell)=O(\zeta^\ell)$. Then we have
	\begin{align*}
		\alpha\left(\left[\frac{T}{q+1}\right]\right)^{\frac{2n}{2n+1}} \asymp \zeta^{-C\log(T)} = \exp\left[-C\log(\zeta)\log(T)\right].
	\end{align*}
	Note that $n$, $\epsilon$ and $\zeta$ are constants independent of the sample size $T$ or dimension $d$. These together with $\tau \asymp (MT_{\textnormal{eff}}/\log(pd+1))^{1/(4+4\epsilon)}$, $T \gtrsim \log(pd+1)$ and the concentration inequality for $\alpha$-mixing processes in Lemma \ref{lemma:Concentration inequailty for alpha mixing process}, imply that 
	\begin{align*}
		&\mathbb{P}\left(\left|\frac{1}{T}\sum_{t=1}^{T}y_{t-l_1,i}(\tau) y_{t-l_2,j}(\tau) - \mathbb{E}\left[y_{t-l_1,i}(\tau) y_{t-l_2,j}(\tau)\right]\right| \geq \delta\right)\\
		\leq& 2[1+T / q+\mu(\delta)] \exp [-q \mu(\delta)]+11 T\left[1+5 \delta^{-1}\left(\mathbb{E} |y_{t-l_1,i}(\tau) y_{t-l_2,j}(\tau)|^n\right)^{\frac{1}{2n+1}}\right] \alpha\left(\left[\frac{T}{q+1}\right]\right)^{\frac{2n}{2n+1}}\\
		\asymp& \left(1+\log(T)+\frac{\log(pd+1)}{T_{\textnormal{eff}}}\right)\exp\left[-C\frac{T}{\log(T)} \cdot \frac{\log(pd+1)}{T_{\textnormal{eff}}}\right]\\
		& + T\left[1 + \left(\frac{T_{\textnormal{eff}}}{M^{1/\epsilon}\log(pd+1)}\right)^{\frac{\epsilon}{1+\epsilon}}\left(\mathbb{E} |y_{t-l_1,i}(\tau) y_{t-l_2,j}(\tau)|^n\right)^{\frac{1}{2n+1}}\right]\exp\left[-C\log(\zeta)\log(T)\right]\\
		\leq& \left(1+\log(T)\log(pd+1)+\frac{\log(pd+1)}{T_{\textnormal{eff}}}\right)\exp\left[-C\log(T)\log(pd+1)\right]\\
		& + T\left[1 + \left(\frac{T_{\textnormal{eff}}}{M^{1/\epsilon}\log(pd+1)}\right)^{\frac{\epsilon}{1+\epsilon}}\cdot\left(C\tau^{4(n-\epsilon-1)} M\right)^{\frac{1}{2n+1}}\right]\exp\left[-C\log(\zeta)\log(T)\right]\\
		\asymp&\log(T)\log(pd+1)\exp\left[-C\log(T)\log(pd+1)\right] + T \left(\frac{T_{\textnormal{eff}}}{\log(pd+1)}\right)^{\frac{n-1}{1+\epsilon}}\exp\left[-C\log(\zeta)\log(T)\right]\\
		\lesssim& \exp\left[\log(\log(T)\log(pd+1))-C\log(T)\log(pd+1)\right] + \exp\left(\left(\frac{n+\epsilon}{1+\epsilon} -C\log(\zeta)\right)\log(T)\right)\\
		\leq& C\exp[-C\log(T)\log(pd+1)],
	\end{align*}
	and
	\begin{align*}
		&\mathbb{P}\left(\left|\frac{1}{T}\sum_{t=1}^{T}y_{t-l_1,i}(\tau) - \mathbb{E}\left[y_{t-l_1,i}(\tau)\right]\right| \geq \delta\right)\\
		\leq& 2[1+T / q+\mu(\delta)] \exp [-q \mu(\delta)]+11 T\left[1+5 \delta^{-1}\left(\mathbb{E} |y_{t-l_1,i}(\tau)|^n\right)^{\frac{1}{2n+1}}\right] \alpha\left(\left[\frac{T}{q+1}\right]\right)^{\frac{2n}{2n+1}}\\
		\asymp& \left(1+\log(T)+\frac{\log(pd+1)}{T_{\textnormal{eff}}}\right)\exp\left[-C\frac{T}{\log(T)} \cdot \frac{\log(pd+1)}{T_{\textnormal{eff}}}\right]\\
		& + T\left[1 + \left(\frac{T_{\textnormal{eff}}}{M^{1/\epsilon}\log(pd+1)}\right)^{\frac{\epsilon}{1+\epsilon}}\left(\mathbb{E} |y_{t-l_1,i}(\tau)|^n\right)^{\frac{1}{2n+1}}\right]\exp\left[-C\log(\zeta)\log(T)\right]\\
		\leq& \left(1+\log(T)\log(pd+1)+\frac{\log(pd+1)}{T_{\textnormal{eff}}}\right)\exp\left[-C\log(T)\log(pd+1)\right]\\
		& + T\left[1 + \left(\frac{T_{\textnormal{eff}}}{M^{1/\epsilon}\log(pd+1)}\right)^{\frac{\epsilon}{1+\epsilon}}\cdot\left(C\tau^{4(n-\epsilon-1)} M\right)^{\frac{1}{2n+1}}\right]\exp\left[-C\log(\zeta)\log(T)\right]\\
		\asymp& \log(T)\log(pd+1)\exp\left[-C\log(T)\log(pd+1)\right] + T \left(\frac{T_{\textnormal{eff}}}{\log(pd+1)}\right)^{\frac{n-1}{1+\epsilon}}\exp\left[-C\log(\zeta)\log(T)\right]\\
		\lesssim& \exp\left[\log(\log(T)\log(pd+1))-C\log(T)\log(pd+1)\right] + \exp\left(\left(\frac{n+\epsilon}{1+\epsilon} -C\log(\zeta)\right)\log(T)\right)\\
		\leq& C\exp[-C\log(T)\log(pd+1)].
	\end{align*}
	Under slightly abuse of notation, for $1 \leq i,j \leq pd+1$,
	\begin{align*}
		&\mathbb{P}\left(\max_{1 \leq i,j \leq pd+1}\left|S_{ij}-\Gamma_x(\tau)_{ij}\right| \geq C\left(\frac{M^{1/\epsilon}\log(pd+1)}{T_{\textnormal{eff}}}\right)^{\frac{\epsilon}{1+\epsilon}}\right)\\
		\leq& \sum_{\substack{1 \leq i \leq pd+1\\ 1 \leq j \leq pd+1}}\mathbb{P}\left(\left|S_{ij}-\Gamma_x(\tau)_{ij}\right| \geq C\left(\frac{M^{1/\epsilon}\log(pd+1)}{T_{\textnormal{eff}}}\right)^{\frac{\epsilon}{1+\epsilon}}\right)\\
		\leq& C(pd+1)^2\exp\left[-C\log(T)\log(pd+1)\right] = C\exp\left[2\log(pd+1)-C\log(T)\log(pd+1)\right]\\
		\leq& C\exp\left[-C\log(T)\log(pd+1)\right].
	\end{align*}
	Therefore, with probability at least $1-C\exp\left[-C\log(T)\log(pd+1)\right]$,
	\begin{align}\label{eq:prob_bound}
		\|\mathbf{S} - \boldsymbol{\Gamma}_x(\tau)\|_{\infty,\infty} \lesssim \biggl(\frac{M^{1/\epsilon} \log(pd+1)}{T_{\textnormal{eff}}}\biggr)^{\frac{\epsilon}{1+\epsilon}}.
	\end{align}
	Hence, the lemma follows from \eqref{eq:det_bound} and \eqref{eq:prob_bound} by the triangle inequality, which concludes the proof.
	\end{proof}

	\begin{proof}[\textbf{Proof of Lemma \ref{lemma:first order derivative}}]
		The loss function given $\bm{\Theta}^*$ can be written as
		\begin{align*}
		\mathbb{L}(\bm{\Theta}^*) = \frac{1}{2T}\sum_{t=1}^{T}\sum_{i=1}^{d}\bigl(y_{t,i}(\tau)-\bbm{x}_t^{\mathrm{T}}(\tau)\bm{\Theta}_{\bm{\cdot},i}^*\bigr)^2.
		\end{align*}
		For $1 \leq l \leq p$, $1 \leq j,k \leq d$, we have 
		\begin{align*}
			\{\nabla \mathbb{L}(\bm{\Theta}^*)\}_{(l-1)d+j+1,k} = -\frac{1}{T}\sum_{t=1}^{T}y_{t-l,j}(\tau)\bigl(y_{t,k}(\tau)-\boldsymbol{x}_t^{\mathrm{T}}(\tau)\bm{\Theta}_{\bm{\cdot},k}^*\bigr)
		\end{align*}
		and
		\begin{align*}
			\{\nabla \mathbb{L}(\bm{\Theta}^*)\}_{1,k} = -\frac{1}{T}\sum_{t=1}^{T}y_{t,k}(\tau)-\boldsymbol{x}_t^{\mathrm{T}}(\tau)\bm{\Theta}_{\bm{\cdot},k}^*.
		\end{align*}
		Note that $\varepsilon_{t,k} = y_{t,k}(\tau)-\boldsymbol{x}_t^{\mathrm{T}}(\tau)\bm{\Theta}_{\bm{\cdot},k}^*$. It follows that
		\begin{align*}
			|\varepsilon_{t,k}| = |y_{t,k}(\tau)-\boldsymbol{x}_t^{\mathrm{T}}(\tau)\bm{\Theta}_{\bm{\cdot},k}^*| \leq |y_{t,k}(\tau)|+|\boldsymbol{x}_t^{\mathrm{T}}(\tau)\bm{\Theta}_{\bm{\cdot},k}^*| \leq \tau^2(1+\|\bm{\Theta}^*\|_{1,\infty})
		\end{align*}
		and
		\begin{align}\label{eq:varepsilon}
			\varepsilon_{t,k} & = (y_{t,k}(\tau) - y_{t,k}) + (y_{t,k} - \langle\boldsymbol{x}_t,\bm{\Theta}_{\bm{\cdot},k}^*\rangle) + \langle\boldsymbol{x}_t - \boldsymbol{x}_t(\tau),\bm{\Theta}_{\bm{\cdot},k}^*\rangle \notag\\
			&= (y_{t,k}(\tau) - y_{t,k}) + \langle\boldsymbol{x}_t - \boldsymbol{x}_t(\tau),\bm{\Theta}_{\bm{\cdot},k}^*\rangle + e_{t,k}.
		\end{align}
		Note that $\mathbb{E}[e_{t,k}] = 0$. Taking expectations on both sides of equation \eqref{eq:varepsilon}, we have
		\begin{align*}
			&\mathbb{E}\left[\varepsilon_{t,k}\right] = \mathbb{E}\left[y_{t,k}(\tau) - y_{t,k}\right] +\mathbb{E}\left[\langle\boldsymbol{x}_t - \boldsymbol{x}_t(\tau),\bm{\Theta}_{\cdot,k}^*\rangle\right].
		\end{align*}
		Multiply both sides of equation \eqref{eq:varepsilon} by $y_{t-l,j}(\tau)$ and then take the expectation. Considering that $\mathbb{E}[e_{t,k} \mid \mathcal{F}_{t-1}] = 0$ and using the iterated expectation theorem, we have
		\begin{align*}
			&\mathbb{E}\left[y_{t-l,j}(\tau)\varepsilon_{t,k}\right] = \mathbb{E}\left[y_{t-l,j}(\tau)(y_{t,k}(\tau) - y_{t,k})\right] + \mathbb{E}\left[y_{t-l,j}(\tau)\langle\boldsymbol{x}_t - \boldsymbol{x}_t(\tau),\bm{\Theta}_{\cdot,k}^*\rangle\right].
		\end{align*}
		Thus, by the triangle inequaility, we have
		\begin{align*}
			&\left|\mathbb{E}\left[\varepsilon_{t,k}\right]\right| \leq \underbrace{\left|\mathbb{E}\left[y_{t,k}(\tau) - y_{t,k}\right]\right|}_{P_1'} + \underbrace{\left|\mathbb{E}\left[\langle\boldsymbol{x}_t - \boldsymbol{x}_t(\tau),\bm{\Theta}_{\cdot,k}^*\rangle\right]\right|}_{P_2'}
		\end{align*}
		and
		\begin{align*}
			&\left|\mathbb{E}\left[y_{t-l,j}(\tau)\varepsilon_{t,k}\right]\right| \leq \underbrace{\left|\mathbb{E}\left[y_{t-l,j}(\tau)(y_{t,k}(\tau) - y_{t,k})\right]\right|}_{P_1} + \underbrace{\left|\mathbb{E}\left[y_{t-l,j}(\tau)\langle\boldsymbol{x}_t - \boldsymbol{x}_t(\tau),\bm{\Theta}_{\cdot,k}^*\rangle\right]\right|}_{P_2}.
		\end{align*}

		The proof of this lemma consists of the following two parts: deriving the deterministic upper bound for $\|\mathbb{E}(\nabla \mathbb{L}(\bm{\Theta}^*))\|_{\infty,\infty}$, and establishing the probabilistic upper bound for $\|\nabla \mathbb{L}(\bm{\Theta}^*) -  \mathbb{E}(\nabla\mathbb{L}(\bm{\Theta}^*))\|_{\infty,\infty}$. Without loss of generality, we assume $y_{t,k}(\tau) = r_{t,k_1}(\tau)r_{t,k_2}(\tau)$, $y_{t-l,j}(\tau) = r_{t-l,j_1}(\tau)r_{t-l,j_2}(\tau)$, and $x_{t,i}(\tau) = y_{t-f(i),g(i)}(\tau) = r_{t-f(i),g_1(i)}(\tau)r_{t-f(i),g_2(i)}(\tau)$ for $i > 1$. Let $\mathcal{S}_j$ denote the support set of $\bm{\Theta}_{\bm{\cdot},j}^*$, with cardinalities $|\mathcal{S}_j| = s_j \leq s$ for each $j$. Here we focus on the case where $1 \notin \mathcal{S}_j$, noting that $x_{t,1} = x_{t,1}(\tau) = 1$ and thus the upper bounds for $P_2$ and $P_2'$ given $1 \in \mathcal{S}_j$ coincide with those in the case where $1 \notin \mathcal{S}_j$. Next, we provide the upper bounds for $P_1$, $P_1'$, $P_2$ and $P_2'$ respectively.
		\begin{align*}
			&P_1 = \left|\mathbb{E}\left[y_{t-l,j}(\tau)(y_{t,k}(\tau) - y_{t,k})\right]\right| = \left|\mathbb{E}\left[y_{t-l,j}(\tau)(y_{t,k}(\tau) - y_{t,k})\left(\mathds{1}(|r_{t,k_1}|>\tau)+\mathds{1}(|r_{t,k_2}|>\tau)\right)\right]\right|\\
			\leq& \mathbb{E}\left[|r_{t-l,j_1}(\tau)r_{t-l,j_2}(\tau)r_{t,k_1}r_{t,k_2}|\left(\mathds{1}(|r_{t,k_1}|>\tau)+\mathds{1}(|r_{t,k_2}|>\tau)\right)\right]\\
			\leq& \mathbb{E}\left[|r_{t,k_1}r_{t,k_2}r_{t-l,j_1}(\tau)r_{t-l,j_2}(\tau)|^{1+\epsilon}\right]^\frac{1}{1+\epsilon}\left(\mathbb{P}(|r_{t,k_1}|>\tau)^\frac{\epsilon}{1+\epsilon}+\mathbb{P}(|r_{t,k_2}|>\tau)^\frac{\epsilon}{1+\epsilon}\right)\\
			\leq& M_{4+4\epsilon}^{\frac{1}{1+\epsilon}} \cdot 2\left(\frac{M_{4+4\epsilon}}{\tau^{4+4\epsilon}}\right)^\frac{\epsilon}{1+\epsilon} = \frac{2M_{4+4\epsilon}}{\tau^{4\epsilon}} \leq \frac{2M}{\tau^{4\epsilon}} \asymp \left(\frac{M^{1/\epsilon}\log(pd+1)}{T_{\textnormal{eff}}}\right)^{\frac{\epsilon}{1+\epsilon}},\\
			\\
			&P_1' = \left|\mathbb{E}\left[y_{t,k}(\tau) - y_{t,k}\right]\right| = \left|\mathbb{E}\left[(y_{t,k}(\tau) - y_{t,k})\left(\mathds{1}(|r_{t,k_1}|>\tau)+\mathds{1}(|r_{t,k_2}|>\tau)\right)\right]\right|\\
			\leq& \mathbb{E}\left[|r_{t,k_1}r_{t,k_2}|\left(\mathds{1}(|r_{t,k_1}|>\tau)+\mathds{1}(|r_{t,k_2}|>\tau)\right)\right]\\
			\leq& \mathbb{E}\left[|r_{t,k_1}r_{t,k_2}|^{2+2\epsilon}\right]^\frac{1}{2+2\epsilon}\left(\mathbb{P}(|r_{t,k_1}|>\tau)^\frac{1+2\epsilon}{2+2\epsilon}+\mathbb{P}(|r_{t,k_2}|>\tau)^\frac{1+2\epsilon}{2+2\epsilon}\right)\\ 
			\leq& M_{4+4\epsilon}^{\frac{1}{2+2\epsilon}} \cdot 2\left(\frac{M_{4+4\epsilon}}{\tau^{4+4\epsilon}}\right)^\frac{1+2\epsilon}{2+2\epsilon} = \frac{2M_{4+4\epsilon}}{\tau^{2+4\epsilon}} \overset{(i)}{\leq} \frac{2CM}{\tau^{4\epsilon}} \asymp \left(\frac{M^{1/\epsilon}\log(pd+1)}{T_{\textnormal{eff}}}\right)^{\frac{\epsilon}{1+\epsilon}},\\
			\\
			&P_2 = \left|\mathbb{E}\left[y_{t-l,j}(\tau)\langle\boldsymbol{x}_t - \boldsymbol{x}_t(\tau),\bm{\Theta}_{\cdot,j}^*\rangle\right]\right| \leq \mathbb{E}\left[\left|y_{t-l,j}(\tau)\langle(\boldsymbol{x}_t - \boldsymbol{x}_t(\tau))_{\mathcal{S}_j},(\bm{\Theta}_{\cdot,j}^*)_{\mathcal{S}_j}\rangle\right|\right]\\
			\leq& \mathbb{E}\left[\left|y_{t-l,j}(\tau)\|(\boldsymbol{x}_t - \boldsymbol{x}_t(\tau))_{\mathcal{S}_j}\|_\infty\|(\bm{\Theta}_{\cdot,j}^*)_{\mathcal{S}_j}\|_1\right|\right] \leq \|\bm{\Theta}^*\|_{1,\infty}\sum_{i \in \mathcal{S}_j}\mathbb{E}\left[\left|y_{t-l,j}(\tau)\left(x_{t,i} - x_{t,i}(\tau)\right)\right|\right]\\ 
			\leq& \|\bm{\Theta}^*\|_{1,\infty}\sum_{i \in \mathcal{S}_j}\mathbb{E}\left[\left|y_{t-l,j}(\tau)\left(r_{t-f(i),g_1(i)}r_{t-f(i),g_2(i)} - r_{t-f(i),g_1(i)}(\tau)r_{t-f(i),g_2(i)}(\tau)\right)\right|\right]\\
			\leq& \|\bm{\Theta}^*\|_{1,\infty}\sum_{i \in \mathcal{S}_j}\mathbb{E}\left[\left|y_{t-l,j}(\tau)r_{t-f(i),g_1(i)}r_{t-f(i),g_2(i)}\right|\left(\mathds{1}(|r_{t-f(i),g_1(i)}|>\tau)+\mathds{1}(|r_{t-f(i),g_2(i)}|>\tau)\right)\right]\\
			\leq& \|\bm{\Theta}^*\|_{1,\infty}\sum_{i \in \mathcal{S}_j}\mathbb{E}\left[\left|r_{t-l,j_1}(\tau)r_{t-l,j_2}(\tau)r_{t-f(i),g_1(i)}r_{t-f(i),g_2(i)}\right|^{1+\epsilon}\right]^\frac{1}{1+\epsilon}\\
			&\qquad\qquad\qquad\qquad\qquad\qquad\qquad\qquad\qquad \cdot\left(\mathbb{P}(|r_{t-f(i),g_1(i)}|>\tau)^\frac{\epsilon}{1+\epsilon}+\mathbb{P}(|r_{t-f(i),g_2(i)}|>\tau)^\frac{\epsilon}{1+\epsilon}\right)\\
			\leq& \|\bm{\Theta}^*\|_{1,\infty}\sum_{i \in \mathcal{S}_j}\frac{2M_{4+4\epsilon}}{\tau^{4\epsilon}}\leq \frac{2s_j\|\bm{\Theta}^*\|_{1,\infty}M}{\tau^{4\epsilon}} \leq \frac{2s\|\bm{\Theta}^*\|_{1,\infty}M}{\tau^{4\epsilon}} \asymp s\|\bm{\Theta}^*\|_{1,\infty}\left(\frac{M^{1/\epsilon}\log(pd+1)}{T_{\textnormal{eff}}}\right)^{\frac{\epsilon}{1+\epsilon}}\\
			\text{and}\\
			&P_2' = \left|\mathbb{E}\left[\langle\boldsymbol{x}_t - \boldsymbol{x}_t(\tau),\bm{\Theta}_{\cdot,j}^*\rangle\right]\right| \leq \mathbb{E}\left[\left|\langle(\boldsymbol{x}_t - \boldsymbol{x}_t(\tau))_{\mathcal{S}_j},(\bm{\Theta}_{\cdot,j}^*)_{\mathcal{S}_j}\rangle\right|\right]\\
			\leq& \|\bm{\Theta}^*\|_{1,\infty}\sum_{i \in \mathcal{S}_j}\mathbb{E}\left[\left|\left(r_{t-f(i),g_1(i)}r_{t-f(i),g_2(i)} - r_{t-f(i),g_1(i)}(\tau)r_{t-f(i),g_2(i)}(\tau)\right)\right|\right]\\
			\leq& \|\bm{\Theta}^*\|_{1,\infty}\sum_{i \in \mathcal{S}_j}\mathbb{E}\left[\left|\left(r_{t-f(i),g_1(i)}r_{t-f(i),g_2(i)}\right)\right|\left(\mathds{1}(|r_{t-f(i),g_1(i)}|>\tau)+\mathds{1}(|r_{t-f(i),g_2(i)}|>\tau)\right)\right]\\
			\leq& \|\bm{\Theta}^*\|_{1,\infty}\sum_{i \in \mathcal{S}_j}\mathbb{E}\left[\left|r_{t-f(i),g_1(i)}r_{t-f(i),g_2(i)}\right|^{2+2\epsilon}\right]^\frac{1}{2+2\epsilon}\left(\mathbb{P}(|r_{t-f(i),g_1(i)}|>\tau)^\frac{1+2\epsilon}{2+2\epsilon}+\mathbb{P}(|r_{t-f(i),g_2(i)}|>\tau)^\frac{1+2\epsilon}{2+2\epsilon}\right)\\
			\leq& \|\bm{\Theta}^*\|_{1,\infty}\sum_{i \in \mathcal{S}_j}\frac{2M_{4+4\epsilon}}{\tau^{2+4\epsilon}} \leq \frac{2s_j\|\bm{\Theta}^*\|_{1,\infty}M}{\tau^{2+4\epsilon}} \overset{(ii)}{\leq} \frac{2Cs\|\bm{\Theta}^*\|_{1,\infty}M}{\tau^{4\epsilon}} \asymp s\|\bm{\Theta}^*\|_{1,\infty}\left(\frac{M^{1/\epsilon}\log(pd+1)}{T_{\textnormal{eff}}}\right)^{\frac{\epsilon}{1+\epsilon}}.
		\end{align*}
		The inequalities $(i)$ and $(ii)$ in $P_1'$ and $P_2'$ hold by the same arguments as in the deterministic bound in the proof of Lemma \ref{lemma:covariance matrix bound}.
		Combining all the above inequalities, we have
		\begin{align*}
			&\left|\mathbb{E}\left[y_{t-l,j}(\tau)\varepsilon_{t,k}\right]\right| \lesssim s\|\bm{\Theta}^*\|_{1,\infty}\left(\frac{M^{1/\epsilon}\log(pd+1)}{T_{\textnormal{eff}}}\right)^{\frac{\epsilon}{1+\epsilon}} \; \text{and} \;\left|\mathbb{E}\left[\varepsilon_{t,k}\right]\right| \lesssim s\|\bm{\Theta}^*\|_{1,\infty}\left(\frac{M^{1/\epsilon}\log(pd+1)}{T_{\textnormal{eff}}}\right)^{\frac{\epsilon}{1+\epsilon}}.
		\end{align*}
		Thus, the \textbf{deterministic bound} for $\|\mathbb{E}(\nabla \mathbb{L}(\bm{\Theta}^*))\|_{\infty,\infty}$ follows as:
		\begin{align}\label{eq:deterministic_bound_nabla}
			&\|\mathbb{E}(\nabla \mathbb{L}(\bm{\Theta}^*))\|_{\infty,\infty} \lesssim s\|\bm{\Theta}^*\|_{1,\infty}\left(\frac{M^{1/\epsilon}\log(pd+1)}{T_{\textnormal{eff}}}\right)^{\frac{\epsilon}{1+\epsilon}}.
		\end{align}
		Next, we show the \textbf{probabilistic bound} for $\|\nabla \mathbb{L}(\bm{\Theta}^*) - \mathbb{E}(\nabla \mathbb{L}(\bm{\Theta}^*))\|_{\infty,\infty}$. Here we only focus on the case where $1 \in \mathcal{S}_j$ for simplicity, noting that similar bounds can be derived for the case where $1 \notin \mathcal{S}_j$. By $|a+b|^{1+\epsilon} \leq 2^\epsilon(|a|^{1+\epsilon}+|b|^{1+\epsilon})$ for $\epsilon \in (0,1]$, it follows that
		\begin{align*}
			&\mathbb{E}\left|y_{t-l,j}(\tau)\varepsilon_{t,k}\right|^2 \leq (\tau^2)^{1-\epsilon}\left(\tau^2(1+\|\bm{\Theta}^*\|_{1,\infty})\right)^{1-\epsilon}\mathbb{E}\left[\left|y_{t-l,j}(\tau)\varepsilon_{t,k}\right|^{1+\epsilon}\right]\\
			\leq& \left(\tau^4(1+\|\bm{\Theta}^*\|_{1,\infty})\right)^{1-\epsilon}\mathbb{E}\left[\left|y_{t-l,j}\right|^{1+\epsilon}|y_{t,k}(\tau) - \langle\boldsymbol{x}_t(\tau),\bm{\Theta}_{\cdot,j}^*\rangle|^{1+\epsilon}\right]\\
			\leq& 2^\epsilon\left(\tau^4(1+\|\bm{\Theta}^*\|_{1,\infty})\right)^{1-\epsilon}\left\{\mathbb{E}\left|y_{t-l,j}y_{t,k}(\tau)\right|^{1+\epsilon} + \mathbb{E}\left|y_{t-l,j}\langle\boldsymbol{x}_t(\tau)_{\mathcal{S}_j}, (\bm{\Theta}_{\cdot,j}^*)_{\mathcal{S}_j}\rangle\right|^{1+\epsilon}\right\}\\
			\leq& 2^\epsilon\left(\tau^4(1+\|\bm{\Theta}^*\|_{1,\infty})\right)^{1-\epsilon}\biggl\{\mathbb{E}\left|r_{t-l,j_1}r_{t-l,j_2}r_{t,k_1}(\tau)r_{t,k_2}(\tau)\right|^{1+\epsilon} + \|\bm{\Theta}^*\|_{1,\infty}^{1+\epsilon}\mathbb{E}\left|r_{t-l,j_1}r_{t-l,j_2}\right|^{1+\epsilon}\\
			&\qquad\qquad\qquad\qquad\qquad\qquad\qquad+ \|\bm{\Theta}^*\|_{1,\infty}^{1+\epsilon}\sum_{\substack{i \in \mathcal{S}_j\\ i \neq 1}}\mathbb{E}\left|r_{t-l,j_1}r_{t-l,j_2}r_{t-f(i),g_1(i)}(\tau)r_{t-f(i),g_2(i)}(\tau)\right|^{1+\epsilon}\biggr\}\\
			\leq& 2^{2\epsilon}\left(\tau^4(1+\|\bm{\Theta}^*\|_{1,\infty})\right)^{1-\epsilon}\biggl(M_{4+4\epsilon}+\|\bm{\Theta}^*\|_{\infty,\infty}^{1+\epsilon}M_{4+4\epsilon}^{1/2} +\|\bm{\Theta}^*\|_{\infty,\infty}^{1+\epsilon} \sum_{\substack{i \in \mathcal{S}_j\\ i \neq 1}}M_{4+4\epsilon}\biggr) \leq C s\|\bm{\Theta}^*\|_{1,\infty}^2M\tau^{4-4\epsilon}\\
			\text{and} \\
			&\mathbb{E}\left|\varepsilon_{t,k}\right|^2 \leq \left(\tau^2(1+\|\bm{\Theta}^*\|_{1,\infty})\right)^{1-\epsilon} \mathbb{E}\left|\varepsilon_{t,k}\right|^{1+\epsilon} \leq \left(\tau^2(1+\|\bm{\Theta}^*\|_{1,\infty})\right)^{1-\epsilon} \mathbb{E}\left|y_{t,k}(\tau) - \langle\boldsymbol{x}_t(\tau), \bm{\Theta}_{\cdot,j}^*\rangle\right|^{1+\epsilon}\\
			\leq& 2^{\epsilon}\left(\tau^2(1+\|\bm{\Theta}^*\|_{1,\infty})\right)^{1-\epsilon}\left\{\mathbb{E}\left|y_{t,k}(\tau)\right|^{1+\epsilon} + \sum_{i \in \mathcal{S}_j}\mathbb{E}\left|\langle\boldsymbol{x}_t(\tau)_{\mathcal{S}_j}, (\bm{\Theta}_{\cdot,j}^*)_{\mathcal{S}_j}\rangle\right|^{1+\epsilon}\right\}\\
			\overset{(i)}{\leq} &2^{\epsilon}\left(\tau^4(1+\|\bm{\Theta}^*\|_{1,\infty})\right)^{1-\epsilon}\left\{\mathbb{E}\left|y_{t,k}(\tau)\right|^{1+\epsilon} + \sum_{i \in \mathcal{S}_j}\mathbb{E}\left|\langle\boldsymbol{x}_t(\tau)_{\mathcal{S}_j}, (\bm{\Theta}_{\cdot,j}^*)_{\mathcal{S}_j}\rangle\right|^{1+\epsilon}\right\}\\
			\leq & 2^\epsilon\left(\tau^4(1+\|\bm{\Theta}^*\|_{1,\infty})\right)^{1-\epsilon}\biggl\{\mathbb{E}\left|r_{t,k_1}(\tau)r_{t,k_2}(\tau)\right|^{1+\epsilon} + \|\bm{\Theta}^*\|_{1,\infty}^{1+\epsilon}\\
			&\qquad\qquad\qquad\qquad\qquad\qquad\qquad + \|\bm{\Theta}^*\|_{1,\infty}^{1+\epsilon}\sum_{\substack{i \in \mathcal{S}_j\\ i \neq 1}}\mathbb{E}\left|r_{t-f(i),g_1(i)}(\tau)r_{t-f(i),g_2(i)}(\tau)\right|^{1+\epsilon}\biggr\}\\
			\leq& 2^{\epsilon} \left(\tau^4(1+\|\bm{\Theta}^*\|_{1,\infty})\right)^{1-\epsilon}\biggl(M_{4+4\epsilon}^{1/2} + \|\bm{\Theta}^*\|_{1,\infty}^{1+\epsilon} + \|\bm{\Theta}^*\|_{1,\infty}^{1+\epsilon}\sum_{\substack{i \in \mathcal{S}_j\\ i \neq 1}}M_{4+4\epsilon}^{1/2}\biggr) \leq C s\|\bm{\Theta}^*\|_{1,\infty}^2M\tau^{4-4\epsilon},
		\end{align*}
		where the inequalities $(i)$ and $(ii)$ hold by the same arguments as in the probabilistic bound in the proof of Lemma \ref{lemma:covariance matrix bound}. 
		Also, for any $n=3,4,\dots$, the higher-order moments of $y_{t-l,j}(\tau)\varepsilon_{t,k}$ and $\varepsilon_{t,k}$ satisfy
		\begin{align*}
			&\mathbb{E}|y_{t-l,j}(\tau)\varepsilon_{t,k}|^n \leq (\tau^2 \cdot \tau^2(1+\|\bm{\Theta}^*\|_{1,\infty}))^{n-2}\mathbb{E}|y_{t-l,j}(\tau)\varepsilon_{t,k}|^2 \leq (C \tau^4\|\bm{\Theta}^*\|_{1,\infty})^{n-2}\mathbb{E}|y_{t-l,j}(\tau)\varepsilon_{t,k}|^2
		\end{align*}
		and
		\begin{align*}
			&\mathbb{E}|\varepsilon_{t,k}|^n \leq (\tau^2(1+\|\bm{\Theta}^*\|_{1,\infty}))^{n-2}\mathbb{E}|\varepsilon_{t,k}|^2 \leq (C \tau^2\|\bm{\Theta}^*\|_{1,\infty})^{n-2}\mathbb{E}|\varepsilon_{t,k}|^2 \overset{(ii)}{\leq} (C \tau^4\|\bm{\Theta}^*\|_{1,\infty})^{n-2}\mathbb{E}|\varepsilon_{t,k}|^2.
		\end{align*}
		Similar to the proof in Lemma \ref{lemma:covariance matrix bound}, we can show that both $\varepsilon_{t,k}$ and $y_{t-l,j}(\tau)\varepsilon_{t,k}$ form $\alpha$-mixing processes with geometric mixing coefficients. Let $q \asymp T/\log(T)$ with $q \in [1,T/2]$ and $\delta \asymp s\|\bm{\Theta}^*\|_{1,\infty}\left(M^{1/\epsilon}\log(pd+1)/T_{\textnormal{eff}}\right)^{\epsilon/(1+\epsilon)}$. Then, $\mu(\delta) = \delta^2/(\tau^4\|\bm{\Theta}^*\|_{1,\infty}\delta+s\|\bm{\Theta}^*\|_{1,\infty}^2M\tau^{4-4\epsilon}) \asymp s\log(pd+1)/T_{\textnormal{eff}} \ll \log(T)$. Assumption \ref{assumption:process conditions}(ii) together with Lemma \ref{lemma:measurable map}, implies that $y_{t-l_1,i}(\tau)$ and $y_{t-l_1,i}(\tau) y_{t-l_2,j}(\tau)$ are $\alpha$-mixing with geometric mixing coefficients $\alpha(\ell)=O(\zeta^\ell)$. Then we have
		\begin{align*}
			\alpha\left(\left[\frac{T}{q+1}\right]\right)^{\frac{2n}{2n+1}} \asymp \zeta^{-C\log(T)} = \exp\left[-C\log(\zeta)\log(T)\right].
		\end{align*}
		Note that $n$, $\epsilon$, $\zeta$, $s$ and $\|\bm{\Theta}^*\|_{1,\infty}$ are constants independent of the sample size $T$ or dimension $d$. Thus, by $\tau \asymp (MT_{\textnormal{eff}}/\log(pd+1))^{1/(4+4\epsilon)}$, $T \gtrsim \log(pd+1)$ and the concentration inequality for $\alpha$-mixing processes in Lemma \ref{lemma:Concentration inequailty for alpha mixing process}, we have
		\begin{align*}
			&\mathbb{P}\left(\left|\frac1T\sum_{t=1}^{T}y_{t-l,j}(\tau)\varepsilon_{t,k} - \mathbb{E}\left[y_{t-l,j}(\tau)\varepsilon_{t,k}\right]\right| > \delta\right)\\
			\leq& 2[1+T / q+\mu(\delta)] \exp [-q \mu(\delta)]+11 T\left[1+5 \delta^{-1}\left(\mathbb{E} |y_{t-l,j}(\tau)\varepsilon_{t,k}|^n\right)^{\frac{1}{2n+1}}\right] \alpha\left(\left[\frac{T}{q+1}\right]\right)^{\frac{2n}{2n+1}}\\
			\asymp& \left(1+\log(T)+\frac{s\log(pd+1)}{T_{\textnormal{eff}}}\right)\exp\left[-C\frac{T}{\log(T)} \cdot \frac{s\log(pd+1)}{T_{\textnormal{eff}}}\right]\\
			& + T\left[1 + (s\|\bm{\Theta}^*\|_{1,\infty})^{-1}\left(\frac{T_{\textnormal{eff}}}{M^{1/\epsilon}\log(pd+1)}\right)^{\frac{\epsilon}{1+\epsilon}}\left(\mathbb{E} |y_{t-l,j}(\tau)\varepsilon_{t,k}|^n\right)^{\frac{1}{2n+1}}\right]\exp\left[-C\log(\zeta)\log(T)\right]\\
			\leq& \left(1+\log(T)\log(pd+1)+\frac{s\log(pd+1)}{T_{\textnormal{eff}}}\right)\exp\left[-Cs\log(T)\log(pd+1)\right]\\
			& + T\left[1 + (s\|\bm{\Theta}^*\|_{1,\infty})^{-1}\left(\frac{T_{\textnormal{eff}}}{M^{1/\epsilon}\log(pd+1)}\right)^{\frac{\epsilon}{1+\epsilon}}\cdot\left(Cs\|\bm{\Theta}^*\|_{1,\infty}^n\tau^{4(n-\epsilon-1)} M\right)^{\frac{1}{2n+1}}\right]\exp\left[-C\log(\zeta)\log(T)\right]\\
			\asymp& \log(T)\log(pd+1)\exp\left[-Cs\log(T)\log(pd+1)\right] + T \left(\frac{T_{\textnormal{eff}}}{\log(pd+1)}\right)^{\frac{n-1}{1+\epsilon}}\exp\left[-C\log(\zeta)\log(T)\right]\\
			\lesssim& \exp\left[\log(\log(T)\log(pd+1))-Cs\log(T)\log(pd+1)\right] + \exp\left(\left(\frac{n+\epsilon}{1+\epsilon} -C\log(\zeta)\right)\log(T)\right)\\
			\leq& C\exp[-Cs\log(T)\log(pd+1)]
		\end{align*}
		and
		\begin{align*}			
			&\mathbb{P}\left(\left|\frac1T\sum_{t=1}^{T}\varepsilon_{t,k} - \mathbb{E}\left[\varepsilon_{t,k}\right]\right| > \delta\right)\\
			\leq& 2[1+T / q+\mu(\delta)] \exp [-q \mu(\delta)]+11 T\left[1+5 \delta^{-1}\left(\mathbb{E} |\varepsilon_{t,k}|^n\right)^{\frac{1}{2n+1}}\right] \alpha\left(\left[\frac{T}{q+1}\right]\right)^{\frac{2n}{2n+1}}\\
			\asymp& \left(1+\log(T)+\frac{s\log(pd+1)}{T_{\textnormal{eff}}}\right)\exp\left[-C\frac{T}{\log(T)} \cdot \frac{s\log(pd+1)}{T_{\textnormal{eff}}}\right]\\
			& + T\left[1 + (s\|\bm{\Theta}^*\|_{1,\infty})^{-1}\left(\frac{T_{\textnormal{eff}}}{M^{1/\epsilon}\log(pd+1)}\right)^{\frac{\epsilon}{1+\epsilon}}\left(\mathbb{E} |\varepsilon_{t,k}|^n\right)^{\frac{1}{2n+1}}\right]\exp\left[-C\log(\zeta)\log(T)\right]\\
			\leq &\left(1+\log(T)\log(pd+1)+\frac{s\log(pd+1)}{T_{\textnormal{eff}}}\right)\exp\left[-Cs\log(T)\log(pd+1)\right]\\
			& + T\left[1 + (s\|\bm{\Theta}^*\|_{1,\infty})^{-1}\left(\frac{T_{\textnormal{eff}}}{M^{1/\epsilon}\log(pd+1)}\right)^{\frac{\epsilon}{1+\epsilon}}\cdot\left(Cs\|\bm{\Theta}^*\|_{1,\infty}^n\tau^{4(n-\epsilon-1)} M\right)^{\frac{1}{2n+1}}\right]\exp\left[-C\log(\zeta)\log(T)\right]\\
			\asymp& \log(T)\log(pd+1)\exp\left[-Cs\log(T)\log(pd+1)\right] + T \left(\frac{T_{\textnormal{eff}}}{\log(pd+1)}\right)^{\frac{n-1}{1+\epsilon}}\exp\left[-C\log(\zeta)\log(T)\right]\\
			\lesssim& \exp\left[\log(\log(T)\log(pd+1))-Cs\log(T)\log(pd+1)\right] + \exp\left(\left(\frac{n+\epsilon}{1+\epsilon} -C\log(\zeta)\right)\log(T)\right)\\
			\leq& C\exp[-Cs\log(T)\log(pd+1)].
		\end{align*}
		Thus, under slightly abuse of notation, for $1 \leq i \leq pd+1$ and $1 \leq j \leq d$, we have 
		\begin{align*}
			&\mathbb{P}\left(\max_{\substack{1 \leq i \leq pd+1\\ 1 \leq j \leq d}}\left|\nabla \mathbb{L}(\bm{\Theta}^*)_{i,j} - \mathbb{E}(\nabla \mathbb{L}(\bm{\Theta}^*))_{i,j}\right| \geq Cs\|\bm{\Theta}^*\|_{1,\infty}\left(\frac{M^{1/\epsilon}\log(pd+1)}{T_{\textnormal{eff}}}\right)^{\frac{\epsilon}{1+\epsilon}}\right)\\
			\leq& \sum_{\substack{1 \leq i \leq pd+1\\ 1 \leq j \leq d}}\mathbb{P}\left(\left|\nabla \mathbb{L}(\bm{\Theta}^*)_{i,j} - \mathbb{E}(\nabla \mathbb{L}(\bm{\Theta}^*))_{i,j}\right| \geq Cs\|\bm{\Theta}^*\|_{1,\infty}\left(\frac{M^{1/\epsilon}\log(pd+1)}{T_{\textnormal{eff}}}\right)^{\frac{\epsilon}{1+\epsilon}}\right)\\
			\leq& C(pd+1)d\exp\left[-Cs\log(T)\log(pd+1)\right] \leq C\exp\left[2\log(pd+1)-Cs\log(T)\log(pd+1)\right]\\
			\leq& C\exp\left[-Cs\log(T)\log(pd+1)\right].
		\end{align*}
		Therefore, with probability at least $1-C\exp\left[-Cs\log(T)\log(pd+1)\right]$,
		\begin{align}\label{eq:prob_bound_nabla}
			\|\nabla \mathbb{L}(\bm{\Theta}^*) - \mathbb{E}(\nabla\mathbb{L}(\bm{\Theta}^*))\|_{\infty,\infty} \lesssim s\|\bm{\Theta}^*\|_{1,\infty}\left(\frac{M^{1/\epsilon}\log(pd+1)}{T_{\textnormal{eff}}}\right)^{\frac{\epsilon}{1+\epsilon}}.
		\end{align}
		The lemma follows from \eqref{eq:deterministic_bound_nabla} and \eqref{eq:prob_bound_nabla} by triangle inequality, which concludes the proof.
	\end{proof}

	\begin{proof}[\textbf{Proof of Lemma \ref{lemma:ell1 cone}}]
		Since the objective function in \eqref{eq:loss function} is separable,  the optimization problem can be decomposed into a set of independent subproblems \eqref{eq:row-wise loss function}, each of which can be solved in parallel:
		\begin{align}\label{eq:row-wise loss function}
			\widehat{\bm{\Theta}}_{\bm{\cdot},j} = \operatorname{argmin}_{\bm{\Theta}_{\bm{\cdot},j} \in \mathbb{R}^{pd+1}} \limits \frac{1}{2T}\|\mathbf{Y}(\tau)_{\bm{\cdot},j}-\mathbf{X}(\tau)\bm{\Theta}_{\bm{\cdot},j}\|_\mathrm{F}^2+\lambda\|\bm{\Theta}_{\bm{\cdot},j}\|_1,\ \text{for} \ 1 \leq j \leq d.
		\end{align}
		By the first order optimality condition of \eqref{eq:row-wise loss function}, for each $1 \leq j \leq d$, there exists $\boldsymbol{\gamma}_j \in \partial\|\widehat{\bm{\Theta}}_{\bm{\cdot},j}\|_1$ such that $\nabla \mathbb{L}(\widehat{\bm{\Theta}}_{\bm{\cdot},j}) + \lambda\boldsymbol{\gamma}_j = \bm{0}$. Consequently, we have
		\begin{align}\label{eq:optimality condition}
			\langle\nabla \mathbb{L}(\widehat{\bm{\Theta}}_{\bm{\cdot},j}) + \lambda\boldsymbol{\gamma}_j, \widehat{\bm{\Theta}}_{\bm{\cdot},j} - \bm{\Theta}^*_{\bm{\cdot},j}\rangle = 0.
		\end{align}
		Due to the convexity of $\mathbb{L}(\bm{\Theta}_{\bm{\cdot},j})$, we have 
		\begin{align}\label{eq:convexity property}
			\langle\nabla \mathbb{L}(\widehat{\bm{\Theta}}_{\bm{\cdot},j}) -\nabla \mathbb{L}(\bm{\Theta}^*_{\bm{\cdot},j}), \widehat{\bm{\Theta}}_{\bm{\cdot},j} - \bm{\Theta}^*_{\bm{\cdot},j}\rangle \geq 0.
		\end{align}
		Combining \eqref{eq:optimality condition} and \eqref{eq:convexity property}, it follows that
		\begin{align*}
			\underbrace{\langle\nabla \mathbb{L}(\bm{\Theta}^*_{\bm{\cdot},j}), \widehat{\bm{\Theta}}_{\bm{\cdot},j} - \bm{\Theta}^*_{\bm{\cdot},j}\rangle}_{P_1} + \underbrace{\lambda\langle\boldsymbol{\gamma}_j, \widehat{\bm{\Theta}}_{\bm{\cdot},j} - \bm{\Theta}^*_{\bm{\cdot},j}\rangle}_{P_2} \leq 0.
		\end{align*}
		Next, we establish the lower bounds for $P_1$ and $P_2$, respectively. \\
		\textbf{Lower bound for $P_1$}:
			\begin{align}\label{eq:lower bound P1}
				P_1 = \langle\nabla \mathbb{L}(\bm{\Theta}^*_{\bm{\cdot},j}), \widehat{\bm{\Theta}}_{\bm{\cdot},j} - \bm{\Theta}^*_{\bm{\cdot},j}\rangle \geq -\|\nabla \mathbb{L}(\bm{\Theta}^*_{\bm{\cdot},j})\|_\infty\|\widehat{\bm{\Theta}}_{\bm{\cdot},j} - \bm{\Theta}^*_{\bm{\cdot},j}\|_1 \geq -\frac\lambda2\|\widehat{\bm{\Theta}}_{\bm{\cdot},j} - \bm{\Theta}^*_{\bm{\cdot},j}\|_1,
			\end{align}
			where the second inequality holds since $\|\nabla \mathbb{L}(\bm{\Theta}^*_{\bm{\cdot},j})\|_\infty \leq \|\nabla \mathbb{L}(\bm{\Theta}^*)\|_{\infty,\infty}\leq \lambda/2$.\\
			\textbf{Lower bound for $P_2$}: By the definition of subgradient of $\ell_1$-norm, we have $\langle\boldsymbol{\gamma}_j, \widehat{\bm{\Theta}}_{\bm{\cdot},j}\rangle = \|\widehat{\bm{\Theta}}_{\bm{\cdot},j}\|_1$ and $\|\boldsymbol{\gamma}_j\|_\infty \leq 1$. Thus, we have
			\begin{align}\label{eq:lower bound P2}
				P_2 &= \lambda\langle\boldsymbol{\gamma}_j, \widehat{\bm{\Theta}}_{\bm{\cdot},j} - \bm{\Theta}^*_{\bm{\cdot},j}\rangle = \lambda\langle(\boldsymbol{\gamma}_j)_{\mathcal{S}_j}, (\widehat{\bm{\Theta}}_{\bm{\cdot},j} - \bm{\Theta}^*_{\bm{\cdot},j})_{\mathcal{S}_j}\rangle + \lambda\langle(\boldsymbol{\gamma}_j)_{\mathcal{S}_j^c}, (\widehat{\bm{\Theta}}_{\bm{\cdot},j} - \bm{\Theta}^*_{\bm{\cdot},j})_{\mathcal{S}_j^c}\rangle \notag\\
				&\geq - \lambda\|(\boldsymbol{\gamma}_j)_{\mathcal{S}_j}\|_{\infty}\|(\widehat{\bm{\Theta}}_{\bm{\cdot},j} - \bm{\Theta}^*_{\bm{\cdot},j})_{\mathcal{S}_j}\|_1 + \lambda\langle(\boldsymbol{\gamma}_j)_{\mathcal{S}_j^c}, (\widehat{\bm{\Theta}}_{\bm{\cdot},j} - \bm{\Theta}^*_{\bm{\cdot},j})_{\mathcal{S}_j^c}\rangle\\
				&\geq - \lambda\|(\widehat{\bm{\Theta}}_{\bm{\cdot},j} - \bm{\Theta}^*_{\bm{\cdot},j})_{\mathcal{S}_j}\|_1 + \lambda\|(\widehat{\bm{\Theta}}_{\bm{\cdot},j} - \bm{\Theta}^*_{\bm{\cdot},j})_{\mathcal{S}_j^c}\|_1,\notag
			\end{align}
			where the last inequality holds because $(\bm{\Theta}_{\bm{\cdot},j}^*)_{\mathcal{S}_j^c} = \boldsymbol{0}$ and $\langle(\boldsymbol{{\gamma}}_j)_{\mathcal{S}_j^c}, (\widehat{\bm{\Theta}}_{\bm{\cdot},j})_{\mathcal{S}_j^c}\rangle = \|(\widehat{\bm{\Theta}}_{\bm{\cdot},j})_{\mathcal{S}_j^c}\|_1$.
		Combining equations \eqref{eq:lower bound P1}--\eqref{eq:lower bound P2} with \eqref{eq:convexity property}, we have
		\begin{align*}
			-\frac{\lambda}{2}\|\widehat{\bm{\Theta}}_{\bm{\cdot},j} - \bm{\Theta}^*_{\bm{\cdot},j}\|_1 - \lambda\|(\widehat{\bm{\Theta}}_{\bm{\cdot},j} - \bm{\Theta}^*_{\bm{\cdot},j})_{\mathcal{S}_j}\|_1 + \lambda\|(\widehat{\bm{\Theta}}_{\bm{\cdot},j} - \bm{\Theta}^*_{\bm{\cdot},j})_{\mathcal{S}_j^c}\|_1 \leq 0.
		\end{align*}
		Rearranging the terms, for each $1 \leq j \leq d $, we have
		\begin{equation*}
			\|(\widehat{\bm{\Theta}}_{\bm{\cdot},j} - \bm{\Theta}^*_{\bm{\cdot},j})_{\mathcal{S}_j^c}\|_1 \leq 3\|(\widehat{\bm{\Theta}}_{\bm{\cdot},j} - \bm{\Theta}^*_{\bm{\cdot},j})_{\mathcal{S}_j}\|_1.
		\end{equation*}
		This completes the proof.
	\end{proof}

	\begin{proof}[\textbf{Proof of Lemma \ref{lemma:frobenius upper bound for symmetric bergman divergence}}]
		Following the similar arguments in the proof of Lemma \ref{lemma:ell1 cone}, for each $1 \leq j \leq d$, there exists $\boldsymbol{\gamma}_j \in \partial\|\widehat{\bm{\Theta}}_{\bm{\cdot},j}\|_1$ such that $\langle\nabla \mathbb{L}(\widehat{\bm{\Theta}}_{\bm{\cdot},j}), \widehat{\bm{\Theta}}_{\bm{\cdot},j} - \bm{\Theta}^*_{\bm{\cdot},j}\rangle =  \lambda\langle\boldsymbol{\gamma}_j, \widehat{\bm{\Theta}}_{\bm{\cdot},j} - \bm{\Theta}^*_{\bm{\cdot},j}\rangle$. Then, substituting $\langle\nabla \mathbb{L}(\widehat{\bm{\Theta}}_{\bm{\cdot},j}),\widehat{\bm{\Theta}}_{\bm{\cdot},j} - \bm{\Theta}_{\bm{\cdot},j}^*\rangle$ with $\lambda\langle\boldsymbol{\gamma}_j, \bm{\Theta}_{\bm{\cdot},j}^* - \widehat{\bm{\Theta}}_{\bm{\cdot},j}\rangle$ in $\langle\nabla \mathbb{L}(\widehat{\bm{\Theta}}_{\bm{\cdot},j}) - \nabla\mathbb{L}(\bm{\Theta}_{\bm{\cdot},j}^*), \widehat{\bm{\Theta}}_{\bm{\cdot},j} - \bm{\Theta}_{\bm{\cdot},j}^*\rangle$, we have
		\begin{align*}
			\langle\nabla \mathbb{L}(\widehat{\bm{\Theta}}_{\bm{\cdot},j}) - \nabla\mathbb{L}(\bm{\Theta}_{\bm{\cdot},j}^*), \widehat{\bm{\Theta}}_{\bm{\cdot},j} - \bm{\Theta}_{\bm{\cdot},j}^*\rangle = \underbrace{\langle\nabla\mathbb{L}(\bm{\Theta}_{\bm{\cdot},j}^*),\bm{\Theta}_{\bm{\cdot},j}^*-\widehat{\bm{\Theta}}_{\bm{\cdot},j}\rangle}_{P_1} + \underbrace{\lambda\langle\boldsymbol{\gamma}_j,\bm{\Theta}_{\bm{\cdot},j}^*-\widehat{\bm{\Theta}}_{\bm{\cdot},j}\rangle}_{P_2}.
		\end{align*}
		By the cone property $\|(\widehat{\bm{\Theta}}_{\bm{\cdot},j} - \bm{\Theta}^*_{\bm{\cdot},j} )_{\mathcal{S}_j^c}\|_1 \leq 3\|(\widehat{\bm{\Theta}}_{\bm{\cdot},j} - \bm{\Theta}^*_{\bm{\cdot},j} )_{\mathcal{S}_j}\|_1$ in Lemma \ref{lemma:ell1 cone}, we establish the following upper bounds for $P_1$ and $P_2$, respectively.\\
		\textbf{Upper bound for $P_1$}:
			\begin{align}\label{eq:upper bound P1}
				P_1 &= \langle\nabla \mathbb{L}(\bm{\Theta}_{\bm{\cdot},j}^*), \bm{\Theta}_{\bm{\cdot},j}^*-\widehat{\bm{\Theta}}_{\bm{\cdot},j}\rangle \leq \|\nabla \mathbb{L}(\bm{\Theta}_{\bm{\cdot},j}^*)\|_\infty\|\bm{\Theta}_{\bm{\cdot},j}^*-\widehat{\bm{\Theta}}_{\bm{\cdot},j}\|_1 \leq \frac\lambda2\|\widehat{\bm{\Theta}}_{\bm{\cdot},j} - \bm{\Theta}_{\bm{\cdot},j}^*\|_1 \notag\\ 
				&= \frac\lambda2\left(\|(\widehat{\bm{\Theta}}_{\bm{\cdot},j} - \bm{\Theta}_{\bm{\cdot},j}^*)_{\mathcal{S}_j}\|_1+\|(\widehat{\bm{\Theta}}_{\bm{\cdot},j} - \bm{\Theta}_{\bm{\cdot},j}^*)_{\mathcal{S}_j^c}\|_1\right) \leq 2\lambda\|(\widehat{\bm{\Theta}}_{\bm{\cdot},j} - \bm{\Theta}_{\bm{\cdot},j}^*)_{\mathcal{S}_j}\|_1.
			\end{align}
			\textbf{Upper bound for $P_2$}: Note that $\|\boldsymbol{\gamma}_j\|_\infty \leq 1$, we have
			\begin{align}\label{eq:upper bound P2}
				P_2 &= \lambda\langle\boldsymbol{\gamma}_j, \bm{\Theta}_{\bm{\cdot},j}^* - \widehat{\bm{\Theta}}_{\bm{\cdot},j}\rangle \leq \lambda\|\boldsymbol{\gamma}_j\|_\infty\|\widehat{\bm{\Theta}}_{\bm{\cdot},j} - \bm{\Theta}_{\bm{\cdot},j}^*\|_1\leq \lambda\|\widehat{\bm{\Theta}}_{\bm{\cdot},j} - \bm{\Theta}_{\bm{\cdot},j}^*\|_1 \notag\\
				&= \lambda\left(\|(\widehat{\bm{\Theta}}_{\bm{\cdot},j} - \bm{\Theta}_{\bm{\cdot},j}^*)_{\mathcal{S}_j}\|_1+\|(\widehat{\bm{\Theta}}_{\bm{\cdot},j} - \bm{\Theta}_{\bm{\cdot},j}^*)_{\mathcal{S}_j^c}\|_1\right) \leq 4\lambda\|(\widehat{\bm{\Theta}}_{\bm{\cdot},j} - \bm{\Theta}_{\bm{\cdot},j}^*)_{\mathcal{S}_j}\|_1.
			\end{align}
		Combining \eqref{eq:upper bound P1} with \eqref{eq:upper bound P2}, implies that
		\begin{align*}
			\langle\nabla \mathbb{L}(\widehat{\bm{\Theta}}_{\bm{\cdot},j}) - \nabla\mathbb{L}(\bm{\Theta}_{\bm{\cdot},j}^*), \widehat{\bm{\Theta}}_{\bm{\cdot},j} - \bm{\Theta}_{\bm{\cdot},j}^*\rangle &= P_1 + P_2 \leq 6\lambda\|(\widehat{\bm{\Theta}}_{\bm{\cdot},j} - \bm{\Theta}_{\bm{\cdot},j}^*)_{\mathcal{S}_j}\|_1 \\
			&\leq 6\lambda\sqrt{s}\|(\widehat{\bm{\Theta}}_{\bm{\cdot},j} - \bm{\Theta}_{\bm{\cdot},j}^*)_{\mathcal{S}_j}\|_2 \leq 6\lambda\sqrt{s}\|\widehat{\bm{\Theta}}_{\bm{\cdot},j} - \bm{\Theta}_{\bm{\cdot},j}^*\|_2.
		\end{align*}
		This completes the proof.
	\end{proof}

	\begin{proof}[\textbf{Proof of Lemma \ref{lemma:localized restricted eigenvalues}}]
		The loss function for LSE of truncated data can be written as
		$$
		\mathbb{L}\bigl(\bm{\Theta}\bigr)=\frac{1}{2T}\sum_{t=1}^{T}\sum_{j=1}^{d}\bigl(y_{t,j}(\tau)-\mathbf{x}_t^{\mathrm{T}}(\tau)\bm{\Theta}_{\cdot,j}\bigr)^2.
		$$
		The Hessian matrix of this loss function with respect to each $\bm{\Theta}_{\bm{\cdot},j}$ is 
		\begin{align*}
			\nabla^2 \mathbb{L}(\bm{\Theta}_{\bm{\cdot},j}) = \frac{1}{T}\sum_{t=1}^{T}\boldsymbol{x}_t(\tau)\boldsymbol{x}_t^{\mathrm{T}}(\tau) = \mathbf{S}.
		\end{align*}
		For any $\mathbf{u} \in \mathcal{C}(\mathcal{S}_j, \gamma)$, given $|\mathcal{S}_j| \leq s$ for $1 \leq j \leq d$, we have 
		\begin{align*}
			&\frac{\mathbf{u}^{\mathrm{T}}\nabla^2 \mathbb{L}(\bm{\Theta}_{\bm{\cdot},j})\mathbf{u}}{\|\mathbf{u}\|_2^2} = \frac{\mathbf{u}^{\mathrm{T}} \boldsymbol{\Gamma}_x \mathbf{u}}{\|\mathbf{u}\|_2^2} + \frac{\mathbf{u}^{\mathrm{T}} \left(\mathbf{S} - \boldsymbol{\Gamma}_x\right)\mathbf{u}}{\|\mathbf{u}\|_2^2} \geq \frac{\mathbf{u}^{\mathrm{T}} \boldsymbol{\Gamma}_x \mathbf{u}}{\|\mathbf{u}\|_2^2} - \|\mathbf{S} - \boldsymbol{\Gamma}_x\|_{\infty,\infty}\frac{\|\mathbf{u}\|_1^2}{\|\mathbf{u}\|_2^2}\\
			\geq& \frac{\mathbf{u}^{\mathrm{T}} \boldsymbol{\Gamma}_x \mathbf{u}}{\|\mathbf{u}\|_2^2} - \|\mathbf{S} - \boldsymbol{\Gamma}_x\|_{\infty,\infty}\left(\frac{(1+\gamma)\|\mathbf{u}_{\mathcal{S}_j}\|_1}{\|\mathbf{u}\|_2}\right)^2 \geq \lambda_{\min}(\boldsymbol{\Gamma}_x) - (1+\gamma)^2s\|\mathbf{S} - \boldsymbol{\Gamma}_x\|_{\infty,\infty}.
		\end{align*}
		Under the conditions in Lemma \ref{lemma:covariance matrix bound}, we have $$\|\mathbf{S} - \boldsymbol{\Gamma}_x\|_{\infty,\infty} \lesssim \biggl(\frac{M^{1/\epsilon} \log(pd+1)}{T_{\textnormal{eff}}}\biggr)^{\frac{\epsilon}{1+\epsilon}},$$ with probability at least $1-C\exp\left[-C\log(T)\log(pd+1)\right]$.
		Thus taking the infimum over $\mathbf{u} \in \mathcal{C}(\mathcal{S}_j, \gamma)$, we have
		\begin{align*}
			\nu(\nabla^2 \mathbb{L}(\bm{\Theta}_{\bm{\cdot},j}),\mathcal{S}_j, \gamma) &\geq \lambda_{\min}(\boldsymbol{\Gamma}_x) - (1+\gamma)^2s\|\mathbf{S} - \boldsymbol{\Gamma}_x\|_{\infty,\infty}\\
			&\geq \lambda_{\min}(\boldsymbol{\Gamma}_x) - C(1+\gamma)^2s\biggl(\frac{M^{1/\epsilon} \log(pd+1)}{T_{\textnormal{eff}}}\biggr)^{\frac{\epsilon}{1+\epsilon}},
		\end{align*}
		with probability at least $1-C\exp\left[-C\log(T)\log(pd+1)\right]$.

		Taking $T > C\log(pd+1)$ for sufficiently large $C$, we have $\nu(\nabla^2 \mathbb{L}(\bm{\Theta}_{\bm{\cdot},j}),\mathcal{S}_j, \gamma) \geq 1/2\lambda_{\min}(\boldsymbol{\Gamma}_x)$ with probability at least $1-C\exp\left[-C\log(T)\log(pd+1)\right]$. This concludes the proof.
	\end{proof}

	\begin{proof}[\textbf{Proof of Lemma \ref{lemma:restricted eigenvalue condition}}]
		Define $\bm{\Theta}_{\bm{\cdot},j}(u) \triangleq u\bm{\Theta}_{\bm{\cdot},j} + (1-u)\bm{\Theta}_{\bm{\cdot},j}^*$ for $u \in [0,1]$ and denote $f(u) \triangleq \nabla \mathbb{L}(\bm{\Theta}_{\bm{\cdot},j}(u))$. Given that $\mathbb{L}(\cdot)$ is almost surely twice continuously differentiable, the function $f(u)$ inherits smoothness property along the path $\bm{\Theta}_{\bm{\cdot},j}(u)$. Thus, we have $\bm{\Theta}_{\bm{\cdot},j}'(u) = \bm{\Theta}_{\bm{\cdot},j} - \bm{\Theta}_{\bm{\cdot},j}^*$ and then $f'(u) = \nabla^2 \mathbb{L}(\bm{\Theta}_{\bm{\cdot},j}(u))(\bm{\Theta}_{\bm{\cdot},j} - \bm{\Theta}_{\bm{\cdot},j}^*)$, where $\nabla^2 \mathbb{L}(\bm{\Theta}_{\bm{\cdot},j}(u))$ is the Hessian matrix of $\mathbb{L}(\cdot)$ at $\bm{\Theta}_{\bm{\cdot},j}(u)$.
		Since $f(u)$ is differentiable on $[0,1]$ and its derivative $f'(u)$ is almost everywhere continuous (i.e., the set of discontinuities has Lebesgue measure zero), then by the fundamental theorem of calculus for Lebesgue integrals, we have
		\begin{align*}
			\nabla\mathbb{L}(\bm{\Theta}_{\bm{\cdot},j}) - \nabla\mathbb{L}(\bm{\Theta}_{\bm{\cdot},j}^*) = f(1) - f(0) = \int_0^1 f'(u)dt = \int_0^1 \nabla^2 \mathbb{L}(\bm{\Theta}_{\bm{\cdot},j}(u))(\bm{\Theta}_{\bm{\cdot},j} - \bm{\Theta}_{\bm{\cdot},j}^*)dt.
		\end{align*}
		Thus we have 
		\begin{align*}
			\langle\nabla\mathbb{L}(\bm{\Theta}_{\bm{\cdot},j}) - \nabla\mathbb{L}(\bm{\Theta}_{\bm{\cdot},j}^*), \bm{\Theta}_{\bm{\cdot},j} - \bm{\Theta}_{\bm{\cdot},j}^*\rangle = \int_0^1 \langle\nabla^2 \mathbb{L}(\bm{\Theta}_{\bm{\cdot},j}(u))(\bm{\Theta}_{\bm{\cdot},j} - \bm{\Theta}_{\bm{\cdot},j}^*), \bm{\Theta}_{\bm{\cdot},j} - \bm{\Theta}_{\bm{\cdot},j}^*\rangle dt.
		\end{align*}

		It is easily verified that $\nabla^2 \mathbb{L}(\bm{\Theta}_{\bm{\cdot},j}(u)) = \nabla^2 \mathbb{L}(\bm{\Theta}_{\bm{\cdot},j}) = 1/T\sum_{t=1}^T\bbm{x}_t(\tau)\bbm{x}_t^T(\tau)$, which is invariant of $u$. Consequently, for $\bm{\Theta}_{\bm{\cdot},j} - \bm{\Theta}_{\bm{\cdot},j}^*$ in $\mathcal{C}(\mathcal{S}_j,\gamma)$, taking the infimum over the cone leads to the bound:
		\begin{align*}
			&\langle\nabla\mathbb{L}(\bm{\Theta}_{\bm{\cdot},j}) - \nabla\mathbb{L}(\bm{\Theta}_{\bm{\cdot},j}^*), \bm{\Theta}_{\bm{\cdot},j} - \bm{\Theta}_{\bm{\cdot},j}^*\rangle = \int_0^1 \langle\nabla^2 \mathbb{L}(\bm{\Theta}_{\bm{\cdot},j}(u))(\bm{\Theta}_{\bm{\cdot},j} - \bm{\Theta}_{\bm{\cdot},j}^*), \bm{\Theta}_{\bm{\cdot},j} - \bm{\Theta}_{\bm{\cdot},j}^*\rangle dt\\
			=& \int_0^1 \langle\nabla^2 \mathbb{L}(\bm{\Theta}_{\bm{\cdot},j})(\bm{\Theta}_{\bm{\cdot},j} - \bm{\Theta}_{\bm{\cdot},j}^*), \bm{\Theta}_{\bm{\cdot},j} - \bm{\Theta}_{\bm{\cdot},j}^*\rangle dt = \langle\nabla^2 \mathbb{L}(\bm{\Theta}_{\bm{\cdot},j})(\bm{\Theta}_{\bm{\cdot},j} - \bm{\Theta}_{\bm{\cdot},j}^*), \bm{\Theta}_{\bm{\cdot},j} - \bm{\Theta}_{\bm{\cdot},j}^*\rangle\\
			\geq& \nu(\nabla^2 \mathbb{L}(\bm{\Theta}_{\bm{\cdot},j}),\mathcal{S}_j,\gamma)\|\bm{\Theta}_{\bm{\cdot},j} - \bm{\Theta}_{\bm{\cdot},j}^*\|_2^2.
		\end{align*}
		By Lemma \ref{lemma:localized restricted eigenvalues}, with probability at least $1-C\operatorname{exp}\bigl[-C\log(T)\log(pd+1)\bigr]$, we have
		\begin{equation*}
			\langle\nabla\mathbb{L}(\bm{\Theta}_{\bm{\cdot},j}) - \nabla\mathbb{L}(\bm{\Theta}_{\bm{\cdot},j}^*), \bm{\Theta}_{\bm{\cdot},j} - \bm{\Theta}_{\bm{\cdot},j}^*\rangle \geq \frac12\lambda_{\min}(\boldsymbol{\Gamma}_x)\|\bm{\Theta}_{\bm{\cdot},j} - \bm{\Theta}_{\bm{\cdot},j}^*\|_2^2.
		\end{equation*}
	This concludes the proof.
	\end{proof}

	\section{Technical Lemmas}\label{append:technical lemmas}
	\renewcommand{\thelemma}{E.\arabic{lemma}}
	\setcounter{lemma}{0}

	Lemmas \ref{lemma:measurable map} and \ref{lemma:lag process mixing} establish the transitivity of the $\alpha$-mixing condition; these results follow directly from \cite{wang2023rate}, and we reproduce their proofs here for completeness. Lemma \ref{lemma:Concentration inequailty for alpha mixing process}, which provides a concentration inequality essential to our error analysis, is drawn from Theorem 1.4 of \cite{bosq2012nonparametric}. For perturbation bounds, we rely on two classical tools: Lemma \ref{lemma:weyl inequality} applies Weyl's inequality \citep{weyl1912asymptotische} for eigenvalues, while Lemma \ref{lemma:davis-kahan theorem} adapts Corollary 1 of \cite{yu2015useful} via minor modifications to establish eigenvector perturbation bounds within our framework.  Proofs are omitted for conciseness.

	\begin{lemma}\label{lemma:measurable map}
		For any $\alpha$-mixing process $\{\boldsymbol{r}_t\}$ and measurable map $f(\cdot)$, the process $\{f(\boldsymbol{r}_t)\}$ is also $\alpha$-mixing with mixing coefficients bounded by those of the original process.
	\end{lemma}

	\begin{proof}[\textbf{Proof of Lemma \ref{lemma:measurable map}}]
	For any measurable map $f(\cdot)$, it is clear that $\sigma(f(\boldsymbol{r}_t)) \subseteq \sigma(\boldsymbol{r}_t)$.
	Then, 
	\begin{align*}
		&\alpha\Bigl(\sigma\bigl(\{f(\boldsymbol{r}_t): t\leq t_1\}\bigr),\sigma\bigl(\{f(\boldsymbol{r}_t): t \geq t_2\}\bigr)\Bigr) = \sup_{\substack{A_1 \in \sigma(\{f(\boldsymbol{r}_t): t\leq t_1\}),\\ A_2 \in \sigma(\{f(\boldsymbol{r}_t): t \geq t_2\})}} \bigl|\mathbb{P}(A_1 \cap A_2) - \mathbb{P}(A_1)\mathbb{P}(A_2)\bigr|\\
		\leq& \sup_{\substack{A_1 \in \sigma(\{\boldsymbol{r}_t: t\leq t_1\}),\\ A_2 \in \sigma(\{\boldsymbol{r}_t: t \geq t_2\})}} \bigl|\mathbb{P}(A_1 \cap A_2) - \mathbb{P}(A_1)\mathbb{P}(A_2)\bigr| = \alpha\Bigl(\sigma\bigl(\{\boldsymbol{r}_t: t\leq t_1\}\bigr),\sigma\bigl(\{\boldsymbol{r}_t: t \geq t_2\}\bigr)\Bigr). 
	\end{align*}
	This concludes the proof.
	\end{proof}

	\begin{lemma}\label{lemma:lag process mixing}
	Suppose $\{\boldsymbol{\psi}_t\}$ is $\alpha$-mixing. For any fixed $p$, the process $\{\boldsymbol{\phi}_t\} = (1,\boldsymbol{\psi}_{t-1}^\mathrm{T},\boldsymbol{\psi}_{t-2}^\mathrm{T},\cdots,\\\boldsymbol{\psi}_{t-p}^\mathrm{T})^\mathrm{T}$ is also $\alpha$-mixing with mixing coefficients that are bounded by those of $\{\boldsymbol{\psi}_t\}$.
	\end{lemma}

	\begin{proof}[\textbf{Proof of Lemma \ref{lemma:lag process mixing}}]
	Since the operations involved (i.e., taking a constant, concatenation) do not affect measurability, $\{\boldsymbol{\phi}_t\}$ is also measurable. In addition, $\sigma\bigl(\{\boldsymbol{\phi}_t: t \leq t_1\}\bigr) \subseteq \sigma\bigl(\{\boldsymbol{\psi}_t: t \leq t_1\}\bigr)$ and $\sigma\bigl(\{\boldsymbol{\phi}_t: t \geq t_2\}\bigr) \subseteq \sigma\bigl(\{\boldsymbol{\psi}_t: t \geq t_2-p\}\bigr)$ holds for any two time points $t_1$ and $t_2$. Following the proof of Lemma \ref{lemma:measurable map},
	\begin{align*}
		\alpha\Bigl(\sigma\bigl(\{\boldsymbol{\phi}_t: t\leq t_1\}\bigr),\sigma\bigl(\{\boldsymbol{\phi}_t: t \geq t_2\}\bigr)\Bigr) \leq \alpha\Bigl(\sigma\bigl(\{\boldsymbol{\psi}_t: t\leq t_1\}\bigr),\sigma\bigl(\{\boldsymbol{\psi}_t: t \geq t_2-p\}\bigr)\Bigr).
	\end{align*}
	Thus, when $\ell \geq p$, the dependence between the $\sigma$-fields $\sigma(\{\boldsymbol{\phi}_t: t \leq t_1\})$ and 
	$\sigma(\{\boldsymbol{\phi}_t: t \geq t_2\})$ is no stronger than that of the original process $\{\boldsymbol{\psi}_t\}$ at lag $\ell$. For $\ell < p$, there are at most $p-1$ such cases, each of which can be bounded by a finite constant. Consequently, for a fixed lag order $p$, $\{\boldsymbol{\phi}_t\}$ is also $\alpha$-mixing with $\alpha$-mixing coefficients controlled by those of $\{\boldsymbol{\psi}_t\}$. This concludes the proof.
	\end{proof}

	\begin{lemma}[Concentration Inequality for $\alpha$-mixing Process]\label{lemma:Concentration inequailty for alpha mixing process}
	Let $\{X_t\}$ be a covariance stationary zero-mean $\alpha$-mixing process with geometric mixing coefficients $\alpha(\ell) = O(\zeta^\ell)$ and $\zeta = \zeta(N) < \bar{\zeta}$. Suppose $\{X_t\}$ satisfies the Cramer's condition: $\mathbb{E}|X_t|^n \leq c^{n-2}n!\mathbb{E}X_t^2$ for some constant $c > 0$ and all $n \geq 2$. Then, for any $T \geq 2$, $q \in \left[1, T/2\right]$, $\delta>0$ and $n \geq 3$, the following inequality holds:
	\begin{align*}
		\mathbb{P}\left(\left|\frac{1}{T}\sum_{t=1}^{T} X_{t}\right| \geq \delta\right) & \leq 2[1+T / q+\mu(\delta)] \exp [-q \mu(\delta)] \notag\\
		&+11 T\left[1+5 \delta^{-1}\left(\mathbb{E} |X_{t}|^n\right)^{\frac{1}{2n+1}}\right] \alpha\left(\left[\frac{T}{q+1}\right]\right)^{\frac{2n}{2n+1}},
	\end{align*}
	where $\mu(\delta) = \delta^2/(5c\delta+25\mathbb{E}X_t^2)$. This concludes the proof.
	\end{lemma}

	\begin{lemma}[Weyl's Inequality]\label{lemma:weyl inequality}
	For any symmetric matrices $\mathbf{M}_1$ and $\mathbf{M}_2$ with the same dimensions, we have
		\begin{align*}
			\max_i|\lambda_i(\mathbf{M}_1) - \lambda_i(\mathbf{M}_2)| \leq \|\mathbf{M}_1 - \mathbf{M}_2\|_{\textnormal{op}}.
		\end{align*}
	\end{lemma}

	\begin{lemma}[A Variant of Davis-Kahan Theorem]\label{lemma:davis-kahan theorem}
		Let $\mathbf{M}_1,\mathbf{M}_2 \in \mathbb{R}^{n \times n}$ be symmetric, with eigenvalues $\lambda_1(\mathbf{M}_1) \geq \cdots \geq \lambda_n(\mathbf{M}_1) \geq 0$ and $\lambda_1(\mathbf{M}_2) \geq \cdots \geq \lambda_n(\mathbf{M}_2)$ respectively. Fix $i \in \{1,\ldots,n\}$, and assume that $\min\bigl(\lambda_{i-1}(\mathbf{M}_1) - \lambda_{i}(\mathbf{M}_1),\lambda_{i}(\mathbf{M}_1) - \lambda_{i+1}(\mathbf{M}_1)\bigr) \geq \vartheta$ for some $\vartheta > 0$, where we define $\lambda_{0}(\mathbf{M}_1) = \infty$ and $\lambda_{n+1}(\mathbf{M}_1) = 0$. If $\boldsymbol{u},\widehat{\boldsymbol{u}} \in \mathbb{R}^n$ satisfy $\mathbf{M}_1\boldsymbol{u} = \lambda_i(\mathbf{M}_1)\boldsymbol{u}$, $\mathbf{M}_2\widehat{\boldsymbol{u}} = \lambda_i(\mathbf{M}_2)\widehat{\boldsymbol{u}}$ and $\langle\boldsymbol{u},\widehat{\boldsymbol{u}}\rangle \geq 0$, then
		\begin{align*}
			\|\widehat{\boldsymbol{u}} - \boldsymbol{u} \|_2 \leq 2^{3/2}\vartheta^{-1}\|\mathbf{M}_1 - \mathbf{M}_2\|_{\textnormal{op}}.
		\end{align*}
	\end{lemma}

	\section{Auxiliary Matrices and Operators}\label{append:Specification}
	\renewcommand{\theassumption}{A.\arabic{assum}}
	\renewcommand{\thetheorem}{A.\arabic{thm}}
	\subsection{Explicit Forms of \texorpdfstring{$\mathbf{D}_N$}{D\_N} and \texorpdfstring{$\mathbf{D}_N^\dagger$}{D\_Ndagger}}\label{subappend:duplication-matrix-elimination-matrix}
	Denote  $f(l,j) = \frac{1}{2} (j-1)(2N-j+2) + (l-j+1)$ for $1 \leq j \leq l \leq N$. $\mathbf{D}_N$ satisfies $(\mathbf{D}_N)_{r, f(l,j)} = \mathds{1}\{l = j\}\,\delta_{r, l+(j-1)N} + \mathds{1}\{l > j\}\, (\delta_{r, l+(j-1)N} + \delta_{r, j+(l-1)N})$, where $\delta_{a,b}$ represents the Kronecker delta, i.e, $\delta_{a,b} = 1$ if $a=b$ and $0$ otherwise. Through straightforward calculations, we can easily verify that $\mathbf{D}_N^{\dagger}$ satisfies $(\mathbf{D}_N^{\dagger})_{f(l,j), r} = \mathds{1}\{l = j\}\,\delta_{r, l+(j-1)N} + \frac12\mathds{1}\{l > j\}\, (\delta_{r, l+(j-1)N} + \delta_{r, j+(l-1)N})$.\\
	\textbf{Examples of $\mathbf{D}_N$ and $\mathbf{D}_N^{\dagger}$ for $N=3$}: The duplication matrix $\mathbf{D}_3$ and its Moore-Penrose inverse $\mathbf{D}_3^{\dagger}$, i.e., elimination matrix are given by:
	\begin{align*}
		\mathbf{D}_3 = \begin{bmatrix}
			1 & 0 & 0 & 0 & 0 & 0 \\
			0 & 1 & 0 & 0 & 0 & 0 \\
			0 & 0 & 1 & 0 & 0 & 0 \\
			0 & 1 & 0 & 0 & 0 & 0 \\
			0 & 0 & 0 & 1 & 0 & 0 \\
			0 & 0 & 0 & 0 & 1 & 0 \\
			0 & 0 & 1 & 0 & 0 & 0 \\
			0 & 0 & 0 & 0 & 1 & 0 \\
			0 & 0 & 0 & 0 & 0 & 1
		\end{bmatrix} \in \mathbb{R}^{9 \times 6}, \quad \mathbf{D}_3^{\dagger} = \begin{bmatrix}
			1 & 0 & 0 & 0 & 0 & 0 & 0 & 0 & 0 \\
			0 & 1/2 & 0 & 1/2 & 0 & 0 & 0 & 0 & 0 \\
			0 & 0 & 1/2 & 0 & 0 & 0 & 1/2 & 0 & 0 \\
			0 & 0 & 0 & 0 & 1 & 0 & 0 & 0 & 0 \\
			0 & 0 & 0 & 0 & 0 & 1/2 & 0 & 1/2 & 0 \\
			0 & 0 & 0 & 0 & 0 & 0 & 0 & 0 & 1
		\end{bmatrix} \in \mathbb{R}^{6 \times 9}.
	\end{align*}

	\subsection{Construction of \texorpdfstring{$\mathcal{H}(\bm{\Phi},\mathbf{W})$}{H(Phi,W)}}\label{subappend:construction of H(W)}
	Throughout, $\mathcal{H}(\cdot,\mathbf{W})$ denotes the padding construction defined below. Apart from substituting $\bm{\Phi}_i^*$ with $\widehat{\bm{\Phi}}_i$, the procedure remains identical, and all results established for $\mathcal{H}(\bm{\Phi}_i^*, \mathbf{W}_i)$ also carry over to $\mathcal{H}(\widehat{\bm{\Phi}}_i, \mathbf{W}_i)$. For brevity, we simply write $\mathcal{H}(\bm{\Phi},\mathbf{W})$.	Let
	\begin{align*}
		\kappa(u,v) =
			\begin{cases}
			\dfrac{(u-1)(2N-u+2)}{2}, & u=v,\\[6pt]
			\dfrac{(u-1)(2N-u+2)}{2} + (v-u), & u<v,
			\end{cases}
	\end{align*}
	\begin{align*}
		e(u,v) = \dfrac{(u-1)(2N-u)}{2} + (v-u-1), \; u<v,
	\end{align*}
	and
	\begin{align*}
		\pi(u,v) = (u-1)N + (v-1), \; 1\leq u,v\leq N.
	\end{align*}
	For any $\mathbf{W} \in \mathbb{R}^{g \times g}$ and $\bm{\Phi}\in\mathbb{R}^{d\times d}$ with $g=N(N-1)/2$ and $d=N(N+1)/2$, the construction of $\mathcal{H}(\bm{\Phi},\mathbf{W}) \in \mathbb{R}^{N^2\times N^2}$ proceeds in two stages: column-wise padding and row-wise padding.\\
	\textbf{Column-wise padding:}
	Define the intermediate column-extended matrix $\widetilde{\mathcal H}(\bm{\Phi},\mathbf{W})\in\mathbb{R}^{d\times N^2}$ column by column as follows.
	For each $1\leq l\leq N$,
	\begin{align*}
		\widetilde{\mathcal H}(\bm{\Phi},\mathbf{W})\big[:\,,\,\pi(l,l)\big]
		\;=\;\bm{\Phi}\big[:\,,\,\kappa(l,l)\big].
	\end{align*}
	For each $1\leq j<l\leq N$, we distinguish two cases:\\
	\textbf{Case 1}: $1\leq h\leq N$,
	\begin{align*}
		\widetilde{\mathcal H}(\bm{\Phi},\mathbf{W})\big[\kappa(h,h),\,\pi(j,l)\big]
		\;=\;\frac{1}{2}\bm{\Phi}\big[\kappa(h,h),\,\kappa(j,l)\big],
	\end{align*}
	\begin{align*}
		\widetilde{\mathcal H}(\bm{\Phi},\mathbf{W})\big[\kappa(h,h),\,\pi(l,j)\big]
		\;=\;\frac{1}{2}\bm{\Phi}\big[\kappa(h,h),\,\kappa(j,l)\big].
	\end{align*}
	\textbf{Case 2}: $1\leq k<h\leq N$,
	\begin{align*}
		\widetilde{\mathcal H}(\bm{\Phi},\mathbf{W})\big[\kappa(k,h),\,\pi(l,j)\big]
		\;=\;\mathbf{W}\big[e(j,l),\,e(k,h)\big],
	\end{align*}
	\begin{align*}
		\widetilde{\mathcal H}(\bm{\Phi},\mathbf{W})\big[\kappa(k,h),\,\pi(j,l)\big]
		\;=\;\bm{\Phi}\big[\kappa(k,h),\,\kappa(j,l)\big]\;-\;\mathbf{W}\big[e(j,l),\,e(k,h)\big].
	\end{align*}
	\textbf{Row-wise padding:}
	The final matrix $\mathcal H(\bm{\Phi},\mathbf{W})\in\mathbb{R}^{N^2\times N^2}$ is obtained by a fixed reindexing of the rows of $\widetilde{\mathcal H}(\bm{\Phi},\mathbf{W})$.\\
	For $1\leq l\leq N$,
	\begin{align*}
		\mathcal H(\bm{\Phi},\mathbf{W})\big[\pi(l,l),:\big]
		\;=\;\widetilde{\mathcal H}(\bm{\Phi},\mathbf{W})\big[\kappa(l,l),:\big].
	\end{align*}
	For $1\leq j<l\leq N$, we distinguish two cases:\\
	\textbf{Case 1}: $1\leq h\leq N$,
	\begin{align*}
		\mathcal H(\bm{\Phi},\mathbf{W})\big[\pi(j,l),\,\pi(h,h)\big]
		\;=\;\widetilde{\mathcal H}(\bm{\Phi},\mathbf{W})\big[\kappa(j,l),\,\pi(h,h)\big],
	\end{align*}
	\textbf{Case 2}: $1\leq k<h\leq N$,
	\begin{align*}
		\mathcal H(\bm{\Phi},\mathbf{W})\big[\pi(l,j),\,\pi(k,h)\big]
		&=\;\widetilde{\mathcal H}(\bm{\Phi},\mathbf{W})\big[\kappa(j,l),\,\pi(k,h)\big],\\
		\mathcal H(\bm{\Phi},\mathbf{W})\big[\pi(l,j),\,\pi(h,k)\big]
		&=\;\widetilde{\mathcal H}(\bm{\Phi},\mathbf{W})\big[\kappa(j,l),\,\pi(h,k)\big],\\
		\mathcal H(\bm{\Phi},\mathbf{W})\big[\pi(j,l),\,\pi(k,h)\big]
		&=\;\widetilde{\mathcal H}(\bm{\Phi},\mathbf{W})\big[\kappa(j,l),\,\pi(h,k)\big],\\
		\mathcal H(\bm{\Phi},\mathbf{W})\big[\pi(j,l),\,\pi(h,k)\big]
		&=\;\widetilde{\mathcal H}(\bm{\Phi},\mathbf{W})\big[\kappa(j,l),\,\pi(k,h)\big].
	\end{align*}

	\subsection{Construction of \texorpdfstring{$\mathcal{R}(\cdot)$}{R(·)}}\label{subappend:construction of R(M)}
	Denote $\iota(i,j) = (i-1)N + j$ for $1 \leq i,j \leq N$. For any matrix $\mathbf{M}\in\mathbb{R}^{N^2\times N^2}$, whose entries we index as $\mathbf{M}[\iota(i,j),\,\iota(k,l)]$, we define the rearrangement operator $\mathcal{R}(\cdot):\mathbb{R}^{N^2\times N^2}\to\mathbb{R}^{N^2\times N^2}$ by
	\begin{align*}
		\mathcal{R}(\mathbf{M})[{\iota(i,k),\,\iota(j,l)}]= \mathbf{M}[\iota(i,j),\,\iota(k,l)],\; 1\leq i,j,k,l\leq N.
	\end{align*}
	\textbf{Example ($N=2$).} For illustration, let $\mathbf{X} \in \mathbb{R}^{2 \times 2}$ to correspond to the kronecker product $\mathbf{X} \otimes \mathbf{X}$ structure in the main paper. Then, we have 
	\begin{align*}
		\mathbf{X} \otimes \mathbf{X} =
		\left[\begin{array}{cc|cc}
			x_{11}x_{11} & x_{11}x_{12} & x_{12}x_{11} & x_{12}x_{12} \\
			x_{11}x_{21} & x_{11}x_{22} & x_{12}x_{21} & x_{12}x_{22} \\
			\hline
			x_{21}x_{11} & x_{21}x_{12} & x_{22}x_{11} & x_{22}x_{12} \\
			x_{21}x_{21} & x_{21}x_{22} & x_{22}x_{21} & x_{22}x_{22} \\
		\end{array} \right] 
	\end{align*}
	and 
	\begin{align*}
		\mathcal{R}(\mathbf{X} \otimes \mathbf{X}) =
		\left[ \begin{array}{c|c|c|c}
			x_{11}x_{11} & x_{21}x_{11} & x_{12}x_{11} & x_{22}x_{11} \\
			x_{11}x_{21} & x_{21}x_{21} & x_{12}x_{21} & x_{22}x_{21} \\
			x_{11}x_{12} & x_{21}x_{12} & x_{12}x_{12} & x_{22}x_{12} \\
			x_{11}x_{22} & x_{21}x_{22} & x_{12}x_{22} & x_{22}x_{22} \\
		\end{array}\right]
		= \operatorname{vec}(\mathbf{X})\operatorname{vec}^\mathrm{T}(\mathbf{X}).
	\end{align*}

	\section{Additional Simulation Results}\label{append:additional simulation results}
	\renewcommand{\thefigure}{E.\arabic{figure}}
	\renewcommand{\thetable}{E.\arabic{table}}
	\setcounter{table}{0}

	To save space, we only report the estimation and prediction errors on the configuration $\mathcal K_3^*=\{2,1,1\}$ in the main paper. In this supplementary material, we provide the corresponding results for $\mathcal{K}_3 = \{1,1,1\}$ and $\{1,2,1\}$ in Tables~\ref{tab:theta_error_combined_K111}--\ref{tab:sigma_error_combined_K121}.
	The overall patterns are consistent with those under $\mathcal{K}_3=\{2,1,1\}$ in Section 5 of the main paper, and therefore we omit a repeated discussion of these results.

	\begin{table}[H]
	\centering
	\setlength{\tabcolsep}{3pt}
	\renewcommand{\arraystretch}{1.15}
	\caption{Comparison of estimation errors for $\bbm \Theta^*$ under $\mathcal{K}_3=\{1,1,1\}$ for $(N,s)=(20,3)$ and $(N,s)=(100,10)$. Entries are means with standard deviations in parentheses. The best values are shown in bold.}
	\label{tab:theta_error_combined_K111}
	\begin{tabular}{llcccc|cccc}
	\toprule \toprule
	&
	& \multicolumn{4}{c|}{$N=20,\ s=3$}
	& \multicolumn{4}{c}{$N=100,\ s=10$} \\
	\cmidrule(lr){3-6} \cmidrule(lr){7-10}
	Distribution $\backslash$ $T$
	&
	& $600$ & $1200$ & $1800$ & $2400$
	& $900$ & $1800$ & $2700$ & $3600$ \\
	\midrule
	\multirow[c]{2.5}{*}{Gaussian}
	& $\widehat{\bbm \Theta}$
	& \makecell{\textbf{1.902}\\\textbf{(0.015)}}
	& \makecell{\textbf{1.456}\\\textbf{(0.012)}}
	& \makecell{\textbf{1.156}\\\textbf{(0.017)}}
	& \makecell{\textbf{0.998}\\\textbf{(0.014)}}
	& \makecell{\textbf{1.257}\\\textbf{(0.006)}}
	& \makecell{\textbf{1.072}\\\textbf{(0.006)}}
	& \makecell{\textbf{1.049}\\\textbf{(0.004)}}
	& \makecell{\textbf{0.963}\\\textbf{(0.003)}} \\
	&
	$\widehat{\bbm \Theta}^{\star}$
	& \makecell{4.175\\(0.226)}
	& \makecell{2.192\\(0.088)}
	& \makecell{1.571\\(0.062)}
	& \makecell{1.305\\(0.048)}
	& \makecell{2.674\\(0.142)}
	& \makecell{2.502\\(0.121)}
	& \makecell{1.906\\(0.146)}
	& \makecell{1.982\\(0.118)} \\
	\midrule
	\multirow[c]{2.5}{*}{Laplace}
	& $\widehat{\bbm \Theta}$
	& \makecell{\textbf{1.372}\\\textbf{(0.013)}}
	& \makecell{\textbf{1.106}\\\textbf{(0.010)}}
	& \makecell{\textbf{0.953}\\\textbf{(0.006)}}
	& \makecell{\textbf{0.892}\\\textbf{(0.006)}}
	& \makecell{\textbf{0.912}\\\textbf{(0.003)}}
	& \makecell{\textbf{0.948}\\\textbf{(0.007)}}
	& \makecell{\textbf{0.862}\\\textbf{(0.006)}}
	& \makecell{\textbf{0.834}\\\textbf{(0.003)}} \\
	&
	$\widehat{\bbm \Theta}^{\star}$
	& \makecell{6.178\\(0.626)}
	& \makecell{2.877\\(0.283)}
	& \makecell{2.055\\(0.247)}
	& \makecell{1.710\\(0.231)}
	& \makecell{2.739\\(0.116)}
	& \makecell{2.514\\(0.173)}
	& \makecell{2.318\\(0.140)}
	& \makecell{1.967\\(0.145)} \\
	\midrule
	\multirow[c]{2.5}{*}{$t_{4.2}$}
	& $\widehat{\bbm \Theta}$
	& \makecell{\textbf{1.744}\\\textbf{(0.021)}}
	& \makecell{\textbf{1.382}\\\textbf{(0.022)}}
	& \makecell{\textbf{1.088}\\\textbf{(0.015)}}
	& \makecell{\textbf{0.938}\\\textbf{(0.015)}}
	& \makecell{\textbf{1.042}\\\textbf{(0.018)}}
	& \makecell{\textbf{0.980}\\\textbf{(0.008)}}
	& \makecell{\textbf{0.871}\\\textbf{(0.007)}}
	& \makecell{\textbf{0.901}\\\textbf{(0.006)}} \\
	&
	$\widehat{\bbm\Theta}^{\star}$
	& \makecell{6.589\\(1.604)}
	& \makecell{3.414\\(0.746)}
	& \makecell{2.511\\(0.700)}
	& \makecell{2.074\\(0.590)}
	& \makecell{2.788\\(0.225)}
	& \makecell{2.597\\(0.223)}
	& \makecell{2.366\\(0.178)}
	& \makecell{1.991\\(0.320)} \\
	\bottomrule
	\end{tabular}
	\end{table}

	\begin{table}[H]
	\centering
	\setlength{\tabcolsep}{3pt}
	\renewcommand{\arraystretch}{1.15}
	\caption{Comparison of estimation errors for $\bbm \Theta^*$ under $\mathcal{K}_3=\{1,2,1\}$ for $(N,s)=(20,3)$ and $(N,s)=(100,10)$. Entries are means with standard deviations in parentheses. The best values are shown in bold.}
	\label{tab:theta_error_combined_K121}
	\begin{tabular}{llcccc|cccc}
	\toprule \toprule
	&
	& \multicolumn{4}{c|}{$N=20,\ s=3$}
	& \multicolumn{4}{c}{$N=100,\ s=10$} \\
	\cmidrule(lr){3-6} \cmidrule(lr){7-10}
	Distribution $\backslash$ $T$
	&
	& $600$ & $1200$ & $1800$ & $2400$
	& $900$ & $1800$ & $2700$ & $3600$ \\
	\midrule
	\multirow[c]{2.5}{*}{Gaussian}
	& $\widehat{\bbm \Theta}$
	& \makecell{\textbf{1.769}\\\textbf{(0.014)}}
	& \makecell{\textbf{1.380}\\\textbf{(0.012)}}
	& \makecell{\textbf{1.155}\\\textbf{(0.018)}}
	& \makecell{\textbf{1.008}\\\textbf{(0.014)}}
	& \makecell{\textbf{1.256}\\\textbf{(0.008)}}
	& \makecell{\textbf{1.238}\\\textbf{(0.004)}}
	& \makecell{\textbf{1.203}\\\textbf{(0.003)}}
	& \makecell{\textbf{1.229}\\\textbf{(0.003)}} \\
	&
	$\widehat{\bbm \Theta}^{\star}$
	& \makecell{4.127\\(0.220)}
	& \makecell{2.221\\(0.095)}
	& \makecell{1.622\\(0.091)}
	& \makecell{1.344\\(0.051)}
	& \makecell{2.734\\(0.062)}
	& \makecell{2.551\\(0.069)}
	& \makecell{2.315\\(0.132)}
	& \makecell{2.020\\(0.105)} \\
	\midrule
	\multirow[c]{2.5}{*}{Laplace}
	& $\widehat{\bbm \Theta}$
	& \makecell{\textbf{1.498}\\\textbf{(0.016)}}
	& \makecell{\textbf{1.118}\\\textbf{(0.011)}}
	& \makecell{\textbf{0.973}\\\textbf{(0.007)}}
	& \makecell{\textbf{0.926}\\\textbf{(0.008)}}
	& \makecell{\textbf{1.268}\\\textbf{(0.010)}}
	& \makecell{\textbf{1.199}\\\textbf{(0.009)}}
	& \makecell{\textbf{1.113}\\\textbf{(0.007)}}
	& \makecell{\textbf{1.074}\\\textbf{(0.006)}} \\
	&
	$\widehat{\bbm \Theta}^{\star}$
	& \makecell{6.298\\(0.695)}
	& \makecell{2.940\\(0.314)}
	& \makecell{2.166\\(0.318)}
	& \makecell{1.773\\(0.196)}
	& \makecell{2.753\\(0.189)}
	& \makecell{2.568\\(0.141)}
	& \makecell{2.357\\(0.095)}
	& \makecell{1.991\\(0.133)} \\
	\midrule
	\multirow[c]{2.5}{*}{$t_{4.2}$}
	& $\widehat{\bbm \Theta}$
	& \makecell{\textbf{1.703}\\\textbf{(0.021)}}
	& \makecell{\textbf{1.369}\\\textbf{(0.022)}}
	& \makecell{\textbf{1.112}\\\textbf{(0.014)}}
	& \makecell{\textbf{0.959}\\\textbf{(0.014)}}
	& \makecell{\textbf{1.249}\\\textbf{(0.012)}}
	& \makecell{\textbf{1.222}\\\textbf{(0.008)}}
	& \makecell{\textbf{1.154}\\\textbf{(0.006)}}
	& \makecell{\textbf{1.133}\\\textbf{(0.005)}} \\
	&
	$\widehat{\bbm\Theta}^{\star}$
	& \makecell{6.513\\(1.648)}
	& \makecell{3.425\\(0.838)}
	& \makecell{2.529\\(0.517)}
	& \makecell{2.164\\(0.607)}
	& \makecell{2.869\\(0.172)}
	& \makecell{2.619\\(0.194)}
	& \makecell{2.369\\(0.169)}
	& \makecell{2.026\\(0.277)} \\
	\bottomrule
	\end{tabular}
	\end{table}

	\begin{table}[H]
	\setlength{\tabcolsep}{3pt}
	\centering
	\renewcommand{\arraystretch}{1.15}
	\caption{Comparison of estimation errors for $\bm A_{11}^*$ under $\mathcal{K}_3=\{1,1,1\}$ for $(N,s)=(20,3)$ and $(N,s)=(100,10)$. Entries are mean with standard deviation in parentheses. The best values are shown in bold and the second-best are underlined.}
	\label{tab:a11_error_combined_K111}
	\begin{tabular}{llcccc|cccc}
	\toprule \toprule
	&
	& \multicolumn{4}{c|}{$N=20,\ s=3$}
	& \multicolumn{4}{c}{$N=100,\ s=10$} \\
	\cmidrule(lr){3-6} \cmidrule(lr){7-10}
	Distribution\ \textbackslash \ $T$
	&
	& $600$ & $1200$ & $1800$ & $2400$
	& $900$ & $1800$ & $2700$ & $3600$ \\
	\midrule
	\multirow[c]{6}{*}{Gaussian}
	& $\widetilde{\mathbf{A}}_{11}$
	& \makecell{\underline{1.169}\\\textbf{(0.043)}}
	& \makecell{\underline{0.843}\\\underline{(0.087)}}
	& \makecell{\underline{0.643}\\\underline{(0.049)}}
	& \makecell{\underline{0.546}\\\underline{(0.026)}}
	& \makecell{1.263\\\textbf{(0.047)}}
	& \makecell{1.187\\\textbf{(0.051)}}
	& \makecell{1.138\\\underline{(0.038)}}
	& \makecell{1.129\\(0.026)} \\
	&
	$\widehat{\mathbf{A}}_{11}$
	& \makecell{\textbf{1.165}\\\underline{(0.049)}}
	& \makecell{\textbf{0.639}\\\textbf{(0.063)}}
	& \makecell{\textbf{0.491}\\\textbf{(0.021)}}
	& \makecell{\textbf{0.436}\\\textbf{(0.014)}}
	& \makecell{\textbf{0.886}\\(0.092)}
	& \makecell{\textbf{0.719}\\(0.058)}
	& \makecell{\textbf{0.656}\\\textbf{(0.018)}}
	& \makecell{\textbf{0.617}\\\textbf{(0.019)}} \\
	&
	$\widetilde{\mathbf{A}}_{11}^{\star}$
	& \makecell{1.494\\(0.071)}
	& \makecell{1.196\\(0.108)}
	& \makecell{0.833\\(0.166)}
	& \makecell{0.645\\(0.094)}
	& \makecell{1.345\\(0.062)}
	& \makecell{1.271\\\underline{(0.056)}}
	& \makecell{1.207\\(0.042)}
	& \makecell{1.163\\\underline{(0.023)}} \\
	&
	$\widehat{\mathbf{A}}_{11}^{\star}$
	& \makecell{1.498\\(0.071)}
	& \makecell{1.194\\(0.123)}
	& \makecell{0.744\\(0.225)}
	& \makecell{0.562\\(0.126)}
	& \makecell{\underline{1.003}\\\underline{(0.048)}}
	& \makecell{\underline{0.984}\\(0.069)}
	& \makecell{\underline{0.927}\\(0.063)}
	& \makecell{\underline{0.881}\\(0.055)} \\
	\midrule
	\multirow[c]{6}{*}{Laplace}
	& $\widetilde{\mathbf{A}}_{11}$
	& \makecell{\underline{1.058}\\\textbf{(0.069)}}
	& \makecell{\underline{0.729}\\\underline{(0.075)}}
	& \makecell{\underline{0.549}\\\underline{(0.033)}}
	& \makecell{\underline{0.482}\\\underline{(0.022)}}
	& \makecell{1.159\\\textbf{(0.051)}}
	& \makecell{1.088\\\underline{(0.036)}}
	& \makecell{1.074\\\underline{(0.030)}}
	& \makecell{1.077\\\underline{(0.019)}} \\
	&
	$\widehat{\mathbf{A}}_{11}$
	& \makecell{\textbf{0.989}\\\underline{(0.116)}}
	& \makecell{\textbf{0.545}\\\textbf{(0.025)}}
	& \makecell{\textbf{0.430}\\\textbf{(0.021)}}
	& \makecell{\textbf{0.395}\\\textbf{(0.013)}}
	& \makecell{\textbf{0.802}\\(0.114)}
	& \makecell{\textbf{0.645}\\\textbf{(0.017)}}
	& \makecell{\textbf{0.597}\\\textbf{(0.017)}}
	& \makecell{\textbf{0.581}\\\textbf{(0.018)}} \\
	&
	$\widetilde{\mathbf{A}}_{11}^{\star}$
	& \makecell{1.708\\(0.130)}
	& \makecell{1.357\\(0.129)}
	& \makecell{1.132\\(0.209)}
	& \makecell{0.950\\(0.258)}
	& \makecell{1.433\\(0.102)}
	& \makecell{1.338\\(0.062)}
	& \makecell{1.237\\(0.063)}
	& \makecell{1.233\\(0.030)} \\
	&
	$\widehat{\mathbf{A}}_{11}^{\star}$
	& \makecell{1.711\\(0.131)}
	& \makecell{1.355\\(0.131)}
	& \makecell{1.101\\(0.251)}
	& \makecell{0.885\\(0.315)}
	& \makecell{\underline{1.049}\\\underline{(0.056)}}
	& \makecell{\underline{1.025}\\(0.061)}
	& \makecell{\underline{0.996}\\(0.082)}
	& \makecell{\underline{0.926}\\(0.069)} \\
	\midrule
	\multirow[c]{6}{*}{$t_{4.2}$}
	& $\widetilde{\mathbf{A}}_{11}$
	& \makecell{\underline{1.127}\\\textbf{(0.059)}}
	& \makecell{\underline{0.787}\\\underline{(0.072)}}
	& \makecell{\underline{0.596}\\\underline{(0.042)}}
	& \makecell{\underline{0.503}\\\underline{(0.022)}}
	& \makecell{1.246\\\textbf{(0.059)}}
	& \makecell{\underline{1.148}\\\underline{(0.045)}}
	& \makecell{\underline{1.100}\\\underline{(0.052)}}
	& \makecell{\underline{1.065}\\\underline{(0.041)}} \\
	&
	$\widehat{\mathbf{A}}_{11}$
	& \makecell{\textbf{1.124}\\\underline{(0.061)}}
	& \makecell{\textbf{0.624}\\\textbf{(0.062)}}
	& \makecell{\textbf{0.473}\\\textbf{(0.022)}}
	& \makecell{\textbf{0.412}\\\textbf{(0.013)}}
	& \makecell{\textbf{0.888}\\(0.135)}
	& \makecell{\textbf{0.662}\\\textbf{(0.020)}}
	& \makecell{\textbf{0.604}\\\textbf{(0.010)}}
	& \makecell{\textbf{0.574}\\\textbf{(0.008)}} \\
	&
	$\widetilde{\mathbf{A}}_{11}^{\star}$
	& \makecell{1.914\\(0.328)}
	& \makecell{1.571\\(0.279)}
	& \makecell{1.417\\(0.289)}
	& \makecell{1.264\\(0.295)}
	& \makecell{1.548\\\underline{(0.108)}}
	& \makecell{1.462\\(0.094)}
	& \makecell{1.373\\(0.079)}
	& \makecell{1.232\\(0.092)} \\
	&
	$\widehat{\mathbf{A}}_{11}^{\star}$
	& \makecell{1.915\\(0.326)}
	& \makecell{1.569\\(0.278)}
	& \makecell{1.412\\(0.294)}
	& \makecell{1.248\\(0.328)}
	& \makecell{\underline{1.169}\\(0.278)}
	& \makecell{1.150\\(0.122)}
	& \makecell{1.114\\(0.080)}
	& \makecell{1.122\\(0.095)} \\
	\bottomrule
	\end{tabular}
	\end{table}

	\begin{table}[H]
	\centering
	\setlength{\tabcolsep}{3pt}
	\renewcommand{\arraystretch}{1.15}
	\caption{Comparison of estimation errors for $\bm A_{11}^*$ under $\mathcal{K}_3=\{1,2,1\}$ for $(N,s)=(20,3)$ and $(N,s)=(100,10)$. Entries are mean with standard deviation in parentheses. The best values are shown in bold and the second-best are underlined.}
	\label{tab:a11_error_combined_K121}
	\begin{tabular}{llcccc|cccc}
	\toprule \toprule
	&
	& \multicolumn{4}{c|}{$N=20,\ s=3$}
	& \multicolumn{4}{c}{$N=100,\ s=10$} \\
	\cmidrule(lr){3-6} \cmidrule(lr){7-10}
	Distribution\ \textbackslash \ $T$
	&
	& $600$ & $1200$ & $1800$ & $2400$
	& $900$ & $1800$ & $2700$ & $3600$ \\
	\midrule
	\multirow[c]{6}{*}{Gaussian}
	& $\widetilde{\mathbf{A}}_{11}$
	& \makecell{\underline{1.141}\\\textbf{(0.046)}}
	& \makecell{\underline{0.831}\\\underline{(0.090)}}
	& \makecell{\underline{0.632}\\\underline{(0.045)}}
	& \makecell{\underline{0.543}\\\underline{(0.022)}}
	& \makecell{1.247\\\textbf{(0.041)}}
	& \makecell{1.221\\(0.058)}
	& \makecell{1.231\\(0.044)}
	& \makecell{1.142\\(0.033)} \\
	&
	$\widehat{\mathbf{A}}_{11}$
	& \makecell{\textbf{1.138}\\\underline{(0.057)}}
	& \makecell{\textbf{0.604}\\\textbf{(0.036)}}
	& \makecell{\textbf{0.490}\\\textbf{(0.022)}}
	& \makecell{\textbf{0.417}\\\textbf{(0.014)}}
	& \makecell{\textbf{0.885}\\(0.085)}
	& \makecell{\textbf{0.738}\\(0.070)}
	& \makecell{\textbf{0.659}\\\underline{(0.045)}}
	& \makecell{\textbf{0.632}\\\textbf{(0.021)}} \\
	&
	$\widetilde{\mathbf{A}}_{11}^{\star}$
	& \makecell{1.506\\(0.083)}
	& \makecell{1.199\\(0.106)}
	& \makecell{0.825\\(0.163)}
	& \makecell{0.646\\(0.100)}
	& \makecell{1.521\\(0.116)}
	& \makecell{1.352\\\textbf{(0.046)}}
	& \makecell{1.279\\\textbf{(0.031)}}
	& \makecell{1.204\\\underline{(0.025)}} \\
	&
	$\widehat{\mathbf{A}}_{11}^{\star}$
	& \makecell{1.512\\(0.081)}
	& \makecell{1.193\\(0.119)}
	& \makecell{0.732\\(0.220)}
	& \makecell{0.564\\(0.145)}
	& \makecell{\underline{0.993}\\\underline{(0.041)}}
	& \makecell{\underline{0.982}\\\underline{(0.057)}}
	& \makecell{\underline{0.896}\\(0.047)}
	& \makecell{\underline{0.861}\\(0.039)} \\
	\midrule
	\multirow[c]{6}{*}{Laplace}
	& $\widetilde{\mathbf{A}}_{11}$
	& \makecell{\underline{1.083}\\\textbf{(0.063)}}
	& \makecell{\underline{0.722}\\\underline{(0.070)}}
	& \makecell{\underline{0.556}\\\underline{(0.032)}}
	& \makecell{\underline{0.497}\\\underline{(0.022)}}
	& \makecell{1.287\\\underline{(0.051)}}
	& \makecell{1.236\\\textbf{(0.045)}}
	& \makecell{1.205\\\underline{(0.037)}}
	& \makecell{1.128\\\underline{(0.030)}} \\
	&
	$\widehat{\mathbf{A}}_{11}$
	& \makecell{\textbf{1.041}\\\underline{(0.106)}}
	& \makecell{\textbf{0.545}\\\textbf{(0.024)}}
	& \makecell{\textbf{0.433}\\\textbf{(0.019)}}
	& \makecell{\textbf{0.401}\\\textbf{(0.013)}}
	& \makecell{\textbf{0.848}\\(0.096)}
	& \makecell{\textbf{0.687}\\(0.066)}
	& \makecell{\textbf{0.620}\\\textbf{(0.021)}}
	& \makecell{\textbf{0.593}\\\textbf{(0.010)}} \\
	&
	$\widetilde{\mathbf{A}}_{11}^{\star}$
	& \makecell{1.715\\(0.133)}
	& \makecell{1.352\\(0.141)}
	& \makecell{1.192\\(0.186)}
	& \makecell{0.999\\(0.238)}
	& \makecell{1.616\\(0.072)}
	& \makecell{1.541\\(0.069)}
	& \makecell{1.438\\(0.063)}
	& \makecell{1.347\\(0.050)} \\
	&
	$\widehat{\mathbf{A}}_{11}^{\star}$
	& \makecell{1.717\\(0.135)}
	& \makecell{1.359\\(0.138)}
	& \makecell{1.170\\(0.220)}
	& \makecell{0.953\\(0.301)}
	& \makecell{\underline{1.001}\\\textbf{(0.045)}}
	& \makecell{\underline{1.014}\\\underline{(0.056)}}
	& \makecell{\underline{0.987}\\(0.073)}
	& \makecell{\underline{0.932}\\(0.083)} \\
	\midrule
	\multirow[c]{6}{*}{$t_{4.2}$}
	& $\widetilde{\mathbf{A}}_{11}$
	& \makecell{1.105\\\textbf{(0.063)}}
	& \makecell{\underline{0.765}\\\underline{(0.078)}}
	& \makecell{\underline{0.599}\\\underline{(0.039)}}
	& \makecell{\underline{0.505}\\\underline{(0.020)}}
	& \makecell{1.339\\\textbf{(0.046)}}
	& \makecell{1.234\\\textbf{(0.052)}}
	& \makecell{\underline{1.123}\\\textbf{(0.061)}}
	& \makecell{\underline{1.070}\\\underline{(0.044)}} \\
	&
	$\widehat{\mathbf{A}}_{11}$
	& \makecell{\textbf{1.095}\\\underline{(0.079)}}
	& \makecell{\textbf{0.595}\\\textbf{(0.035)}}
	& \makecell{\textbf{0.473}\\\textbf{(0.020)}}
	& \makecell{\textbf{0.413}\\\textbf{(0.012)}}
	& \makecell{\textbf{0.854}\\\underline{(0.094)}}
	& \makecell{\textbf{0.704}\\\underline{(0.062)}}
	& \makecell{\textbf{0.652}\\\underline{(0.068)}}
	& \makecell{\textbf{0.598}\\\textbf{(0.010)}} \\
	&
	$\widetilde{\mathbf{A}}_{11}^{\star}$
	& \makecell{1.897\\(0.353)}
	& \makecell{1.558\\(0.300)}
	& \makecell{1.418\\(0.240)}
	& \makecell{1.281\\(0.353)}
	& \makecell{1.667\\(0.121)}
	& \makecell{1.632\\(0.109)}
	& \makecell{1.542\\(0.113)}
	& \makecell{1.399\\(0.109)} \\
	&
	$\widehat{\mathbf{A}}_{11}^{\star}$
	& \makecell{1.897\\(0.352)}
	& \makecell{1.558\\(0.298)}
	& \makecell{1.415\\(0.244)}
	& \makecell{1.268\\(0.376)}
	& \makecell{\underline{1.139}\\(0.226)}
	& \makecell{\underline{1.123}\\(0.106)}
	& \makecell{1.125\\(0.149)}
	& \makecell{1.121\\(0.104)} \\
	\bottomrule
	\end{tabular}
	\end{table}

	\setlength{\tabcolsep}{3pt}
	\renewcommand{\arraystretch}{1.0}

	\begin{longtable}{llcccc|cccc}
	\caption{Comparison of prediction errors for $\bm \Sigma_t$ under $\mathcal{K}_3=\{1,1,1\}$ for $(N,s)=(20,3)$ and $(N,s)=(100,10)$. Entries are mean with standard deviation in parentheses. The best values are shown in bold and the second-best are underlined.}
	\label{tab:sigma_error_combined_K111}\\

	\toprule \toprule
	&
	& \multicolumn{4}{c|}{$N=20,\ s=3$}
	& \multicolumn{4}{c}{$N=100,\ s=10$} \\
	\cmidrule(lr){3-6} \cmidrule(lr){7-10}
	Distribution\ \textbackslash \ $T$
	&
	& $600$ & $1200$ & $1800$ & $2400$
	& $900$ & $1800$ & $2700$ & $3600$ \\
	\midrule
	\endfirsthead

	\toprule \toprule
	&
	& \multicolumn{4}{c|}{$N=20,\ s=3$}
	& \multicolumn{4}{c}{$N=100,\ s=10$} \\
	\cmidrule(lr){3-6} \cmidrule(lr){7-10}
	Distribution\ \textbackslash \ $T$
	&
	& $600$ & $1200$ & $1800$ & $2400$
	& $900$ & $1800$ & $2700$ & $3600$ \\
	\midrule
	\endhead

	\midrule
	\multicolumn{10}{r}{Continued on next page}
	\endfoot

	\bottomrule
	\endlastfoot
	\multirow[c]{11}{*}{Gaussian}
	& $\breve{\boldsymbol{\Sigma}}_t$
	& \makecell{1.088\\(0.012)}
	& \makecell{0.898\\\textbf{(0.005)}}
	& \makecell{0.814\\\textbf{(0.004)}}
	& \makecell{0.768\\\textbf{(0.003)}}
	& \makecell{2.713\\(0.464)}
	& \makecell{2.586\\(0.226)}
	& \makecell{2.538\\(0.208)}
	& \makecell{2.494\\(0.162)} \\
	&
	$\widetilde{\boldsymbol{\Sigma}}_t$
	& \makecell{\underline{0.983}\\\textbf{(0.008)}}
	& \makecell{\underline{0.870}\\(0.017)}
	& \makecell{0.740\\(0.016)}
	& \makecell{0.654\\(0.013)}
	& \makecell{1.381\\(0.435)}
	& \makecell{1.101\\(0.211)}
	& \makecell{1.056\\(0.168)}
	& \makecell{1.063\\(0.114)} \\
	&
	$\widehat{\boldsymbol{\Sigma}}_t$
	& \makecell{\textbf{0.983}\\\underline{(0.009)}}
	& \makecell{\textbf{0.772}\\\underline{(0.018)}}
	& \makecell{\textbf{0.637}\\\underline{(0.014)}}
	& \makecell{\textbf{0.562}\\\underline{(0.010)}}
	& \makecell{\textbf{0.854}\\(0.048)}
	& \makecell{\textbf{0.789}\\\underline{(0.033)}}
	& \makecell{\textbf{0.750}\\\underline{(0.022)}}
	& \makecell{\textbf{0.719}\\\underline{(0.017)}} \\
	&
	$\breve{\boldsymbol{\Sigma}}_t^{\star}$
	& \makecell{2.529\\(0.031)}
	& \makecell{1.786\\(0.025)}
	& \makecell{1.452\\(0.026)}
	& \makecell{1.272\\(0.025)}
	& \makecell{3.229\\\textbf{(0.025)}}
	& \makecell{3.134\\\textbf{(0.013)}}
	& \makecell{3.040\\\textbf{(0.013)}}
	& \makecell{2.979\\\textbf{(0.010)}} \\
	&
	$\widetilde{\boldsymbol{\Sigma}}_t^{\star}$
	& \makecell{1.626\\(0.298)}
	& \makecell{1.086\\(0.152)}
	& \makecell{0.775\\(0.147)}
	& \makecell{0.632\\(0.107)}
	& \makecell{1.979\\(0.168)}
	& \makecell{1.786\\(0.100)}
	& \makecell{1.551\\(0.079)}
	& \makecell{1.305\\(0.031)} \\
	&
	$\widehat{\boldsymbol{\Sigma}}_t^{\star}$
	& \makecell{1.635\\(0.293)}
	& \makecell{1.086\\(0.152)}
	& \makecell{\underline{0.717}\\(0.176)}
	& \makecell{\underline{0.566}\\(0.121)}
	& \makecell{\underline{1.048}\\\underline{(0.041)}}
	& \makecell{\underline{1.029}\\(0.056)}
	& \makecell{\underline{0.971}\\(0.029)}
	& \makecell{\underline{0.942}\\(0.020)} \\
	\midrule
	\multirow[c]{11}{*}{Laplace}
	& $\breve{\boldsymbol{\Sigma}}_t$
	& \makecell{\textbf{0.877}\\\textbf{(0.003)}}
	& \makecell{0.840\\\textbf{(0.002)}}
	& \makecell{0.833\\\textbf{(0.002)}}
	& \makecell{0.830\\\textbf{(0.002)}}
	& \makecell{2.676\\(0.331)}
	& \makecell{2.548\\(0.187)}
	& \makecell{2.474\\(0.217)}
	& \makecell{2.460\\(0.191)} \\
	&
	$\widetilde{\boldsymbol{\Sigma}}_t$
	& \makecell{0.911\\\underline{(0.010)}}
	& \makecell{\underline{0.800}\\(0.020)}
	& \makecell{\underline{0.694}\\(0.016)}
	& \makecell{\underline{0.653}\\(0.017)}
	& \makecell{\underline{0.992}\\(0.218)}
	& \makecell{1.053\\(0.174)}
	& \makecell{\underline{0.987}\\(0.217)}
	& \makecell{1.055\\(0.143)} \\
	&
	$\widehat{\boldsymbol{\Sigma}}_t$
	& \makecell{\underline{0.891}\\(0.020)}
	& \makecell{\textbf{0.697}\\\underline{(0.013)}}
	& \makecell{\textbf{0.594}\\\underline{(0.014)}}
	& \makecell{\textbf{0.563}\\\underline{(0.011)}}
	& \makecell{\textbf{0.818}\\\underline{(0.045)}}
	& \makecell{\textbf{0.731}\\\underline{(0.025)}}
	& \makecell{\textbf{0.687}\\\underline{(0.027)}}
	& \makecell{\textbf{0.672}\\\underline{(0.022)}} \\
	&
	$\breve{\boldsymbol{\Sigma}}_t^{\star}$
	& \makecell{2.691\\(0.063)}
	& \makecell{1.965\\(0.068)}
	& \makecell{1.634\\(0.086)}
	& \makecell{1.451\\(0.098)}
	& \makecell{3.241\\\textbf{(0.025)}}
	& \makecell{3.144\\\textbf{(0.016)}}
	& \makecell{3.073\\\textbf{(0.010)}}
	& \makecell{2.997\\\textbf{(0.009)}} \\
	&
	$\widetilde{\boldsymbol{\Sigma}}_t^{\star}$
	& \makecell{2.377\\(0.568)}
	& \makecell{1.413\\(0.392)}
	& \makecell{1.128\\(0.371)}
	& \makecell{0.986\\(0.389)}
	& \makecell{1.998\\(0.293)}
	& \makecell{1.920\\(0.213)}
	& \makecell{1.667\\(0.070)}
	& \makecell{1.417\\(0.043)} \\
	&
	$\widehat{\boldsymbol{\Sigma}}_t^{\star}$
	& \makecell{2.383\\(0.564)}
	& \makecell{1.415\\(0.389)}
	& \makecell{1.115\\(0.376)}
	& \makecell{0.950\\(0.404)}
	& \makecell{1.092\\(0.050)}
	& \makecell{\underline{1.045}\\(0.028)}
	& \makecell{1.028\\(0.046)}
	& \makecell{\underline{0.989}\\(0.049)} \\
	\midrule
	\multirow[c]{11}{*}{$t_{4.2}$}
	& $\breve{\boldsymbol{\Sigma}}_t$
	& \makecell{1.031\\\textbf{(0.009)}}
	& \makecell{0.913\\\textbf{(0.005)}}
	& \makecell{0.807\\\textbf{(0.004)}}
	& \makecell{0.753\\\textbf{(0.003)}}
	& \makecell{2.602\\(1.180)}
	& \makecell{2.434\\(0.247)}
	& \makecell{2.339\\(0.243)}
	& \makecell{2.317\\(0.177)} \\
	&
	$\widetilde{\boldsymbol{\Sigma}}_t$
	& \makecell{\underline{0.966}\\\underline{(0.011)}}
	& \makecell{\underline{0.805}\\\underline{(0.020)}}
	& \makecell{\underline{0.685}\\(0.017)}
	& \makecell{\underline{0.605}\\(0.013)}
	& \makecell{\underline{0.886}\\(0.154)}
	& \makecell{\underline{1.027}\\(0.197)}
	& \makecell{\underline{0.959}\\(0.454)}
	& \makecell{\underline{0.967}\\(0.205)} \\
	&
	$\widehat{\boldsymbol{\Sigma}}_t$
	& \makecell{\textbf{0.964}\\(0.012)}
	& \makecell{\textbf{0.717}\\(0.020)}
	& \makecell{\textbf{0.594}\\\underline{(0.014)}}
	& \makecell{\textbf{0.518}\\\underline{(0.010)}}
	& \makecell{\textbf{0.836}\\\underline{(0.094)}}
	& \makecell{\textbf{0.720}\\\underline{(0.036)}}
	& \makecell{\textbf{0.672}\\\underline{(0.048)}}
	& \makecell{\textbf{0.650}\\\textbf{(0.021)}} \\
	&
	$\breve{\boldsymbol{\Sigma}}_t^{\star}$
	& \makecell{2.601\\(0.128)}
	& \makecell{2.049\\(0.172)}
	& \makecell{1.761\\(0.229)}
	& \makecell{1.572\\(0.228)}
	& \makecell{3.248\\\textbf{(0.072)}}
	& \makecell{3.173\\\textbf{(0.036)}}
	& \makecell{2.982\\(0.326)}
	& \makecell{2.955\\\underline{(0.024)}} \\
	&
	$\widetilde{\boldsymbol{\Sigma}}_t^{\star}$
	& \makecell{3.278\\(1.571)}
	& \makecell{1.983\\(0.895)}
	& \makecell{1.632\\(0.796)}
	& \makecell{1.391\\(0.706)}
	& \makecell{2.232\\(1.182)}
	& \makecell{2.052\\(0.305)}
	& \makecell{1.716\\(0.243)}
	& \makecell{1.528\\(0.234)} \\
	&
	$\widehat{\boldsymbol{\Sigma}}_t^{\star}$
	& \makecell{3.275\\(1.564)}
	& \makecell{1.980\\(0.892)}
	& \makecell{1.625\\(0.794)}
	& \makecell{1.377\\(0.707)}
	& \makecell{1.554\\(0.101)}
	& \makecell{1.261\\(0.078)}
	& \makecell{1.103\\\textbf{(0.033)}}
	& \makecell{1.045\\(0.025)} \\
	\bottomrule
	\end{longtable}

	\setlength{\tabcolsep}{3pt}
	\renewcommand{\arraystretch}{1.0}

	\begin{longtable}{llcccc|cccc}
	\caption{Comparison of prediction errors for $\bm \Sigma_t$ under $\mathcal{K}_3=\{1,2,1\}$ for $(N,s)=(20,3)$ and $(N,s)=(100,10)$. Entries are mean with standard deviation in parentheses. The best values are shown in bold and the second-best are underlined.}
	\label{tab:sigma_error_combined_K121}\\

	\toprule \toprule
	&
	& \multicolumn{4}{c|}{$N=20,\ s=3$}
	& \multicolumn{4}{c}{$N=100,\ s=10$} \\
	\cmidrule(lr){3-6} \cmidrule(lr){7-10}
	Distribution\ \textbackslash \ $T$
	&
	& $600$ & $1200$ & $1800$ & $2400$
	& $900$ & $1800$ & $2700$ & $3600$ \\
	\midrule
	\endfirsthead

	\toprule \toprule
	&
	& \multicolumn{4}{c|}{$N=20,\ s=3$}
	& \multicolumn{4}{c}{$N=100,\ s=10$} \\
	\cmidrule(lr){3-6} \cmidrule(lr){7-10}
	Distribution\ \textbackslash \ $T$
	&
	& $600$ & $1200$ & $1800$ & $2400$
	& $900$ & $1800$ & $2700$ & $3600$ \\
	\midrule
	\endhead

	\midrule
	\multicolumn{10}{r}{Continued on next page}
	\endfoot

	\bottomrule
	\endlastfoot
	\multirow[c]{4}{*}{Gaussian}
	& $\breve{\boldsymbol{\Sigma}}_t$
	& \makecell{1.039\\\underline{(0.009)}}
	& \makecell{0.867\\\textbf{(0.004)}}
	& \makecell{0.809\\\textbf{(0.003)}}
	& \makecell{0.774\\\textbf{(0.002)}}
	& \makecell{2.826\\(0.324)}
	& \makecell{2.737\\(0.189)}
	& \makecell{2.659\\(0.148)}
	& \makecell{2.634\\(0.186)} \\
	&
	$\widetilde{\boldsymbol{\Sigma}}_t$
	& \makecell{\underline{0.974}\\\textbf{(0.008)}}
	& \makecell{\underline{0.887}\\(0.018)}
	& \makecell{\underline{0.775}\\\underline{(0.016)}}
	& \makecell{0.688\\\underline{(0.015)}}
	& \makecell{1.266\\(0.232)}
	& \makecell{1.156\\(0.236)}
	& \makecell{1.065\\(0.086)}
	& \makecell{1.054\\(0.067)} \\
	&
	$\widehat{\boldsymbol{\Sigma}}_t$
	& \makecell{\textbf{0.973}\\(0.010)}
	& \makecell{\textbf{0.802}\\\underline{(0.016)}}
	& \makecell{\textbf{0.678}\\(0.019)}
	& \makecell{\textbf{0.574}\\(0.016)}
	& \makecell{\textbf{0.836}\\\underline{(0.043)}}
	& \makecell{\textbf{0.786}\\\underline{(0.036)}}
	& \makecell{\textbf{0.742}\\\underline{(0.017)}}
	& \makecell{\textbf{0.722}\\\underline{(0.016)}} \\
	&
	$\breve{\boldsymbol{\Sigma}}_t^{\star}$
	& \makecell{2.591\\(0.034)}
	& \makecell{1.836\\(0.026)}
	& \makecell{1.503\\(0.028)}
	& \makecell{1.318\\(0.026)}
	& \makecell{3.417\\\textbf{(0.035)}}
	& \makecell{3.312\\\textbf{(0.022)}}
	& \makecell{3.260\\\textbf{(0.013)}}
	& \makecell{3.196\\\textbf{(0.007)}} \\
	&
	$\widetilde{\boldsymbol{\Sigma}}_t^{\star}$
	& \makecell{1.936\\(0.369)}
	& \makecell{1.204\\(0.157)}
	& \makecell{0.914\\(0.205)}
	& \makecell{0.737\\(0.115)}
	& \makecell{2.186\\(0.271)}
	& \makecell{1.989\\(0.094)}
	& \makecell{1.662\\(0.057)}
	& \makecell{1.389\\(0.078)} \\
	&
	$\widehat{\boldsymbol{\Sigma}}_t^{\star}$
	& \makecell{1.942\\(0.365)}
	& \makecell{1.205\\(0.156)}
	& \makecell{0.869\\(0.220)}
	& \makecell{\underline{0.683}\\(0.128)}
	& \makecell{\underline{1.038}\\(0.046)}
	& \makecell{\underline{1.014}\\(0.055)}
	& \makecell{\underline{0.949}\\(0.025)}
	& \makecell{\underline{0.946}\\(0.047)} \\
	\midrule
	\multirow[c]{11}{*}{Laplace}
	& $\breve{\boldsymbol{\Sigma}}_t$
	& \makecell{\textbf{0.893}\\\textbf{(0.003)}}
	& \makecell{0.839\\\textbf{(0.002)}}
	& \makecell{0.829\\\textbf{(0.002)}}
	& \makecell{0.826\\\textbf{(0.002)}}
	& \makecell{2.729\\(0.279)}
	& \makecell{2.729\\(0.250)}
	& \makecell{2.606\\(0.250)}
	& \makecell{2.676\\(0.860)} \\
	&
	$\widetilde{\boldsymbol{\Sigma}}_t$
	& \makecell{0.925\\\underline{(0.010)}}
	& \makecell{\underline{0.814}\\(0.017)}
	& \makecell{\underline{0.718}\\(0.017)}
	& \makecell{\underline{0.685}\\(0.015)}
	& \makecell{1.363\\(0.307)}
	& \makecell{1.151\\(0.115)}
	& \makecell{1.057\\(0.123)}
	& \makecell{1.049\\(0.283)} \\
	&
	$\widehat{\boldsymbol{\Sigma}}_t$
	& \makecell{\underline{0.915}\\(0.017)}
	& \makecell{\textbf{0.733}\\\underline{(0.015)}}
	& \makecell{\textbf{0.614}\\\underline{(0.017)}}
	& \makecell{\textbf{0.579}\\\underline{(0.012)}}
	& \makecell{\textbf{0.815}\\\underline{(0.045)}}
	& \makecell{\textbf{0.747}\\\underline{(0.028)}}
	& \makecell{\textbf{0.707}\\\underline{(0.019)}}
	& \makecell{\textbf{0.685}\\\underline{(0.036)}} \\
	&
	$\breve{\boldsymbol{\Sigma}}_t^{\star}$
	& \makecell{2.779\\(0.059)}
	& \makecell{2.021\\(0.070)}
	& \makecell{1.707\\(0.106)}
	& \makecell{1.507\\(0.084)}
	& \makecell{3.446\\\textbf{(0.022)}}
	& \makecell{3.354\\\textbf{(0.016)}}
	& \makecell{3.259\\\textbf{(0.017)}}
	& \makecell{3.196\\\textbf{(0.032)}} \\
	&
	$\widetilde{\boldsymbol{\Sigma}}_t^{\star}$
	& \makecell{2.889\\(0.815)}
	& \makecell{1.618\\(0.522)}
	& \makecell{1.383\\(0.471)}
	& \makecell{1.132\\(0.334)}
	& \makecell{2.194\\(0.259)}
	& \makecell{2.046\\(0.247)}
	& \makecell{1.690\\(0.125)}
	& \makecell{1.430\\(0.233)} \\
	&
	$\widehat{\boldsymbol{\Sigma}}_t^{\star}$
	& \makecell{2.892\\(0.809)}
	& \makecell{1.622\\(0.520)}
	& \makecell{1.375\\(0.475)}
	& \makecell{1.111\\(0.353)}
	& \makecell{\underline{1.091}\\(0.068)}
	& \makecell{\underline{1.038}\\(0.031)}
	& \makecell{\underline{1.035}\\(0.066)}
	& \makecell{\underline{0.992}\\(0.093)} \\
	\midrule
	\multirow[c]{3}{*}{$t_{4.2}$}
	& $\breve{\boldsymbol{\Sigma}}_t$
	& \makecell{1.046\\\textbf{(0.010)}}
	& \makecell{0.917\\\textbf{(0.005)}}
	& \makecell{0.822\\\textbf{(0.004)}}
	& \makecell{0.766\\\textbf{(0.003)}}
	& \makecell{2.759\\(1.236)}
	& \makecell{2.655\\(0.302)}
	& \makecell{2.559\\(0.232)}
	& \makecell{2.561\\(0.211)} \\
	&
	$\widetilde{\boldsymbol{\Sigma}}_t$
	& \makecell{\underline{0.969}\\\underline{(0.011)}}
	& \makecell{\underline{0.830}\\(0.020)}
	& \makecell{\underline{0.716}\\\underline{(0.017)}}
	& \makecell{\underline{0.624}\\(0.013)}
	& \makecell{1.315\\(0.325)}
	& \makecell{1.166\\(0.147)}
	& \makecell{\underline{1.099}\\(0.332)}
	& \makecell{\underline{1.035}\\(0.086)} \\
	&
	$\widehat{\boldsymbol{\Sigma}}_t$
	& \makecell{\textbf{0.964}\\(0.013)}
	& \makecell{\textbf{0.758}\\\underline{(0.020)}}
	& \makecell{\textbf{0.622}\\(0.022)}
	& \makecell{\textbf{0.533}\\\underline{(0.013)}}
	& \makecell{\textbf{0.811}\\(0.067)}
	& \makecell{\textbf{0.752}\\\underline{(0.043)}}
	& \makecell{\textbf{0.712}\\\underline{(0.033)}}
	& \makecell{\textbf{0.685}\\\textbf{(0.023)}} \\
	&
	$\breve{\boldsymbol{\Sigma}}_t^{\star}$
	& \makecell{2.637\\(0.124)}
	& \makecell{2.088\\(0.204)}
	& \makecell{1.802\\(0.175)}
	& \makecell{1.628\\(0.215)}
	& \makecell{3.474\\\textbf{(0.037)}}
	& \makecell{3.390\\\textbf{(0.024)}}
	& \makecell{3.290\\(0.055)}
	& \makecell{3.195\\(0.148)} \\
	&
	$\widetilde{\boldsymbol{\Sigma}}_t^{\star}$
	& \makecell{3.708\\(1.780)}
	& \makecell{2.240\\(1.136)}
	& \makecell{1.773\\(0.713)}
	& \makecell{1.642\\(0.869)}
	& \makecell{2.437\\(1.089)}
	& \makecell{2.098\\(0.354)}
	& \makecell{1.807\\(0.288)}
	& \makecell{1.544\\(0.301)} \\
	&
	$\widehat{\boldsymbol{\Sigma}}_t^{\star}$
	& \makecell{3.711\\(1.774)}
	& \makecell{2.236\\(1.131)}
	& \makecell{1.768\\(0.709)}
	& \makecell{1.633\\(0.869)}
	& \makecell{\underline{1.237}\\\underline{(0.045)}}
	& \makecell{\underline{1.161}\\(0.068)}
	& \makecell{1.177\\\textbf{(0.021)}}
	& \makecell{1.081\\\underline{(0.026)}} \\
	\bottomrule
	\end{longtable}

	By providing the positive definite rate (PD-rate) of $\operatorname{vech}^{-1}(\widehat{\bm{\Theta}}^{\mathrm{T}}\bbm{x}_t)$, we show that the projection $\mathcal{P}(\cdot)$ is not frequently needed in practice. The PD-rate is defined as the proportion of time points $t$ such that $\operatorname{vech}^{-1}(\widehat{\bm{\Theta}}^{\mathrm{T}}\bbm{x}_t)$ is positive definite among $t=p+1,\ldots,T$ and taken average over 500 replications, i.e., 
	\begin{align*}
		\text{PD-proportion} = \frac{1}{500}\sum_{i=1}^{500}\frac{1}{T-p}\sum_{t=p+1}^T\mathds{1}\{\operatorname{vech}^{-1}(\widehat{\bm{\Theta}}_i^{(t)\mathrm{T}}\bbm{x}_t) \succ \bm{0}\} \times 100,
	\end{align*}
	where $\mathds{1}\{\cdot\}$ is the indicator function and $\widehat{\bm{\Theta}}_i^{(t)}$ is the regularized LSE estimated from the training data up to time $t$ in the $i$-th replication. Tables \ref{tab:PD-proportion} reports the PD-proportion for different $(N,s)$ configurations. We observe that the PD-proportion is very close to $100$ as $T$ increases, indicating that the projection $\mathcal{P}(\cdot)$ is rarely needed in practice.

	\begin{table}[H]
	\centering
	\caption{PD-proportion under different $(N,s)$ configurations.}
	\label{tab:PD-proportion}
	\renewcommand{\arraystretch}{1.0}
	\setlength{\tabcolsep}{6pt}
	\makebox[\textwidth][c]{
		\begin{tabular}{c ccc ccc ccc}
			\toprule \toprule
			$N=20$\multirow{2}{*} & \multicolumn{3}{c}{$\mathcal{K}_3 = \{1, 1, 1\}$} & \multicolumn{3}{c}{$\mathcal{K}_3 = \{2,1,1\}$} & \multicolumn{3}{c}{$\mathcal{K}_3 = \{1,2,1\}$}\\
			\cmidrule(lr{4pt}){2-4} \cmidrule(lr{4pt}){5-7} \cmidrule(lr{4pt}){8-10}
			$T\backslash \bm{\eta}_t$ & Normal & Laplace & $t$ & Normal & Laplace & $t$ & Normal & Laplace & $t$\\
			\midrule
			100 & 91.92 & 20.02 & 18.69 & 92.25 & 28.16 & 24.24 & 93.75 & 23.73 & 30.84\\
			200 & 99.98 & 90.57 & 77.24 & 99.99 & 95.84 & 86.38 & 100 & 93.59 & 92.14\\
			300 & 100 & 99.16 & 94.69 & 100 & 99.86 & 97.84 & 100 & 99.62 & 99.20\\
			400 & 100 & 99.95 & 98.73 & 100 & 100 & 99.65 & 100 & 99.99 & 99.92\\
			\midrule
			$N=100$\multirow{2}{*} & \multicolumn{3}{c}{$\mathcal{K}_3 = \{1, 1, 1\}$} & \multicolumn{3}{c}{$\mathcal{K}_3 = \{2,1,1\}$} & \multicolumn{3}{c}{$\mathcal{K}_3 = \{1,2,1\}$}\\
			\cmidrule(lr{4pt}){2-4} \cmidrule(lr{4pt}){5-7} \cmidrule(lr{4pt}){8-10}
			$T\backslash \bm{\eta}_t$ & Normal & Laplace & $t$ & Normal & Laplace & $t$ & Normal & Laplace & $t$\\
			\midrule
			100 & 79.73 & 79.09 & 32.73 & 80.32 & 88.10 & 77.56 & 79.73 & 84.61 & 24.59\\
			200 & 100 & 99.99 & 99.57 & 100 & 100 & 100 & 100 & 100 & 99.15\\
			300 & 100 & 100 & 99.99 & 100 & 100 & 100 & 100 & 100 & 99.99\\
			400 & 100 & 100 & 100 & 100 & 100 & 100 & 100 & 100 & 100\\
			\bottomrule
		\end{tabular}
	}
	\end{table}

\end{document}